\documentclass[onefignum,onetabnum]{siamart190516}

\usepackage{lipsum}
\usepackage{amsfonts}
\usepackage{amsmath}
\usepackage{mathtools}
\usepackage{graphicx}
\usepackage{booktabs}
\usepackage{epstopdf}
\usepackage{algorithmic}
\usepackage{todonotes}
\usepackage[caption=false]{subfig}
\usepackage{tikz}
\DeclareFontFamily{OT1}{pzc}{}
\DeclareFontShape{OT1}{pzc}{m}{it}{<-> s * [1.10] pzcmi7t}{}
\DeclareMathAlphabet{\mathpzc}{OT1}{pzc}{m}{it}
\usetikzlibrary{spy,shapes,shadows}
\usepackage{pgfplots}  
\pgfplotsset{compat=1.15}
\ifpdf
  \DeclareGraphicsExtensions{.eps,.pdf,.png,.jpg}
\else
  \DeclareGraphicsExtensions{.eps}
\fi

\newsiamremark{remark}{Remark}
\newsiamremark{example}{Example}
\newsiamremark{hypothesis}{Hypothesis}
\crefname{hypothesis}{Hypothesis}{Hypotheses}
\newsiamthm{claim}{Claim}

\headers{Learned infinite elements}{T. Hohage, C. Lehrenfeld and J. Preuss}

\title{Learned infinite elements\thanks{Submitted to the editors DATE.
}}

\author{Thorsten Hohage\thanks{Institut f\"ur Numerische und Angewandte Mathematik, Universit\"at G\"ottingen and 
   Max-Planck-Institut f\"ur Sonnensystemforschung, G\"ottingen, 
  (\email{hohage@math.uni-goettingen.de}).}
\and Christoph Lehrenfeld\thanks{Institut f\"ur Numerische und Angewandte Mathematik, Universit\"at G\"ottingen,
  (\email{{lehrenfeld@math.uni-goettingen.de}).}    }
  \and Janosch Preuss\thanks{Institut f\"ur Numerische und Angewandte Mathematik, Universit\"at G\"ottingen and 
  Max-Planck-Institut f\"ur Sonnensystemforschung, G\"ottingen, 
  (\email{j.preuss@math.uni-goettingen.de}).}
  }

\usepackage{amsopn}
\newcommand{\DtN}{\mathpzc{DtN}}
\DeclareMathOperator{\dtn}{\mathpzc{dtn}}

\DeclareMathOperator*{\argmin}{argmin}
\DeclareMathOperator{\Id}{Id}
\DeclareMathOperator{\Div}{div}

\newcommand{\dDtN}{\operatorname{DtN}}
\newcommand{\ddtn}{\operatorname{dtn}}

\newcommand{\calA}{\mathpzc{A}}
\newcommand{\calB}{\mathpzc{B}}

\newcommand{\calL}{\mathpzc{L}}

\newcommand{\mass}{M}
\newcommand{\massr}{\mathcal{B}}
\newcommand{\dmassr}{B}   %formerly L^{(1)}

\newcommand{\stiff}{K}
\newcommand{\stiffr}{\mathcal{A}}
\newcommand{\dstiffr}{A}

\newcommand{\massFEM}{M} % mass matrix FEM  
\newcommand{\stiffFEM}{K} % stiffness matrix FEM (formerly A)
\newcommand{\FEshape}{\phi} % FE basis function symbol

\newcommand{\fh}{\underline{f}}
\newcommand{\uh}{\underline{u}}
\newcommand{\vh}{\underline{v}}
\newcommand{\wh}{\underline{w}}

\newcommand{\dof}{DOF}
\newcommand{\dofs}{DOFs}

\newcommand{\RJump}{R_{\mathrm{J}}} % position of jump in wavenumber
\newcommand{\RScatter}{R_{\mathrm{s}}} % radius of scatterer 
\newcommand{\RVALC}{R_{\mathrm{V}}} % radius of scatterer 

\newcommand{\wavenr}{k} % wavenumber 

\newcommand{\colora}[1]{\textcolor{red}{#1}}
\newcommand{\colorb}[1]{\textcolor{blue}{#1}}

\ifpdf
\hypersetup{
  pdftitle={Learned infinite elements},
  pdfauthor={T. Hohage, C. Lehrenfeld and J. Preuss} 
}
\fi

\begin{document}

\maketitle

% REQUIRED
\begin{abstract}
We study the numerical solution of scalar time-harmonic wave equations on unbounded domains which can be split into a bounded 
interior domain of primary interest and an exterior domain with separable geometry. To compute the solution in the interior domain, 
approximations to the Dirichlet-to-Neumann (DtN) map of the exterior domain have to be imposed as transparent boundary conditions 
on the artificial coupling boundary. Although the DtN map can be computed by separation of variables, it is a nonlocal operator with 
dense matrix representations, and hence computationally inefficient. Therefore, approximations of DtN maps 
by sparse matrices, usually involving 
additional degrees of freedom, have been studied intensively in the literature using a variety of approaches including different types of infinite elements, local non-reflecting boundary conditions, and perfectly matched layers. The entries of these sparse matrices are derived analytically, e.g.\ from transformations or asymptotic expansions of solutions to the differential equation in the exterior domain. 
In contrast, in this paper we propose to `learn' the matrix entries from the DtN map in its separated form by solving an optimization 
problem as a preprocessing step.  
Theoretical considerations suggest that the approximation quality of learned infinite elements improves exponentially with increasing 
number of infinite element degrees of freedom, which is confirmed in numerical experiments. 
These numerical studies also show that learned infinite elements outperform 
state-of-the-art methods for the Helmholtz equation. At the same time, learned infinite elements are much more flexible than traditional methods as they, e.g.,  work similarly well for exterior domains involving strong reflections.
As main motivating example we study the atmosphere of the Sun, which is strongly inhomogeneous and exhibits reflections at the corona.
%This paper proposes transparent boundary conditions for time-harmonic wave equations posed in stratified media.
%Assuming a separable geometry the Dirichlet-to-Neumann operator (DtN), which describes the exact transparent boundary condition, is separable as well.
%Utilizing this observation straightforwardly, like in the DtN FE method, may lead to very accurate transparent boundary conditions, yet dense linear systems. 
%Local absorbing boundary conditions, on the other hand, often sacrifice accuracy in pursuit of preserving the sparsity structure of the linear systems. 
%In this paper, we combine the advantages of both methods by `learning' a highly-accurate and sparse discretization for the DtN in its separable form.
%To this end, novel infinite elements are utilized, which are obtained from the background medium by learning.
%Theoretical considerations suggest that the approximation quality of learned infinite elements should improve exponentially with increasing number of infinite element degrees of freedom. 
%This is confirmed in numerical experiments. 
%Consequently, learned infinite elements are very competetive with established transparent boundary conditions which rely on much more restrictive assumptions on the medium.
%The flexibility of our method for treating inhomogeneous exterior domains is further illustrated by examples from computational helioseismology.
\end{abstract}

% REQUIRED
\begin{keywords}
  transparent boundary conditions, Dirichlet-to-Neumann map, helioseismology, learning, infinite elements, rational approximation, 
  Helmholtz equation
\end{keywords}

% REQUIRED
\begin{AMS}
   	65N30, 35L05, 35J05, 33C10, 85-08
\end{AMS}

\section{Introduction}

To treat time-harmonic wave equations posed on unbounded domains with a numerical discretization, an artificial boundary is typically introduced 
to obtain a finite computational domain.
On this artificial boundary appropriate boundary conditions need to be imposed to ensure that the discrete solution on the computational domain 
is a good approximation of the restriction of the true solution on the unbounded domain.
Such conditions are called absorbing or transparent boundary conditions. 
A large variety of transparent boundary conditions is available in the literature, see
\cite{hagstrom:99,G99}.
%Since their introduction in the 1970s \cite{U73,ZB76,AU77,B77} infinite elements have been applied to a variety of different problems, see the review articles \cite{B80,B88} for an overview. More recent contributions include \cite{A00,HN09}. 
As one class of such methods, infinite elements have been devised and analyzed since the 1970s
for a variety of problems. We only refer to \cite{U73,B77,DG:98,A00,HN09}. 
As an alternative, (low order) \emph{local} non-reflecting boundary conditions have been proposed since the end of the 1970s, see \cite{EM:79,BT:80} for two prominent examples. Later also high-order local non-reflecting boundary conditions using auxiliary variables were developed, %as reviewed in
see \cite{G04,G98,G02}.
%For example, the conditions introduced in \cite{G98} are optimal in the sense that they deliver best approximations of the DtN map in a certain norm and may be implemented by introducing auxiliary variables \cite{G02}.
%For example, the conditions introduced in \cite{G98} are constructed from the requirement to provide best approximations of the DtN map applied to smooth functions in the $L^2$-norm and may be implemented by introducing auxiliary variables \cite{G02}.
%Non-local approximations of the Dirichlet-to-Neumann operator (DtN) are considered in the DtN FE method \cite{G99}. 
Finally we mention -- as probably the most prominent candidate -- Perfectly Matched Layers (PML) \cite{B94}. 
%Finding a suitable transparent boundary condition for a specific application is challenging since all approaches have their specific strengths and limitations.
\par 

Despite the wide variety of approaches, research on transparent boundary conditions is far from finished since for many important problems established methods encounter severe difficulties. 
In particular, this work has been initiated by challenges in the simulation of acoustic waves 
in computational helioseismology \cite{GB05,LBC10}.
Helioseismology studies the solar interior by analyzing oscillation data measured at the visible solar surface (photosphere).
Sound speed and density are strongly varying (nonanalytic) functions in the solar atmosphere (see \cref{fig:c_rho_VALC}) rendering perfectly matched layers, classical infinite elements or approaches based on Green's functions inapplicable.
State-of-the-art local transparent boundary conditions (\cite{FL17,BC18,BFP19}) 
based on strongly simplified models of the 
solar atmosphere and their limitations will be discussed in \cref{sec:helio_models}.
%New methods are required that are flexible enough to cope with the extreme conditions in the atmosphere of the Sun.  
\par 

All approaches discussed above are derived from analytical studies of the differential equation in 
the exterior domain. In this paper we will access the exterior differential equation only implicitly 
via the associated DtN map on the coupling boundary. 
We propose to `learn' infinite elements from this DtN map in a preprocessing step 
in situations where the DtN map can be computed by separation of variables. 
To outline the main ideas, let $\Gamma$ 
denote the coupling boundary between 
the interior and the exterior domain, let $\Delta_{\Gamma}$ be the Laplace-Beltrami operator on $\Gamma$, and 
suppose that the exterior PDE is separable with respect to eigenpairs $(\lambda_{\ell}, v_{\ell})$ of $-\Delta_\Gamma$ where 
$\{v_{\ell}:\ell\in\mathbb{N}\}$ forms a complete orthonormal system of $L^2(\Gamma)$. 
Then the DtN map can be written as a functional calculus
\[
\DtN u_0 =  \dtn(-\Delta_\Gamma)u_0 = \sum\nolimits_{\ell=1}^\infty \dtn(\lambda_\ell) \langle u_0,v_\ell\rangle v_\ell,
\]
where the complex-valued function $\dtn$ is initially defined on the spectrum of $-\Delta_\Gamma$ by quotients of 
derivatives and function values of separated ordinary differential equations. 
This formula, which will be rederived in section \cref{ssec:dtnmaps_framework}, can be found in many places (see, e.g., section 3.2 of \cite{FI98}), although the interpretation as a functional calculus is less common.
It turns out that $\dtn$ has a natural analytic extension to a neighborhood in the complex plane. 

In a nutshell, learned infinite elements take the algebraic structure of  classical infinite element discretizations as in \cite{A00,DG:98} 
and optimize over the infinite element matrix entries such that the Schur complement of these matrix entries 
provides the best possible approximation to the Dirichlet-to-Neumann map $\DtN$. Here and in the following by Schur complement we always
mean the Schur complement restricted to the boundary degrees of freedom. 
Since we optimize the Schur complement of the infinite element matrix entries, the shape functions corresponding to these matrix 
entries become irrelevant. In contrast to classical infinite elements, our approach only provides an approximation of the interior solution, but no approximation of the exterior solution at all. 
(Note that shape functions of different types of infinite elements would provide different values of the exterior solution!) 
If the solution in the exterior domain is required as well, it can be evaluated in a postprocessing step using 
integral representations and fast summation techniques (see, e.g., \cite{GD:05}). For a given accuracy, this may be more efficient than a direct approximation 
by classical infinite elements since a smaller computational domain and less exterior degrees of freedom are needed. 

In this paper we study the situation that the computational domain is surrounded by a layer of infinite elements in tensor product form. 
Note that a layer of tensor product finite elements leads to the same algebraic structure of the system matrix, and also the case of 
a finite number of tensor product finite element layers, which occurs in particular in tensor product PML discretizations, 
is a special case of this algebraic structure. We will show in \cref{sec:derivation} and \cref{section:tensor_product_discr} that the Schur complement of 
the exterior degrees of freedom in system matrices of this structure can be expressed in terms of the discrete version
$-\Delta_{\Gamma,h}$ of the Laplace-Beltrami operator on $\Gamma$ corresponding to the finite element discretization of the boundary $\Gamma$.
More precisely, if $(\underline{\lambda}_{\ell},\vh_{\ell})$ are eigenpairs of $-\Delta_{\Gamma,h}$ which are orthonormal in 
the discrete approximation $\langle \cdot, \cdot \rangle_h$ of the $L^2$ inner product, we show that that the Schur complement 
is described by a discrete DtN map of the form
%
% On the discrete level consider the mass and stiffness matrices $\massFEM$ and $\stiffFEM$ of a finite element discretization on $\Gamma$, 
% and let $(\underline{\lambda}_{\ell},v_{\ell})$ be eigenpairs of the matrix pencil $(\stiffFEM,\massFEM)$ such that 
% $v_{\ell}^\top \massFEM v_k= \delta_{\ell k}$ for all $\ell,k$. Then transparent boundary conditions of tensor product form lead to discrete 
% DtN maps of the form 
%At the discrete level let $-\Delta_{\Gamma,h}$ be the discrete version of the Laplace-Beltrami operator on $\Gamma$, 
%and let $(\underline{\lambda}_{\ell},\vh_{\ell})$ be eigenpairs of the operator which are orthonormal in a suitable inner product $\langle \cdot, \cdot \rangle$. 
%We will show in \cref{sec:derivation} and \cref{section:tensor_product_discr} that transparent boundary conditions of tensor product form lead to discrete DtN maps of the form 
\[
\dDtN\,\uh_0 = \ddtn(-\Delta_{\Gamma,h})\uh_0
= \sum\nolimits_{\ell} \ddtn(\underline{\lambda}_{\ell})
\langle \uh_0,\vh_\ell\rangle_h \vh_{\ell}
\]
with a complex-valued rational function $\ddtn$.
%by taking the Schur complement of all exterior degrees of freedom. Here $\ddtn$ is a complex-valued rational function 
%which is defined in terms of the exterior system matrices in propagation direction. 

Therefore, we end up with a rational approximation problem. More precisely, we have to find a function 
$\ddtn(\lambda)=p(\lambda)/q(\lambda)$ with 
complex polynomials $p$ and $q$ of degrees $N+1$ and $N$, respectively, that best approximates the function $\dtn$ with respect to some norm. 
Here $N$ is the number of exterior degrees of freedom per degree of freedom on the boundary, and $\ddtn$ is defined in terms of the  
$(N+1)\times (N+1)$ components of the local system matrices of the infinite elements, over which we optimize.  
%The idea of learned infinite elements is to choose the (small, say $(N+1)\times (N+1)$) system matrices in propagation direction such that 
%$\ddtn$ is a good approximation to $\dtn$ in some sense. This leads to a rational approximation problem. 
After a reduction step the number of unknowns grows only linearly with $N$. 
%, the number of additional degrees of freedom for each degree of freedom on $\Gamma$. 
Due to the meromorphic extension property of the function $\dtn$ we can appeal to results in rational 
approximation theory to show that the approximation error decreases exponentially in $N$, at least on bounded intervals. 

While the approach can easily be generalized to most other discretization schemes, in this article we will consider the finite element method as the underlying discretization of the interior domain for ease of presentation.

%We consider exterior domains which can be parameterized by $\Gamma\times [1,\infty)$, where $\Gamma$ is a smooth Riemannian manifold 
%with Laplace-Beltrami operator $\Delta_{\Gamma}$, and 

%This paper aims to develop efficient and accurate transparent boundary conditions separable geometries such as 
%stratified media. It is motivated by forward simulations in helioseismology \cite{GB05,GBS10}. 
%To this end, we take advantage of a separable geometry and stratification of the coefficients.

The remainder of this paper is structured as follows.
In \cref{sec:derivation} we introduce the concept of learned infinite elements starting from a discrete algebraic point of view. A crucial assumption used in that section and throughout the whole paper is that the discrete differential operator associated to the exterior domain problem can be written in tensor product form. This requires that the same can be done at the continuous level. 
A corresponding class of PDE problems is discussed in \cref{sec:dtnmaps}, and several examples are given.
The \cref{sec:approx_DtN} is devoted to the approximation problem between $ \dtn$ and $\ddtn$.
Theoretical considerations suggest that even when the number of parameters in the ansatz for $\ddtn$ arising from learned infinite elements is reduced to $O(N)$ the approximation should converge exponentially with increasing number $N$ of infinite element degrees of freedom. 
In \cref{section:tensor_product_discr} we turn our attention to (finite element based) discretizations with tensor product form of the exterior domain including the generic form of learned infinite elements proposed in this paper.
A variety of numerical experiments presented in \cref{sec:numexp} corroborates this finding. 
Both homogeneous problems which allow for a comparison with established transparent boundary conditions and first results for helioseismology are presented.
The paper ends with a conclusion and outlook towards further research.

%\Cref{eq:Helmholtz int,eq:Neumann bc obstacle,eqeqSommerfeld rc} show three aligned equations.
%\begin{align}
%-\nabla \cdot \left( \frac{1}{\rho} \nabla{u} \right) - \frac{\sigma^2}{\rho c^2} u &= 0,  \quad\text{in } \mathbb{R}^{2} \setminus \calK \label{eq:Helmholtz int} \\
%\frac{\partial u}{ \partial \mathbf{n}} &= g, \quad\text{on } \partial \calK \label{eq:Neumann bc obstacle} \\
 %\lim_{r \rightarrow \infty}{ r^{1/2} \left(  \frac{\partial}{\partial r} - i \omega  \right) u(x)  } &= 0 \label{eq:Sommerfeld rc}.
%\end{align}

\section{Tensor product Learned Infinite Elements}\label{sec:derivation}
In this section we derive the learned infinite element approach at an algebraic discrete level.

\subsection{Problem description}
Suppose we have a finite element discretization of a linear PDE with a vector $\uh_I$ of degrees of freedom (\dofs) in the 
interior of the domain $\Omega_{\mathrm{int}} \subset \mathbb{R}^d,~d\in\mathbb{N}$ and a vector $\uh_\Gamma$ of \dofs~on some boundary $\Gamma\subset \partial \Omega_{\mathrm{int}}$.
As a partitioned linear system we obtain
\[
\begin{bmatrix} L_{II} & L_{I\Gamma}\\ L_{\Gamma I}& L_{\Gamma\Gamma}^{\mathrm{int}}\end{bmatrix}
\begin{bmatrix}\uh_I \\ \uh_\Gamma\end{bmatrix}
= \begin{bmatrix}\fh_I \\ \fh_\Gamma^{\mathrm{int}} + \fh_\Gamma^{\mathrm{N}}\end{bmatrix}.
\]
Here $L_{\Gamma\Gamma}^{\mathrm{int}}$ and $\fh_{\Gamma}^{\mathrm{int}}$ represent couplings and sources stemming from the interior discretization, especially volume terms from $\Omega_{\mathrm{int}}$ close 
to $\Gamma$. $\fh_\Gamma^{\mathrm{N}}=\mass \vh^{\mathrm{N}}$
where $\vh^{\mathrm{N}}$ represents the Neumann data (or more generally the right hand side of the natural boundary condition, but we will assume that both coincide) and
$\mass$ is the mass matrix of the finite element discretization on $\Gamma$, i.e.\ $\mass_{i j} = \int_{ \Gamma }{  \FEshape_{i}  \FEshape_{j} \, \text{d} \hat{x}  }$ and $\FEshape_i$ denote the shape functions on $\Gamma$ such that $\uh_\Gamma$ represents the function $\sum_i (\uh_\Gamma)_i\FEshape_i$. 

In this paper we only study the solution of PDEs on unbounded domains $\Omega\subset\mathbb{R}^d$ 
so that the Neumann data $\vh^{\mathrm{N}}$ is not given but is rather the result of the solution of a linear, homogeneous \emph{exterior} problem on $\Omega_{\mathrm{ext}} = \Omega\setminus \overline{\Omega_{\mathrm{int}}}$ for the Dirichlet data provided by 
$\uh_\Gamma$. 
The one-to-one relation between $\uh_\Gamma$ and $\vh^{\mathrm{N}}$ is encoded in the Dirichlet-to-Neumann operator $\dDtN^{\mathrm{ext}}$ ($\vh^{\mathrm{N}} = - \dDtN^{\mathrm{ext}} \uh_\Gamma$), a tool that plays an important role
also in other domain decomposition methods in the quest for appropriate transmission conditions \cite{TW05}.
The $\dDtN^{\mathrm{ext}}$ operator is a dense matrix and allows to write the problem as 
\[
\begin{bmatrix} L_{II} & L_{I\Gamma}\\ L_{\Gamma I}& L_{\Gamma\Gamma}^{\mathrm{int}}+ M\dDtN^{\mathrm{ext}}\end{bmatrix}
\begin{bmatrix}\uh_I \\ \uh_\Gamma\end{bmatrix}
= \begin{bmatrix}\fh_I \\ \fh_\Gamma^{\mathrm{int}}\end{bmatrix}.
\]
Since the use of the full matrix $\dDtN^{\mathrm{ext}}$ may be computationally inefficient, we wish 
to approximate this system of equations by a sparse system involving additional degrees of freedom $\uh_E$:
\[
\begin{bmatrix} L_{II} & L_{I\Gamma} &0\\ L_{\Gamma I}& L_{\Gamma\Gamma}^{\mathrm{int}}+L_{\Gamma\Gamma} & L_{\Gamma E} \\
0 & L_{E\Gamma} & L_{EE}\end{bmatrix}
\begin{bmatrix}\uh_I \\ \uh_\Gamma \\ \uh_E\end{bmatrix}
= \begin{bmatrix}\fh_I \\ \fh_\Gamma \\ 0\end{bmatrix}.
\]
In other words, we wish to find sparse matrices $L_{\Gamma \Gamma}$, $L_{\Gamma E}$, $L_{E\Gamma}$ and $L_{EE}$ of small size 
such that the Schur complement of $L_{EE}$ approximates $M\dDtN^{\mathrm{ext}}$:
\begin{equation}\label{def:discreteDtN}
\dDtN^{\mathrm{ext}} \stackrel{!}{\approx} \dDtN \qquad \mbox{with}\qquad 
\dDtN:= M^{-1}\left(L_{\Gamma \Gamma} - L_{\Gamma E}L_{EE}^{-1}L_{E\Gamma}\right).
\end{equation}
%In the specific setting we will consider later, it turn out that the full matrix $\dDtN^{\mathrm{ext}}$ does not have
% to be set up, but for the moment assume that it is given. 

\subsection{Tensor product ansatz}
In analogy to the mass matrix on $\Gamma$, let us introduce the stiffness matrix 
$\stiff = (\int_{ \Gamma }{   \nabla_{\Gamma} \FEshape_{i}  \nabla_{\Gamma} \FEshape_{j}    \,   \text{d} \hat{x}  })_{ij}$ 
on the boundary $\Gamma$ so that the discrete Laplace-Beltrami operator on $\Gamma$ is $-\Delta_{\Gamma,h} = M^{-1} K$. Moreover, 
recall that the Kronecker (or tensor) product of an $m \times n$ matrix $A$ and a $p \times q$ matrix $B$  is the $mp \times nq$ block matrix given by 
\begin{equation*}
A \otimes B  =
 \begin{bmatrix}
  a_{11} B &  \ldots &   a_{1n} B \\
   \vdots    &  \ddots &   \vdots  \\
 a_{m1} B   &  \ldots &   a_{mn} B \\
\end{bmatrix}.
\end{equation*} 
In this paper we will study the tensor product ansatz
\begin{align}\label{eq:ansatz}
\begin{bmatrix}  L_{\Gamma\Gamma} & L_{\Gamma E} \\ L_{E\Gamma} & L_{EE}\end{bmatrix}
= \dstiffr \otimes \mass + \dmassr\otimes \stiff 
= \begin{bmatrix}\dstiffr_{\Gamma\Gamma} & \dstiffr_{\Gamma E}\\ \dstiffr_{E\Gamma}& \dstiffr_{EE}\end{bmatrix}\otimes \mass 
+ \begin{bmatrix}\dmassr_{\Gamma\Gamma} & \dmassr_{\Gamma E}\\ \dmassr_{E\Gamma}& \dmassr_{EE}\end{bmatrix}\otimes \stiff 
\end{align}
with matrices $\dstiffr,\dmassr \in \mathbb{C}^{(N+1)\times (N+1)}$ and $\dstiffr_{\Gamma\Gamma},\dmassr_{\Gamma\Gamma}\in\mathbb{C}$.  
Here $N$ should be small, say of the order of $10$. This ansatz is motivated by the fact that many successful transparent
boundary conditions have this algebraic structure as detailed in \cref{section:tensor_product_discr}. 

In the following we will study for which class of matrices $\dDtN^{\mathrm{ext}}$ this ansatz can be successful. 
It has the advantage that only the small matrices $\dstiffr$ and $\dmassr$ have to be `learned'. 
Of course, we could use more general matrices $L_{\Gamma \Gamma}$, $L_{\Gamma E}$, $L_{E\Gamma}$ and $L_{EE}$ with more 
free parameters. This would enable the approximation of more general matrices $\dDtN^{\mathrm{ext}}$, 
but for the examples we have in mind the ansatz \eqref{eq:ansatz} turns out to be successful. 

\subsection{Feasibility analysis}
The following result analyzes the eigenvectors and eigenvalues of $\dDtN$. 
\begin{proposition}\label{prop:1}
If $(\underline{\lambda}_{\ell},\vh_{\ell})$ is an eigenpair of the discrete Laplace-Beltrami operator 
$-\Delta_{\Gamma,h} = M^{-1}K$ on $\Gamma$, then 
\[
\dDtN\vh_\ell = \ddtn(\underline{\lambda}_{\ell}) \vh_{\ell}
\]
with the rational function $\ddtn$ defined, for $\lambda\in\mathbb{C}$ by
\begin{align}\label{eq:dtn_fct_discrete}
\ddtn(\lambda):= \dstiffr_{\Gamma\Gamma}+\lambda\dmassr_{\Gamma\Gamma}
-   (\dstiffr_{\Gamma E}+\lambda\dmassr_{\Gamma E}) (\dstiffr_{EE}+\lambda\dmassr_{EE})^{-1}
(\dstiffr_{E\Gamma}+\lambda\dmassr_{E\Gamma}).   
\end{align}
\end{proposition}

\begin{proof}
For the first term in the definition of $\dDtN$, cf. \eqref{def:discreteDtN}, we have
\begin{align*}
\mass^{-1}L_{\Gamma\Gamma}\vh_{\ell} &= \mass^{-1}(\dstiffr_{\Gamma\Gamma}\otimes\mass +\dmassr_{\Gamma\Gamma}\otimes\stiff) \vh_{\ell}
= M^{-1}(\dstiffr_{\Gamma\Gamma} \mass \vh_{\ell}+\dmassr_{\Gamma\Gamma}\stiff \vh_{\ell})\\
&= \mass^{-1}(\dstiffr_{\Gamma\Gamma} \mass \vh_{\ell}+\dmassr_{\Gamma\Gamma}\lambda_{\ell} \mass \vh_{\ell})
= (\dstiffr_{\Gamma\Gamma}+\lambda_{\ell} \dmassr_{\Gamma\Gamma}) \vh_{\ell}.
\end{align*}
Using a similar computation and the rule $(C\otimes D)(E\otimes F) = (CE\otimes DF)$, we obtain 
\begin{align*}
&M^{-1}L_{\Gamma E}L_{EE}^{-1}L_{E\Gamma}\vh_{\ell}\\
&= M^{-1}[(\dstiffr_{\Gamma E}+\lambda_{\ell}\dmassr_{\Gamma E})\otimes M]
[(\dstiffr_{EE}+\lambda_{\ell}\dmassr_{EE})\otimes M]^{-1}
[(\dstiffr_{E\Gamma}+\lambda\dmassr_{E\Gamma})\otimes M]\vh_{\ell}\\
&= (\dstiffr_{\Gamma E}+\lambda_{\ell}\dmassr_{\Gamma E}) (\dstiffr_{EE}+\lambda_{\ell}\dmassr_{EE})^{-1}
(\dstiffr_{E\Gamma}+\lambda_{\ell}\dmassr_{E\Gamma})(M^{-1}MM^{-1}M)\vh_{\ell}
\end{align*}
for the second term. Putting both together yields the result. 
\end{proof}

If we orthonormalize the eigenvectors $\vh_{\ell}\in \mathbb{R}^{n_\Gamma}$ with respect to the natural inner product 
$\langle \vh,\wh\rangle_M:= \vh^* M\wh$ on $\mathbb{C}^{n_{\Gamma}}$, then applying $\dDtN$ to 
the identity 
$\vh=\sum_{\ell}\vh_{\ell}\, \vh_{\ell}^* M \vh$ yields 
\[
\dDtN = \sum\nolimits_{\ell}\ddtn(\lambda_\ell) \vh_{\ell} \vh_{\ell}^* M =:  \ddtn(-\Delta_{\Gamma,h}).
\]
In other words, $\dDtN$ is a functional calculus at the discrete Laplace-Beltrami operator  $-\Delta_{\Gamma,h}$
in the space $\mathbb{C}^{n_\Gamma}$ equipped with the inner product $\langle\cdot,\cdot\rangle_M$. 
Therefore, we may hope to be able to approximate $\dDtN^{\mathrm{ext}}$ well using our ansatz \eqref{eq:ansatz} 
if it has the same property, i.e.\ if
\begin{align}\label{eq:Dtnex}
\dDtN^{\mathrm{ext}} = \ddtn^{\mathrm{ext}}(-\Delta_{\Gamma,h}) %M^{-1}K)
\end{align}
for some function $\ddtn^{\mathrm{ext}}$. Note that $\ddtn$ is a rational function 
with numerator of degree $N-1$ and nominator of degree $N$. With such rational function we may hope to interpolate 
$\ddtn^{\mathrm{ext}}$ at at most $O(N)$ points. However, our approach will only be efficient, 
if $N$ can be chosen much smaller than $n_{\Gamma}$. Therefore, in order to at least approximate $\ddtn^{\mathrm{ext}}$ 
accurately on the spectrum of $-\Delta_{\Gamma,h}$ by a rational function of small degree, $\ddtn^{\mathrm{ext}}$ must 
have an extension to a well-behaved smooth function. Of course, every function defined on a finite subset of $\mathbb{R}$ 
can be extended to an analytic function on $\mathbb{R}$ or even $\mathbb{C}$ in many ways, e.g.\ by polynomial interpolation. 
However, fast approximation rates only hold true if derivatives of such an extension satisfy reasonable bounds, and this is what 
we mean by 'well-behaved'. By Cauchy's integral formula, this holds true in particular if $\ddtn^{\mathrm{ext}}$ 
has an analytic extension to a complex neighborhood of the spectrum which does not grow too rapidly with the distance to the spectrum. There is one exception to this smoothness requirement which will turn out relevant: Of course, with our approach we can deal with isolated singularities described by poles of rational functions. 

To summarize our results the ansatz \eqref{eq:ansatz} will be successful if 
\begin{itemize}
\item the matrix $\dDtN^{\mathrm{ext}}$ diagonalizes in the basis of the eigenvectors of the 
discrete Laplace-Beltrami operator $-\Delta_{\Gamma,h}$ such that \eqref{eq:Dtnex} holds and 
\item $\ddtn^{\mathrm{ext}}$, initially defined on the spectrum of $-\Delta_{\Gamma,h}$ can be extended to a 
well-behaved smooth function, possibly with the exception of isolated poles. 
\end{itemize}

\section{DtN maps for separable PDEs and examples}\label{sec:dtnmaps}
In this section we turn to a continuous setting and show that the two criteria at the end of the previous section are typically 
satisfied for PDEs which are separable in the exterior domain.  

\subsection{Framework}\label{ssec:dtnmaps_framework} Let us consider a domain $\Omega\subset \mathbb{R}^d$ consisting of a bounded interior domain $\Omega_{\text{int}}$ 
and an exterior domain  $\Omega_{\text{ext}}$, which is typically unbounded. More precisely, we assume 
that $\overline{\Omega}= \overline{\Omega}_{\text{int}}
\cup \overline{\Omega}_{\text{ext}}$, $\Omega_{\text{int}}\cap \Omega_{\text{ext}} = \emptyset$, and the 
coupling boundary $\Gamma:= \overline{\Omega}_{\text{int}}\cap \overline{\Omega}_{\text{ext}}$ is smooth. 
We consider a linear second-order elliptic PDE (elliptic in the sense that the principal part is elliptic)
\begin{align}\label{eq:separable_bvp}
\calL u = f \qquad \mbox{in }\Omega\qquad + \qquad \text{radiation condition},
\end{align}
such that $\operatorname{supp} f\subset \overline{\Omega}_{\text{int}}$. The radiation conditions at infinity 
will be specified in the examples. 
If the boundary of $\Omega$ is not empty ($\partial \Omega\neq \emptyset$), then a boundary condition on $\partial \Omega$ is required in addition.  
Our main focus will be on the exterior Dirichlet problem 
\begin{align}\label{eq:bvp_ext}
\begin{aligned}
&\calL u = 0 && \mbox{in }\Omega_{\text{ext}}\qquad + \qquad \text{radiation condition},\\
& u=u_0 && \mbox{on }\Gamma,
\end{aligned}
\end{align}
for some Dirichlet data $u_0$. The Dirichlet-to-Neumann map is defined by 
$\DtN u_0:= \frac{\partial u}{\partial \nu}|_\Gamma$ where $\nu$ denotes the unit normal vector pointing into the exterior 
of $\Omega_{\text{ext}}$. 

We assume that there exists a diffeomorphism
% such that 
\begin{subequations}
\begin{align}
\Psi: &[a,\infty)\times \Gamma \to \overline{\Omega}_{\text{ext}} \text{ such that} \nonumber \\
&\Psi(\{a\}\times\Gamma) = \Gamma,\\
\label{eq:Psi_normalization}
&|\partial_r\Psi(a,\widehat{x})|_2 =1,& \widehat{x}\in \Gamma,
\end{align}
\end{subequations}
with the derivative $\partial_r$ with respect to the first variable. Further, we assume that $\calL$ separates in the coordinates given by $\Psi$ as follows:
\begin{align}\label{eq:separable_bvp2}
(\calL u) \circ \Psi = \left[\stiffr\otimes \Id_{\Gamma}+  \massr\otimes (-\Delta_\Gamma)\right] (u\circ \Psi)
\end{align}
Here $\stiffr$ is a second-order differential operator in $r$, 
$\massr$ is a multiplication operator on $L^2([a,\infty))$, $\Delta_\Gamma$ is the Laplace-Beltrami 
operator on $\Gamma$,  and $\Id_{\Gamma}$ the identity operator 
on $L^2(\Gamma)$. Further recall that tensor products of linear operators are uniquely defined by $(A\otimes B)(v\otimes w) = Av\otimes Bw$, 
and for $L^2$-functions $v$ and $w$ one has $(v\otimes w)(x,y)= v(x)w(y)$. (In other words, $A$ acts on the first variable and 
$B$ on the second variable.) 

As before, let $\{\vh_{\ell}:\ell\in \mathbb{N}\}$ be an orthonormal basis of $L^2(\Gamma)$ consisting of eigenvectors of 
$-\Delta_{\Gamma}$ with corresponding eigenvalues $\lambda_{\ell}$. Such an orthonormal basis exists since the 
Laplace-Beltrami operator on a smooth compact manifold is self-adjoint and has a compact resolvent. 
To solve \eqref{eq:separable_bvp} in $\Omega_{\textrm{ext}}$ in the coordinates $(r,\widehat{x})$ 
by separation of variables, we have to study the separated equations 
\begin{subequations}\label{eqs:bvp_dtn_numbers}
\begin{align}\label{eq:ODE_dtn_numbers}
&\left[ \stiffr  + \lambda_{\ell}\massr  \right] \Lambda_{r}( \lambda_{\ell}) = 0\qquad\mbox{for } r \in [a,\infty),\\
\label{eq:bc_dtn_numbers}
&\Lambda_a(\lambda_{\ell}) =1, \qquad \Lambda_r \mbox{ satisfies the radiation condition.} 
\end{align}
\end{subequations}
Then the radiating solution to $\calL u=0$ in $\Omega_{\text{ext}}$ with Dirichlet data $u|_{\Gamma}=u_0$ can be represented as 
\begin{equation}\label{eq:abstract_separation}
(u\circ\Psi)(r,\cdot) = \Lambda_{r}( -\Delta_{\Gamma}) u_{0} = \sum\nolimits_{\ell=1}^\infty \Lambda_{r}( \lambda_{\ell}) \vh_{\ell} (u_{0},\vh_{\ell} )_{\Gamma}.
\end{equation}
Since Neumann data (with the interior normal vector on $\Omega_{\text{int}}$) are given by $-\partial_r(u\circ\Psi)|_{r=a}$, the Dirichlet-to-Neumann map 
$\DtN:u_0\to -\partial_r(u\circ\Psi)|_{r=a}$ is given by 
\begin{align}\label{eq:DtN}
\DtN = \dtn(-\Delta_{\Gamma}) \quad \mbox{with}\quad 
\dtn(\lambda) := - \partial_{r} \Lambda_{r}( \lambda) \vert_{r = a}. 
\end{align}

Often, analytical expressions for the solutions of \cref{eq:ODE_dtn_numbers}-\cref{eq:bc_dtn_numbers} and consequently for $\dtn$ are known as the following examples show. 
Otherwise, these solutions can be obtained numerically at negligible costs by solving ordinary differential equations. 
This is how we proceed for the case of the solar atmosphere as considered in \cref{ex:helio}.

\subsection{Examples of (piecewise) homogeneous exterior domains} 
We first discuss two examples of homogeneous exterior domains. 
We will show in particular that the second criterion at the end of \cref{sec:derivation} 
is satisfied in the continuous setting since $\dtn$-functions have natural analytic extensions to complex neighborhoods of the spectrum.

\begin{example}[homogeneous Helmholtz equation in the exterior of a ball]\label{ex:Helmholtz_ball}
We consider the exterior Dirichlet problem 
\begin{subequations}\label{eqs:scat}
\begin{align}
&-\Delta u -\wavenr^2 u= 0&& \text{in }\Omega_{\text{ext}}:=\{x\in\mathbb{R}^d:|x|_2>a\}, \\
&u=u_0&&\mbox{on }\Gamma = a \mathbb{S}^{d-1},\\
&\lim_{r\to\infty}r^{(d-1)/2}\left(\frac{\partial u}{\partial r}-i\wavenr u\right) =0&& r=|x|_2
\end{align}
\end{subequations}
which occurs in the modelling of waves scattered by bounded obstacles, either penetrable or impenentrable 
(see, e.g., \cite{CK:12} for more information). Setting
\begin{align}\label{eq:spherical_coo}
\Psi(r,\widehat{x}) := \frac{r}{a}\widehat{x} \qquad \text{for }r\in [a,\infty),\,\widehat{x}\in a\mathbb{S}^{d-1},
\end{align}
we obtain 
\begin{equation} \label{eq:helmholtz sep}
(-\Delta u -\wavenr^2 u)\left(\Psi(r,\widehat{x})\right) 
= \left(-r^{1-d}\partial_r(r^{d-1}\partial_r)-\frac{a^2}{r^2}\Delta_{\Gamma}-\wavenr^2\Id\right)(u\circ \Psi)(r,\widehat{x}),
\end{equation}
i.e.\ $\stiffr = -r^{1-d}\partial_r(r^{d-1}\partial_r)-\wavenr^2\Id$ and $(\massr v)(r) = \frac{a^2}{r^2}v(r)$. 
Therefore, \eqref{eq:ODE_dtn_numbers} is the Bessel differential equation. 
For $d=2$ the solution satisfying \eqref{eq:bc_dtn_numbers} is given by 
$\Lambda_r(\lambda) = H^{(1)}_{ a \sqrt{\lambda   }}(\wavenr r)/ H^{(1)}_{a \sqrt{\lambda   }}(\wavenr a) $ 
in terms of the Hankel function $H^{(1)}_{a \sqrt{ \lambda   }}$ of the first kind of order $a \sqrt{\lambda}$. 
This leads to the $\dtn$-function 
\begin{equation}\label{eq:dtn_fct_helmholtz2D}
\dtn^{\mathrm{hom}}(\lambda) = \frac{-  \wavenr }{ H^{(1)}_{a \sqrt{\lambda   }}(\wavenr a) }
(H^{(1)}_{ a \sqrt{\lambda   }})^{\prime}(\wavenr a).
\end{equation} 
This function is initially defined on the spectrum $\sigma(-\Delta_\Gamma) = \{(\ell/a)^2:\ell\in \mathbb{N}\cup\{0\}\}$ where the order 
of the Hankel functions is integer-valued. However, Hankel functions can be defined 
also for complex-valued orders $\nu$, and $H^{(1)}_{\nu}(x)$ is an entire function of $\nu$  
(see \cite{MK60}). Due to the recurrence relation 
$2(H^{(1)}_{\nu})^{\prime}(x)= H^{(1)}_{\nu-1}(x) -H^{(1)}_{\nu+1}(x)$, the derivative
$(H^{(1)}_{\nu})^{\prime}(x)$ with respect to $x$ is an entire function of $\nu$ as well. 
Therefore, choosing the negative real axis $\mathbb{R}_-:=\{-x:x\geq 0\}$ 
as branch cut of the square root function, 
$\dtn^{\mathrm{hom}}$ is a meromorphic function on $\mathbb{C}\setminus\mathbb{R}_-$. 
\cite{MK60} also contains analytical results, e.g.\ on the asymptotic behavior of the zeros of $H^{(1)}_{\bullet}(x)$ and hence of the poles of $\dtn^{\mathrm{hom}}$. 
These poles are plotted in \cref{fig:dtn-scat} using numerical computations. 
\end{example}

For computing roots and poles of meromorphic functions we utilize a mesh-based technique presented in \cite{PK15,PK18}. 
 First, roots and poles are isolated using adaptive mesh refinement driven by an analysis of the function's phase and then verified by applying a discretized version of 
 the argument principle. The special functions for evaluating $\dtn$ at complex numbers are taken from the library  \texttt{mpmath} \cite{FJ13}.
 %These are given by the roots of the function $z \mapsto H^{(1)}_{a \sqrt{z}}(\omega a)$ which have been computed using the package \texttt{cxroots} \cite{cxroots}.
 %Notice here that this function is analytic outside the origin so that all of its roots within a given contour of the complex plane can be computed by techniques based
 %on the argument principle \cite{KB00}. 

Our next example concerns an inhomogeneous exterior domain involving a jump: 
\begin{example}[Exterior domain with a jump]\label{ex:jump}
We again consider the problem \eqref{eqs:scat} in dimension $d=2$, but now assume that $\wavenr$ depends in a piecewise constant manner on $r$. More precisely, we assume that for some $a \leq \RJump < \infty$ and 
$\wavenr_{I}, \wavenr_{\infty} > 0$ we have 
\begin{equation}\label{eq:disc_wavespeed}
\wavenr(r) = \begin{cases}
\wavenr_{I} & r < \RJump, \\
\wavenr_{\infty} & r > \RJump.
\end{cases}
\end{equation}
From the boundary condition at $r = a$ and matching conditions at $r=\RJump$ it follows
that the $\dtn$ function is given by 
\begin{equation*}
\dtn^{\mathrm{jump}}(\lambda) = \zeta(a \sqrt{\lambda}), \quad   \zeta(\nu) =  - \wavenr_I \left[ A_{\nu} J_{\nu}^{\prime}(\wavenr_I a )  + B_{\nu} Y_{\nu}^{\prime}(\wavenr_I a )  \right].
\end{equation*}
Here $Y_{\nu}$ denotes the Bessel function of the second kind of order $\nu$.
The constants $A_{\nu}$ and $B_{\nu}$, which are derived in the appendix, depend on the problem parameters $a,\RJump,\wavenr_{I}$ and $\wavenr_{\infty}$.

The function $\dtn^{\mathrm{jump}}(\lambda )$ for  $a = 1$, $\RJump=2$ and $\wavenr_{I} = 16$, 
and $\wavenr_{\infty}=8$ is shown in \cref{fig:dtn-jump}.
Additionally, the poles of its meromorphic extension are displayed which are all simple and located extremely close to the positive real axis.
As a result, the behavior $\dtn^{\mathrm{jump}}(\lambda )$ on the domain of interest is essentially dominated by these poles. 
It is then natural to expect and will later be confirmed in numerical experiments that an excellent fit of $\dtn^{\mathrm{jump}}$ may 
be obtained by means of a rational approximation with a few simple poles. 
\end{example}

\begin{example}[Helmholtz equation in a wave-guide]\label{ex:Helmholtz_waveguide} 
Let $\tilde{\Gamma}\subset \mathbb{R}^{d-1}$ be a smooth, bounded domain and define 
$\Gamma:=\{a\}\times \tilde{\Gamma}$ for some $a\in \mathbb{R}$ and 
\[
\Psi(r,\widehat{x}):= \begin{pmatrix}r-a\\ 0 \end{pmatrix}+ \tilde{x}\qquad \mbox{for }r\in[a,\infty), \;\widehat{x}\in \Gamma. 
\]
We consider the Helmholtz equation 
\begin{align}\label{eq:Helmholz_waveguide}
\begin{aligned}
&(-\Delta-\wavenr^2)u =0&& \mbox{in }\Omega_{\text{ext}}:=(a,\infty)\times\tilde{\Gamma}\\
&u=0 &&\text{on }(a,\infty)\times \partial \tilde{\Gamma}
\end{aligned}
\end{align}
together with a radiation condition that will be discussed below. 
Let $\Delta_\Gamma$ denote the Laplace-Beltrami operator on $\Gamma$ with Dirichlet boundary conditions. 
Then obviously 
\[
(-\Delta u -\wavenr^2 u)\left( \Psi(r,\widehat{x})\right) 
= \left(-\partial_r^2 - \Delta_{\Gamma} -\wavenr^2\right) (u \circ\Psi)(r,\widehat{x}),
\]
i.e.\ $\stiffr = -\partial_r^2-\wavenr^2\Id$ and $\massr = \Id$. 
Therefore, \eqref{eq:ODE_dtn_numbers} has the two independent solutions $\exp(\pm i \sqrt{\wavenr^2-\lambda_{\ell}}(\cdot-a))$. 
Here we choose the negative imaginary axis as branch cut of the square root function. Moreover, we assume that 
\begin{align}\label{eq:waveguide_assump}
\wavenr^2\notin \sigma(-\Delta_\Gamma).
\end{align}
Then we say that a solution to \eqref{eq:Helmholz_waveguide} satisfies the radiation condition if it has an expansion 
\eqref{eq:abstract_separation} in terms of the functions $\Lambda_r(\lambda_{\ell}):= \exp(i \sqrt{\wavenr^2-\lambda_{\ell}}(r-a))$.

This leads to the $\dtn$ function
\[
\dtn^{\mathrm{guide}}(\lambda) = -i \sqrt{\wavenr^2-\lambda}.
\]
Under the assumption \eqref{eq:waveguide_assump} it has a holomorphic extension to a neighborhood of $\sigma(-\Delta_\Gamma)$ (see \cref{fig:dtn-guide}). 
However, due to the branch cut singularity it does not have a natural meromorphic extension to the entire complex plane. 
\end{example}

\begin{figure}[htbp]
   \centering
\subfloat[ Homogen.: $\dtn^{\mathrm{hom}}$ ]%
{\label{fig:dtn-scat}
\includegraphics[scale=0.38]{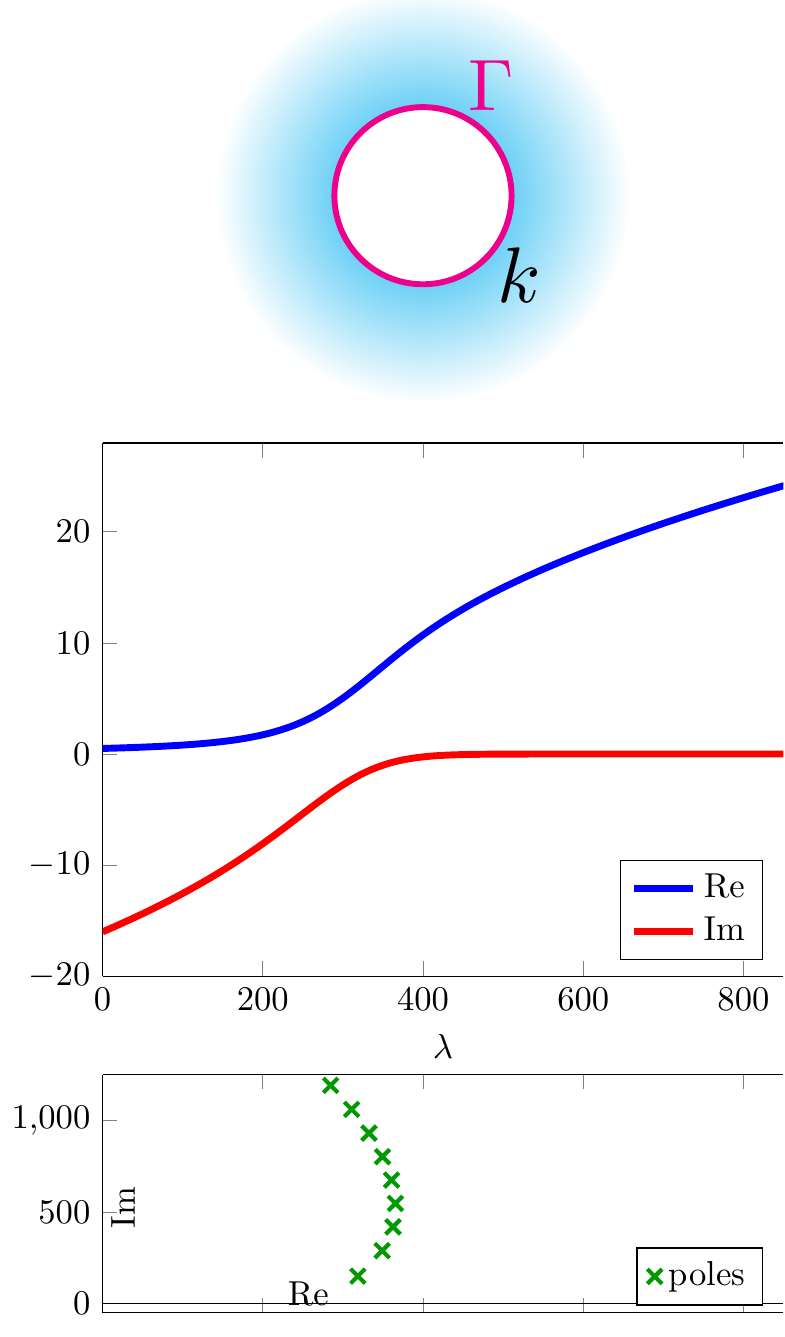}}
\centering
\subfloat[ Jump: $\dtn^{\mathrm{jump}}$]
{ \label{fig:dtn-jump}%
\includegraphics[scale=0.38]{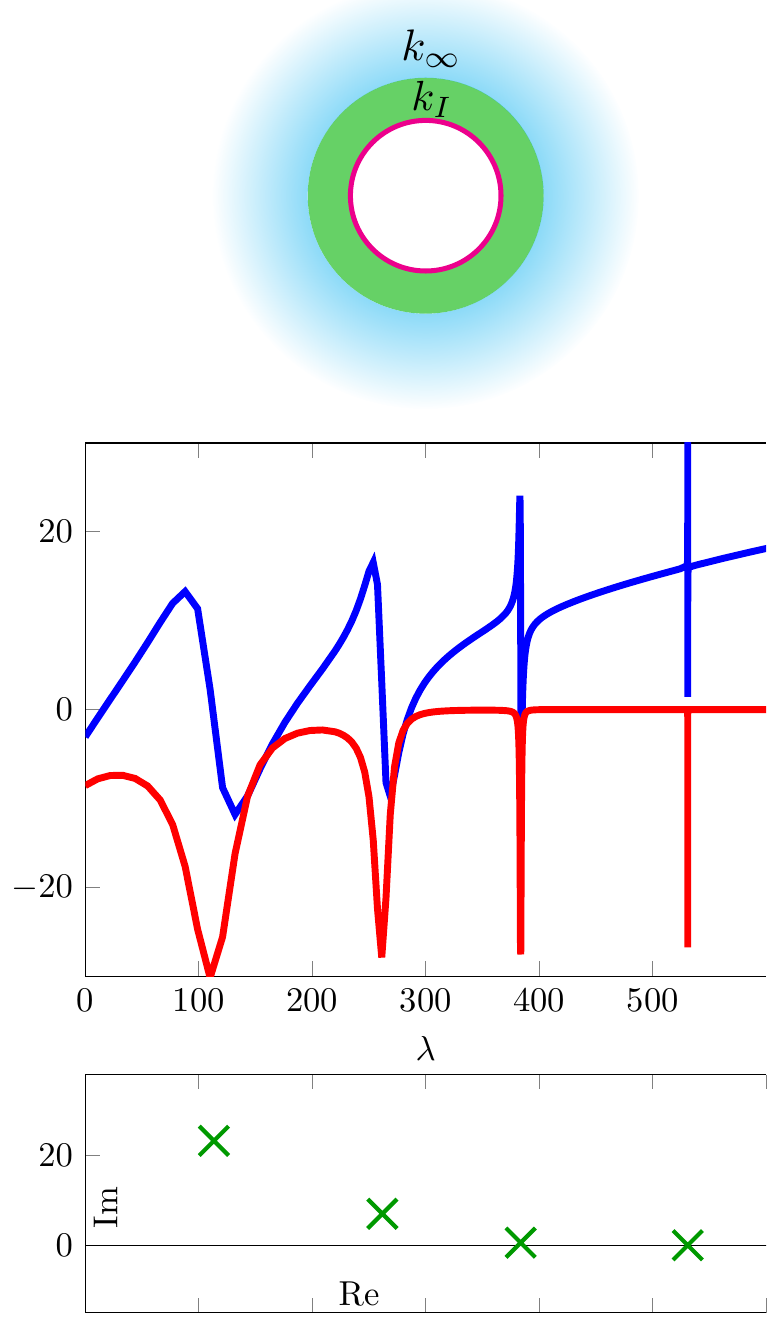}}
\subfloat[ Waveguide:$\,\dtn^{\mathrm{guide}}$]
{ \label{fig:dtn-guide}%
\includegraphics[scale=0.38]{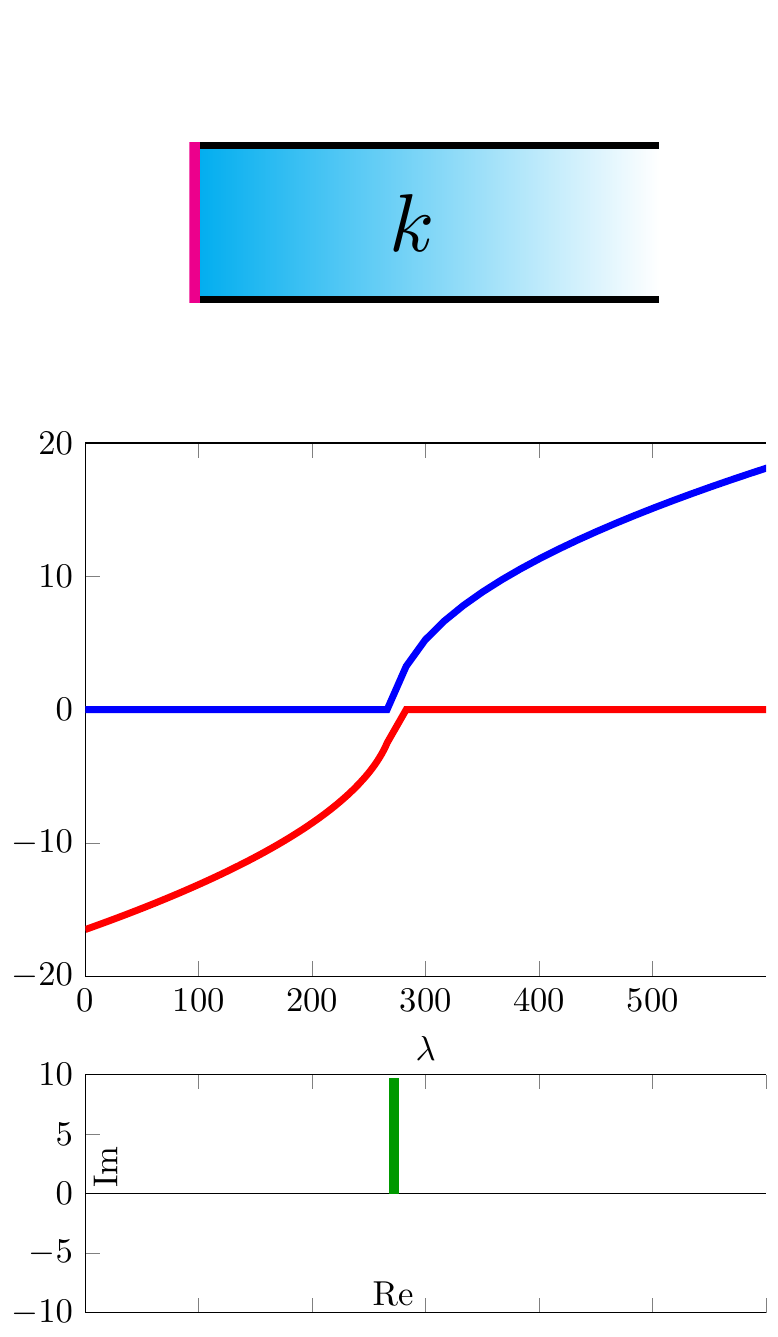}}
\subfloat[ Sun: $\dtn^{\mathrm{VAL-C}}$]
{ \label{fig:dtn-VALC}%
\includegraphics[scale=0.38]{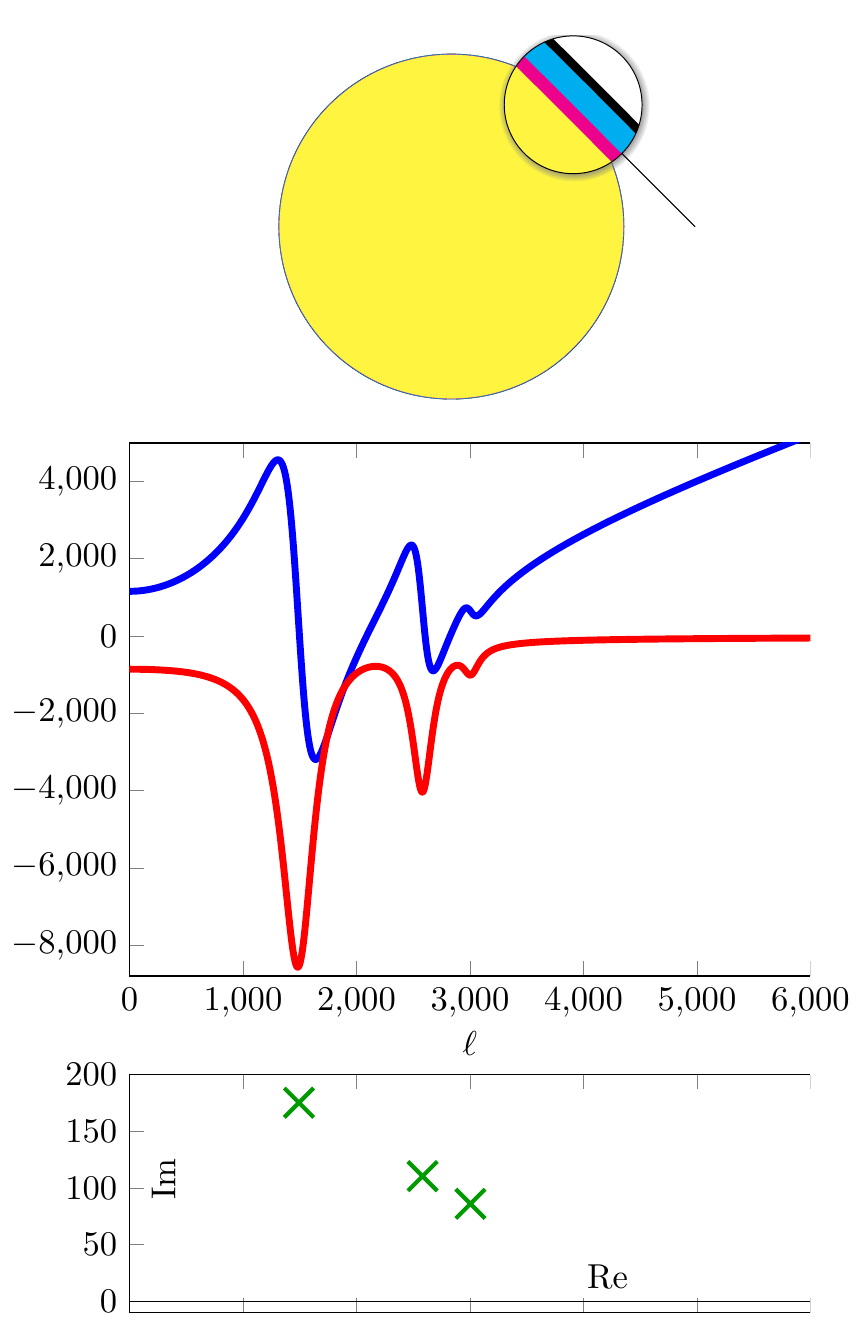}}
\caption{ Comparison of the $\dtn$ functions considered in this paper: 
(a):  homogeneous Helmholtz equation in the exterior of a ball (\cref{ex:Helmholtz_ball}), 
(b) same as (a) with jumping coefficient (\cref{ex:jump}), 
(c) homogeneous wave-guide with Dirichlet boundary conditions (\cref{ex:Helmholtz_waveguide}), 
(d) VAL-C model of the solar atmosphere with Neumann condition at the corona (\cref{ex:helio}). 
First row: geometries, second row: real and imaginary part of the $\dtn$ functions, third row: 
structure of the singularities of analytic extensions of the $\dtn$ functions, which are poles 
in (a), (b) and (d), and a branch cut in (c).}
\label{fig:dtn-comparison}
\end{figure}

\subsection{A structural argument for the analyticity of $\dtn$}\label{sec:structarg}
In the examples of the previous subsection, analytic extensions could be derived from the explicit form of the $\dtn$-function. 
In the following, we wish to provide a more structural argument for the existence of such natural extensions which is also applicable if $\dtn$ cannot be computed explicitly: 

A main difficulty in the analysis of the boundary value problems 
\eqref{eqs:bvp_dtn_numbers} defining $\dtn$ is the treatment of the radiation condition, 
which is a condition on the asymptotic behavior at infinity. 
As we are only interested in the $\dtn$-function, we may use a transformation of the differential equation, which allows to 
incorporate the radiation condition in the function space and leads to the same $\dtn$ function. The most prominent methods are 
Perfectly Matched Layers (PMLs) which formulate a differential equation for the analytic extensions $\tilde{\Lambda}_r(\lambda)$ of 
$r\mapsto \Lambda_r(\lambda)$ 
to a certain path in the complex plane (see \cref{sec:TP_PML}) for more details). This is done such that the radiation condition is satisfied if and only if the 
analytic extension is bounded. Therefore, $\tilde{\Lambda}_r(\lambda)$ is characterized by a boundary value problem 
$[\tilde{\stiffr}+\lambda\tilde{\massr}]\tilde{\Lambda}_r(\lambda)=0$, $\tilde{\Lambda}_a(\lambda)=1$, which is 
of the same form as \eqref{eqs:bvp_dtn_numbers}, but without the radiation condition. 
If $\tilde{\stiffr}+\lambda\tilde{\massr}$ is a Fredholm operator for all $\lambda$ in some domain $D\subset\mathbb{C}$ and 
$\tilde{\stiffr}+\lambda\tilde{\massr}$ is invertible for at least one $\lambda\in D$ (e.g.\ $\lambda\in \sigma(-\Delta_\Gamma)$), 
then by analytic Fredholm theory $\lambda\mapsto \tilde{\Lambda}_r(\lambda)$ is a function-valued meromorphic mapping on $D$, 
and hence $\dtn$ is a meromorphic function on $D$. 

%In particular, if $D=\mathbb{C}$, then we are in the situation of \cref{ex:Helmholtz_ball}, i.e.\ $\dtn$ is %meromorphic on $\mathbb{C}$. 
This situation occurs in Examples \ref{ex:Helmholtz_ball} and \ref{ex:jump}.
In \cref{ex:Helmholtz_waveguide} the branch cut occurs since the PML formulation breaks down for $\lambda=\wavenr^2$.

\subsection{Models of the solar atmosphere}\label{sec:helio_models} 
We now discuss in more detail PDE models in helioseismology as mentioned already in the introduction. 
%as the main motivating application which initiated this work. 
Although more complex systems of differential equations have been derived 
to describe solar (and more generally stellar) oscillations, 
it has been shown in \cite{GB17}  that the scalar time-harmonic wave equation 
\begin{equation}\label{eq:helio_original}
-\Div \left( \frac{1}{\rho} \nabla{u} \right)  - \frac{\sigma^2}{\rho c^2} u = 0
\end{equation}
captures the physics of oscillating pressure modes in the Sun to a reasonable extent. 
$\sigma = \omega+i\gamma$ contains the frequency $\omega>0$ and a positive absorption coefficient 
$\gamma(r)>0$.
In the solar atmosphere, sound speed $c$ and density $\rho$ essentially depend only on the radial variable 
$r$, 
but they are usually nonanalytic functions derived from tabulated values that vary 
over several orders of magnitude (see \cref{fig:c_rho_VALC}). We use the VAL-C model 
proposed in \cite{VAL81} here combined with  Models S \cite{CD96} for the solar interior.  
Whereas pressure decreases sharply and monotonically, sound speed first descreases in the solar atmosphere and then increases again towards 
the corona, which is related to an increase of temperature from about $5000^{\circ}\mathrm{C}$ at the photosphere 
to several million $^{\circ}\mathrm{C}$ in the corona. In the so-called convection zone,  
which comprises that upper 28\% of the solar interior, first order terms in \eqref{eq:helio_original}
play a dominant role. Since our focus is on the solar atmosphere, we have omitted this term in 
\eqref{eq:helio_original}. Note that in this form, \eqref{eq:helio_original} satisfies 
the separation condition \eqref{eq:separable_bvp2} with $\Psi$ given by \eqref{eq:spherical_coo}, 
$\stiffr = -r^{1-d}\partial_r(r^{d-1} \frac{1}{\rho} \partial_r)- \frac{\sigma^2}{\rho c^2} \Id$ 
and $(\massr v)(r) = \frac{a^2}{r^2}v(r)$. \par 

 \begin{figure}[htbp]
   \centering
\subfloat[ Sound speed and density. ]%
{\label{fig:c_rho_VALC}\includegraphics[scale=1.0]{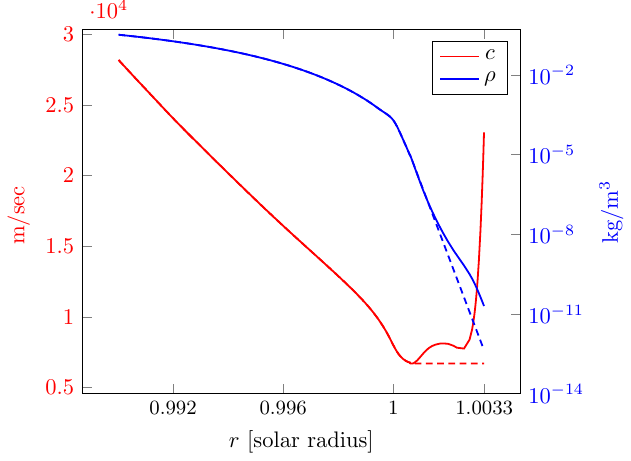}}
\centering
\subfloat[ Real part of the effective potential $v$.]{ \label{fig:pot_VALC}
\includegraphics[scale=1.0]{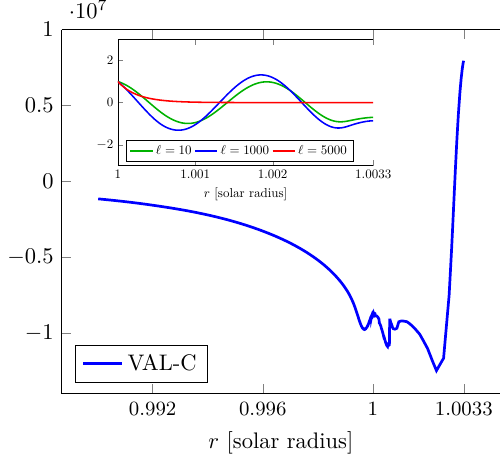}
 }
\caption{ Background coefficients for the Sun. 
Dashed lines in the left plot indicate a simplified model used in \cite{BC18,FL17}. 
The right plot shows the corresponding effective potential after transformation to a Schr\"odinger-type equation at $7.0$ mHz. 
	 Additionally, some of the functions $\Re{(\Lambda_r(\lambda_{\ell}))}$ appearing in the computation of  $\dtn^{\mathrm{VAL-C}}$ are displayed in the small inset. }
\label{fig:helio_coeff}
\end{figure}

\begin{example}[transformed equation for solar sound waves]\label{ex:helio}
%To apply our method we proceed similar as for the Helmholtz equation and obtain 
%\begin{align} \label{eq:helio_original_sep}
%\begin{aligned}
%&	\left(-\Div ( \frac{1}{\rho} \nabla{u} )  - \frac{\sigma^2}{\rho c^2}u\right)\left(\Psi(r,\widehat{x})\right) 
%	= \left(\stiffr\otimes\Id + \massr\otimes (-\Delta_{\Gamma}) \right)(u\circ \Psi)(r,\widehat{x})\\
%& \text{with} \quad \stiffr = -r^{1-d}\partial_r(r^{d-1} \frac{1}{\rho} \partial_r)- \frac{\sigma^2}{\rho c^2} \Id\quad\mbox{and} \quad (\massr v)(r) = \frac{a^2}{r^2}v(r).
%\end{aligned}
%\end{align}
Instead of directly working with equation  \cref{eq:helio_original} we follow the literature, see e.g.\  \cite{AHN18,BFP19} and perform a transformation to a Schr\"odinger-type equation with an 
effective potential $v$ by means of the substitution $u = \sqrt{\rho} \tilde{u}$. This yields the equation
\begin{subequations}\label{eqs:schroedinger}
\begin{align}\label{eq:schroedinger}
(-\Delta + v) \tilde{u} = 0 \qquad \mbox{with} \qquad 
v = \rho^{1/2} \Delta ( \rho^{-1/2}) - \sigma^2/c^2.
\end{align} 
Of course, \eqref{eq:schroedinger} also satisfies the separation condition \eqref{eq:separable_bvp2} with 
$\Psi$ given by \eqref{eq:spherical_coo}, 
$\stiffr = -r^{1-d}\partial_r(r^{d-1}  \partial_r)+ v$ 
and $(\massr w)(r) = \frac{a^2}{r^2}w(r)$. Since the absorption coefficient $\gamma = \Im \sigma$ is 
positive, it suffices to impose boundedness as radiation condition. Alternatively, we may just impose 
a homogeneous Neumann condition close to the corona at a distance $\RVALC$ from the center where the 
influence of pressure on the movement of matter becomes negligible and the PDE models \eqref{eq:helio_original} 
and \eqref{eq:schroedinger} break down: 
\begin{align}
&\tilde{u} = u_0 &&\mbox{on }a\mathbb{S}^2\\
&\partial_r\tilde{u} = 0 &&\mbox{on }\RVALC\mathbb{S}^2.
\end{align}
\end{subequations}
A small frequency dependent damping  is added as described in Section 7.3 of \cite{GB17}. 
%Next the DtN approximation will be considered. 
The coupling boundary $\Gamma$ is positioned directly at the solar surface, i.e.\ $a=1.0$ (in terms of the solar radius).

For realistic sound speed and density as given by the VAL-C model, 
the $\dtn$ function admits no analytical expression. Hence, the 
values $\dtn^{\text{VAL-C}}(\lambda)$ have to be computed 
by solving the ODE \cref{eq:ODE_dtn_numbers} with Dirichlet condition at $r=a$ and Neumann condition 
at $r=\RVALC$. However, one can apply the argument of 
\cref{sec:structarg} to show the existence of a natural meromorphic extension of $\dtn^{\text{VAL-C}}$. 
In fact, by standard arguments using compactness of embeddings one can see that the weak formulation of 
\eqref{eqs:schroedinger} leads to a Fredholm operator of index $0$. Moreover, since $\gamma>0$ it can be shown  
to be injective by taking the imaginary part of the variational formulation. 
Therefore, the operator is also surjective, and \eqref{eqs:schroedinger} is well-posed. 
%For a more detailed discussion and numerical results we refer to \cref{ssection:helio}.

%The reference  DtN numbers for the VAL-C model are obtained by solving the ODE \eqref{eq:ODE_dtn_numbers}
%with boundary condition $u_{\ell}(a) = 1$ and a homogeneous Neumann boundary condition at the end of the %model.
%Then $\dtn^{\text{VAL-C}}(\lambda_{\ell}) = - \partial_{r} u_{\ell}(a)$. 
The values of $\dtn^{\text{VAL-C}}$ at the eigenvalues $\lambda_{\ell}$ 
for a frequency of $7.0$ mHz are shown  in \cref{fig:helio_coeff}.
These values were obtained by solving an ODE for which 
we discretized the interval $[a,\RVALC]$ with one hundred equidistant finite elements of order eight.
The potential well produced by the increase in sound speed apparently leads to a similar behavior as the discontinuous wavespeed considered in \cref{ssection:jump_numexp}. However, due to the huge magnitude of the 
potential extremely large modes (note the scaling of the abscissa) are required to capture the behavior of 
$\dtn^{\mathrm{VAL-C}}$.  
\end{example}

%As detailed in \cref{ssection:helio}, a full finite element discretization of the solar atmosphere would involve a huge amplification of computational effort due to small wavelengths and must be avoided. 
As current state-of-the-art in computational local helioseismology, 
local non-reflecting boundary conditions for a simplified model of the solar atmosphere 
(so-called \emph{atmospheric radiation conditions}) have been proposed and validated in \cite{BC18,FL17}.  
In this simplified model it is assumed that density decays exponentially and the sound speed is constant 
(see \cref{fig:c_rho_VALC}). This simplification (even without further approximations) only 
yields reasonable results if 
the coupling boundary is placed a few hundred kilometers above the photosphere whereas our approach 
offers the possibility to place the 
the coupling boundary on the photosphere. This allows for a significant reduction of the number of {\dof}s. 
Moreover, the spikes in $\dtn^{\mathrm{VAL-C}}$ displayed in \cref{fig:dtn-VALC} are not captured in this simplified model. They seem to be associated with reflections at the solar corona, which are discussed in 
\cite{FL17} as potential reason for certain discrepancies between their simulation results and helioseismic 
observations.

\section{Approximation of $\dtn$ functions}
\label{sec:approx_DtN}

The idea of learned infinite elements is to optimize the parameters $\dstiffr$ and $\dmassr$ in the discrete DtN function $ \ddtn$, cf. \cref{sec:derivation},
representing the approximate 
transparent boundary conditions to achieve an optimal fit of the true DtN function $ \dtn$. A corresponding minimization problem is formulated in \cref{ssection:minimization_problem}. A more efficient reduced ansatz for the matrices  $\dstiffr$ and $\dmassr$ is introduced in \cref{ssection:redans} and its use is demonstrated  in \cref{ssection:diag_and_poles}.
%it is shown that the number of parameters in the ansatz for $ \ddtn$ can be significantly reduced. 
This reduction exposes $ \ddtn$ as a rational approximation of $\dtn$ with simple poles. Convergence of this approximation is discussed in \cref{ssection:convergence}.

\subsection{The minimization problem}\label{ssection:minimization_problem}

We are now in a position to define misfit function and minimization problem for the task of finding optimal $\dstiffr, \dmassr$.
\begin{definition}
Let $\lambda_{\ell}$ be the continuous eigenvalues of the Laplace-Beltrami operator on $\Gamma$.
Denote by  $ \dtn(\lambda_{\ell})$  the DtN numbers obtained from solving the ordinary differential equations \cref{eq:ODE_dtn_numbers}. 
Further, let $ \ddtn( \lambda_{\ell})$ be the approximate DtN numbers of the learned infinite elements 
as in equation \cref{eq:dtn_fct_discrete}. For positive weights $w_{\ell}$ we define the misfit function
\begin{equation}\label{eq:misfit}
J(\dstiffr,\dmassr) = \frac{1}{2} \sum\nolimits_{\ell}{   \vert w_{\ell}(  \dtn(\lambda_{\ell}) -   \ddtn( \lambda_{\ell}   ) )   \vert^2}.
\end{equation}
The minimization problem is to find $\dstiffr,\dmassr \in \mathbb{C}^{(N+1)\times (N+1)}$ so that
\begin{equation}\label{eq:minimization problem}
\dstiffr,\dmassr \in \argmin_{\dstiffr,\dmassr \in \mathbb{C}^{(N+1)\times (N+1)}} J \left( \dstiffr,\dmassr \right). 
\end{equation}
\end{definition}

Some remarks about this minimization problem are given below.
\begin{itemize}
\item In applications the transparent boundary condition is  usually combined with a finite element discretization of an interior problem. 
 The weights $w_{\ell}$ in the misfit function \cref{eq:misfit} determine how accurate the corresponding 
 $\dtn$-function will be fitted for which modes. As solutions to PDEs with constant or analytic coefficients 
 are analytic, the coefficients of these modes usually decay exponentially. Therefore, also the weights 
 should asymptotically decay exponentially. 
 However, to achieve optimal results the profile pertaining the propagating modes should be 
 chosen with some care.
 Appropriate choices are discussed when considering example problems in  \cref{sec:numexp}.
\item The choice of the objective function \cref{eq:misfit} is motivated by our demand to obtain an optimal approximation of the DtN map in a reasonable sense. 
Approaches which are to some extent related have been pursued before. 
For example, in \cite{G98} high-order \emph{local} non-reflecting boundary conditions are constructed based on the requirement to provide $L^2$ best approximations of the DtN map for smooth functions which can be realized via auxiliary variables \cite{G02}.
Here we do not confine ourselves to local approximations, and 
instead of using the $L^2$ norm, the approach considered here is based on the natural norm of $\DtN$ as a continuous linear operator from $H^s(\Gamma), s \geq 1/2$ to $H^{-1/2}(\Gamma)$. 
The attentive reader may have noticed that this is not completely true since the operator norm would be bounded by $\sup_{\ell} \vert  \dtn(\lambda_{\ell}) -   \ddtn( \lambda_{\ell}   ) \vert $  while a weighted $\ell^2$ norm is considered in the objective function. 
This choice is made for ease of implementation and justified for smooth solutions for which it suffices to control a finite number of modes so that equivalence of norms holds true.
%For example, the conditions introduced in are constructed from the requirement to provide best approximations of the DtN map applied to smooth functions in the $L^2$-norm and may be implemented by introducing auxiliary variables 
\item Note that the misfit function is formulated in terms of the continuous eigenvalues $\lambda_{\ell}$. The discrete eigenvalues and eigenvectors are not needed. 
  % only, even though the discrete eigenvalue relation in \cref{prop:1} involves the discrete eigenvalues $ \underline{\lambda}_{\ell}$. However, this is justified since for a reasonably accurate finite element discretization continuous and discrete eigenvalues agree very well for small $\ell$ while large $\ell$ are insignificant due to 
% the rapidly decaying weights. Note that this means that we do not need to solve the corresponding discrete eigenvalue problem. For $\Gamma \coloneqq a \mathbb{S}^{d-1}$ the continuous eigenvalues are known analytically. For example, for $d=2$ the eigenvalues are 
% $\lambda_{\ell}  = (\ell /a)^2 $.
% \thnote{Diesen Punkt evt. weglassen. Auch wenn $ \underline{\lambda}_{\ell}$ bekannt ist, ist mir nicht klar, was die beste Fit-Funktion ist. In beiden Funktionen $\underline{\lambda}_{\ell}$ oder nur in $\ddtn$? }
\item We had good success with using trust region methods for solving the nonlinear least squares problem \cref{eq:minimization problem}.
Specifically, we utilize the Levenberg-Marquardt algorithm as implemented in the \texttt{Ceres-solver} \cite{ceres-solver}.
Implementational details are described in \cref{section:impl_details}. 
\end{itemize}

% \section{Approximating the DtN function}

\subsection{Reduced ansatz}\label{ssection:redans}
For dense matrices $\dstiffr$ and $\dmassr$  the number of nonzero entries of the tensor product system $\dstiffr \otimes \massFEM  + \dmassr \otimes \stiffFEM$ grows quadratically with $N$. 
To improve efficiency and to reduce the number of free parameters it is therefore desirable to sparsify $\dstiffr$ and $\dmassr$. 
The next step is to identify which entries are dispensable.
To this end, let us assume that
\begin{equation}\label{eq:diagonalization_assumption}
	(\dmassr_{E E})^{-1}  \dstiffr_{E E} \; \text{is diagonalizable,} 
\end{equation}
that is, there exists a diagonal matrix $D$ and an invertible matrix $P$ such that   $(\dmassr_{E E})^{-1}  \dstiffr_{E E} = P D P^{-1}$ 
(see also \cref{rem:diagonalizable}). Then 
\begin{equation*}
(  \dstiffr_{E E} + \lambda \dmassr_{E E}  )^{-1}  =  P ( D + \lambda I )^{-1} (\dmassr_{E E}  P)^{-1}
\end{equation*}
holds.
Inserting this identity into \cref{eq:dtn_fct_discrete} yields that the reduced ansatz
\begin{equation*}
\dstiffr = 
 \begin{bmatrix}
	 \dstiffr_{00}  & \dstiffr_{01}  &   \cdots  &  \cdots &  \dstiffr_{0  N }  \\
    \dstiffr_{10}     &   \ddots       &   &  \mathbf{0} &  \\
    \vdots &  & \ddots   &  &  \\
        \vdots & \mathbf{0}  &   & \ddots  &  \\
    \dstiffr_{ N 0}  &  &     &  & \dstiffr_{ N N}    \\
    \end{bmatrix} , 
    \quad 
   \dmassr = 
 \begin{bmatrix}
   \dmassr_{00}  & \dmassr_{01} &   \cdots  &  \cdots &  \dmassr_{0 N} \\
    1      &   1     &   &  \mathbf{0} &  \\
    \vdots &  & \ddots   &  &  \\
        \vdots & \mathbf{0}  &   & \ddots  &  \\
   1  &  &     &  & 1   \\
    \end{bmatrix}
     \end{equation*}
 leads to the same $\ddtn$ function as the approach with dense matrices $\dstiffr$ and $\dmassr$.
 The resulting system $\dstiffr \otimes \massFEM  + \dmassr \otimes \stiffFEM$ will be much sparser though, 
 and the number of free parameters grows only linearly with $N$.  \par 

In the reduced ansatz $\ddtn$ is a rational function which is asymptotically linear and has a finite number of simple poles. 
This can be seen from the formula
 \begin{equation}\label{eq:zeta_learned_diag}
 \ddtn(\lambda)  = \dstiffr_{00} + \lambda \dmassr_{00} - \sum\limits_{j=1}^{N}{ \frac{(\dstiffr_{0j} + \lambda \dmassr_{0j}) ( \dstiffr_{j0} + \lambda  )  }{  \dstiffr_{jj} + \lambda  }     }. 
 \end{equation}
 The second term is a rational function with simple poles at  
 \begin{equation}\label{eq:learned_poles}
 \lambda_{j}^{*}(N) = -\dstiffr_{j,j}, \; j= 1, \ldots,N. 
 \end{equation}
This means that the diagonal entries of $-\dstiffr$ correspond to simple poles of the rational approximation. 
Thus, by optimizing the matrices we optimize the poles at the same time. 

\begin{remark}[Assumption \eqref{eq:diagonalization_assumption}]\label{rem:diagonalizable}
Poles with higher multiplicities are excluded by assumption \cref{eq:diagonalization_assumption} and are consequently not covered by the reduced ansatz.
If this assumption breaks down, then the ansatz using dense matrices is more general as it would allow for higher multiplicites and may lead to a better approximation of $\dtn$ by $\ddtn$. However, below, in \cref{ssection:convergence} we show that the reduced ansatz is still sufficient to guarantee exponential convergence rates on finite intervals. 
Additionally, all $\dtn$ functions which we investigated so far had only simple poles which renders an ansatz for $\ddtn$ of the form \cref{eq:zeta_learned_diag} very natural.
\end{remark}

\subsection{Pole structure}\label{ssection:diag_and_poles}
To justify the previously introduced reduction the following numerical experiment is considered. 
The analytic DtN numbers for the Helmholtz equation in \cref{ex:Helmholtz_ball} with $d = 2$ are obtained by evaluating the function given in \cref{eq:dtn_fct_helmholtz2D} 
at $ \lambda  = \lambda_{\ell} = (\ell / a)^2$. 
These numbers $ \dtn(\lambda_{\ell})$ with $a=1$ and $\wavenr =16$ are used as reference values for solving the minimization
problem \cref{eq:misfit}-\cref{eq:minimization problem}. 
If the diagonalization is justified then the reduced ansatz  should lead the to same 
results as the dense approach. 
For this test case exponentially decaying weights $w_{\ell} \sim \exp(- 2 \ell  / 3 )$ are chosen. 
% \par 
%
In \cref{fig:learning_curves} the relative errors 
\begin{equation*}
  \frac{\vert \dtn(\lambda_{\ell}) -  \ddtn(\lambda_{\ell}) \vert }{ \vert \dtn(\lambda_{\ell}) \vert} 
\end{equation*}
for the dense and reduced ansatz are compared for $N \in \{0,2,4,6\}$. 
Apparently, both approaches yield the same results. 
Since the learned DtN based on the reduced ansatz offers the same accuracy as the dense approach and results additionally into a sparser exterior system it is in practice the method 
of choice. \par 

% In the reduced ansatz the continuous DtN function $\dtn(\lambda)$ is approximated 
% by a sum of poles. 
% This can be seen from the formula
%  \begin{equation}\label{eq:zeta_learned_diag2}
%  \ddtn(\lambda)  = \dstiffr_{00} + \lambda \dmassr_{00} - \sum\limits_{j=1}^{N -1}{ \frac{(\dstiffr_{0j} + \lambda \dmassr_{0j}) ( \dstiffr_{j0} + \lambda  )  }{  \dstiffr_{jj} + \lambda  }     }. 
%  \end{equation}
%  The second term is a rational function with poles at  $\lambda_{j}^{*}(N) = -\dstiffr_{j,j} , j= 1, \ldots,N -1$. 
 The poles for the numerical experiment from the previous paragraph are shown 
 in \cref{fig:poles-plot}. They are located in the quadrant $\{z \in \mathbb{C} \mid \Re(z) > 0, \Im(z) > 0 \}$. 
 Additionally, the exact poles of the meromorphic extension of $\dtn^{\mathrm{hom}}$ are displayed. 
 \cref{fig:poles-plot} suggests that some of the learned poles seem to converge to the exact poles. 
 In particular, the analytic and learned poles with the smallest 
 imaginary part are placed at the same location for $N \in \{4,5,6\}$. 
 
 We point out that simply using the $N$ poles of $\dtn^{\mathrm{hom}}$ with smallest imaginary parts 
 to define $\ddtn$ yields a much worse approximation in this example. 
%To quantify this further,  \cref{fig:pole-tracking-conv} shows the absolute difference between the two analytic poles with the smallest imaginary part and the two learned poles with the smallest imaginary  part for increasing $N$.
%A clear convergence is observed.
 
% \jpnote{In the supplement I have also included a table that displays the performance of the minimization routine (\# iterations, cost, gradient and solution time). Maybe we could discuss whether such a table should appear in the final manuscript and where (main article or supplement). An argument for including it would be to demonstrate that the time required to solve the minimization problem is moderate. }
% \clnote{
%   at least refer to performance table that is (then) put into the suppl. mat.
% }

%\begin{figure}[htbp]
%\centering
%\includegraphics[scale=1]{Figures/learning-curves-comp.pdf}

%\caption{Comparison of the relative error $  \vert \dtn(\lambda_{\ell}) -  \ddtn(\lambda_{\ell}) \vert / \vert\dtn(\lambda_{\ell}) \vert $ in terms of $N$ for solving the minimization problem for the analytic DtN numbers from (\cref{eq:dtn_fct_helmholtz2D}). The circles `o' display the results for the dense ansatz of $\dstiffr$ and $\dmassr$ while the crosses `x' belong to the reduced ansatz. Note that all crosses lie perfectly inside the circles.  For better illustration only results for even $\ell$ are shown.  }
%\label{fig:learning_curves}
%\end{figure}

 \begin{figure}[htbp]
\centering
\subfloat[ Learning ]{ \label{fig:learning_curves} 
\includegraphics[scale=0.7]{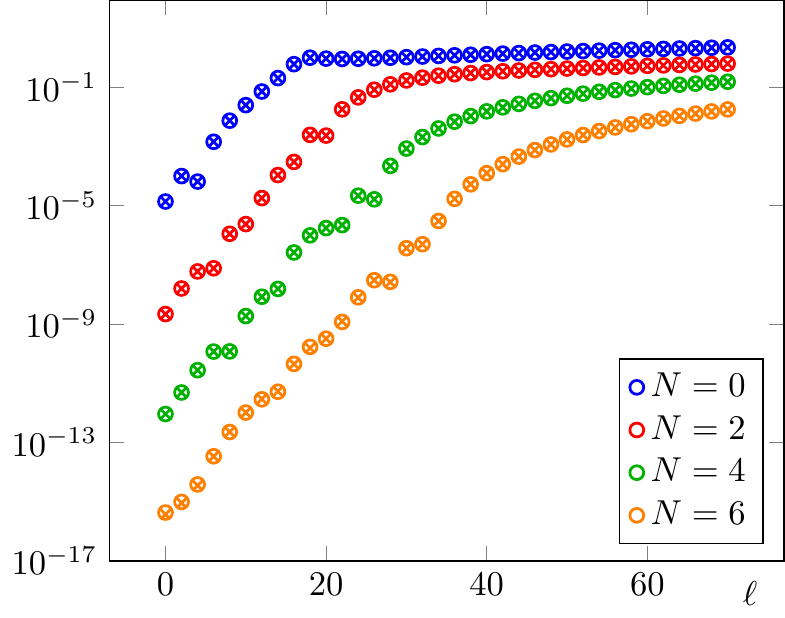}
}
\subfloat[ Poles ]{ \label{fig:poles-plot} 
\includegraphics[scale=0.95]{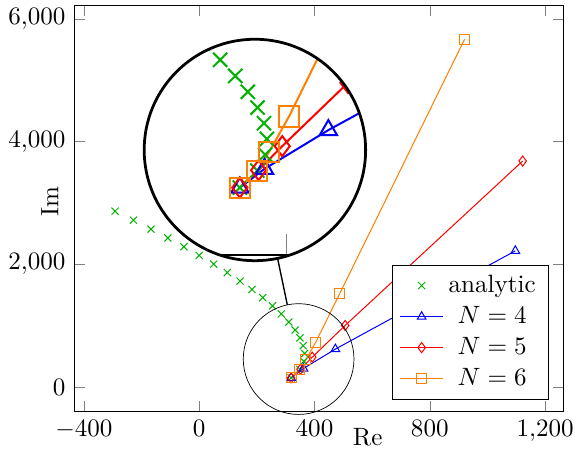}
}
\caption{ Left: Comparison of the relative error $  \vert \dtn(\lambda_{\ell}) -  \ddtn(\lambda_{\ell}) \vert / \vert\dtn(\lambda_{\ell}) \vert $ in terms of $N$ for solving the minimization problem for the analytic DtN numbers from \cref{eq:dtn_fct_helmholtz2D}. The circles `o' display the results for the dense ansatz of $\dstiffr$ and $\dmassr$ while the crosses `x' belong to the reduced ansatz. Note that all crosses lie perfectly inside the circles.  For better illustration only results for even $\ell$ are shown. Right: Poles for the diagonal ansatz when learning the DtN numbers from \cref{eq:dtn_fct_helmholtz2D}.    
}
\label{fig:learning-poles}
\end{figure}

\subsection{Convergence}\label{ssection:convergence}

The results shown in \cref{fig:learning_curves} suggest that the approximation quality of $ \ddtn$ improves exponentially with the number of infinite element \dofs~$N$.
A rigorous proof of this observation could serve as the basis of an error analysis for finite element simulations with learned infinite elements as transparent boundary condition.
However, note that no uniform rates in $\lambda$ for  $  \sup_{\lambda \geq 0} \vert \dtn(\lambda)  -  \ddtn(\lambda)  \vert $ can be expected due to the nontrivial asymptotic behavior of $\dtn(\lambda)$ at infinity. 
Hence, smoothness of the solution has to be exploited in order to reduce the approximation problem to finite intervals. 
Here, exponential convergence rates can be established as outlined below. \par 

The approximation space of learned infinite elements contains in particular functions of the form 
\begin{equation*}
 r_{n}(\lambda) = \sum\limits_{j=1}^{n} \frac{1}{\lambda - a_{j}}, \quad \{ a_{j} \} \subset \mathbb{C},
\end{equation*}
which are known in the literature as simple partial fractions (SPFs). 
It seems that the approximation properties of SPFs have been investigated mostly in the Russian literature. 
A recent overview is given in \cite{DKC18}. 
For our purpose the following result derived in \cite{K01}  is of major importance.
Let $K$ be a compact and rectifiable set of the complex plane and $f$ be analytic at its interior points and continuous on the closure. 
Denote by $SR_{n} $  the set of SPFs  with degree at most $n$, 
the supremum norm on $K$ by $\|\cdot\|_{C(K)}$, 
and the best uniform approximation on $K$ by
\begin{equation*}
\rho_{n}(f,K) \coloneqq \inf \{ \| f - r_{n} \|_{C(K)} \mid r_{n} \in SR_{n} \}.
\end{equation*}
Denote the same quantity for  polynomials $\mathbb{P}_{\mathbb{C}}^{n} $ of degree at most $n$ by
\begin{equation*}
E_{n}(f,K) \coloneqq \inf \{ \| f - r_{n} \|_{C(K)} \mid r_{n} \in \mathbb{P}_{\mathbb{C}}^{n} \}.
\end{equation*}
Then the weak equivalence 
\begin{equation}\label{eq:weak_equiv_poly_approx}
\rho_{n+1}(f,K) \simeq E_{n}\left(fe^{\theta(f;b,\cdot)},K \right)
\end{equation}
between the best approximation error with polynomials and SPFs holds.
Here, for a fixed  $b \in K$ the function $\theta(f;b,z) = \int\nolimits_{b}^{z} f(t) ~dt$ is an antiderivative 
for $f$ taken along a path contained in $K$ with length smaller than some $\tilde{d}>0$. 
The notation $ \simeq$ means bounded from above and below by a constant which depends (exponentially) 
on $\Vert f \Vert_{C(K)}$ and  $\tilde{d}$.  \par 

It remains to specialize this result to our setting and combine it with classical approximation theory for polynomials.
To this end, it is important to note that $\lambda \mapsto \dtn(\lambda) $ is smooth on $(0,\infty)$, cf.\ \cref{sec:structarg}.
% For the Helmholtz equation this follows immediately from the expression \cref{eq:dtn_fct_helmholtz2D} since $\nu \mapsto H^{(1)}_{ \nu }(z)$ is an entire function of $\nu$ for any $z \neq 0$ and 
% $H^{(1)}_{ \nu  }(\omega a) \neq 0$  for $0< \omega a <\infty  $ and $\nu \geq 0$ holds. 
% The theory of ODEs with an essential singularity at infinity (see chapter 7 of \cite{O97}) even implies that $\dtn$ extends as a meromorphic function to the complex 
% plane under much more general conditions on the functions $\{ \varphi_j \}$ in \eqref{eq:separable_bvp}.
This allows to invoke standard results (e.g.\ Jackson's theorem) to bound the best approximation error with polynomials appearing in \cref{eq:weak_equiv_poly_approx}.
Even the use of  approximation results for analytic functions on subsets of the complex plane is possible (see e.g. chapter 4.5 of \cite{W60}) thanks to the analytic extensions of $\dtn$ as discussed in 
\cref{sec:dtnmaps}.
Exponential convergence in $N$ on finite intervals follows. The derivation of explicit error estimates is planned to be the subject of another publication. \par 

\section{Tensor product discretizations and infinite elements}
\label{section:tensor_product_discr}
In this section we will relate the tensor product ansatz \eqref{eq:ansatz}, which is the starting point for our learned infinite elements, 
to existing successful transparent boundary condition, in particular infinite elements and demonstrate that the algebraic structure of \eqref{eq:ansatz} is shared by these methods.

\subsection{Tensor product infinite elements}
Infinite elements can be regarded as an extension of finite elements to infinite domains. 
%%%%%
Recall that a finite element is a triple $(T,\Pi,\Sigma)$ consisting of a closed bounded domain $T$ with sufficiently regular boundary, a finite dimensional space $\Pi$ of functions defined on $T$, and a set $\Sigma$ of linearly independent functionals with 
$\#(\Sigma)=\dim \Pi$. 
Infinite elements relax the first condition by allowing infinite domains.  
We say that an infinite element $E$ is the tensor product of a finite element $(T_{\Gamma},\Pi_{\Gamma},\Sigma_{\Gamma})$ on $\Gamma$ 
and another finite element $(T_r,\Pi_r,\Sigma_r)$ if $E=(T_r\times T_{\Gamma}, \Pi_r\otimes \Pi_{\Gamma},\Sigma_r\otimes \Sigma_\Gamma)$.   

If tensor product infinite elements of the same type, say with $\Pi_r=\operatorname{span}\{g_0,\dots, g_{N}\}$ are attached to each patch 
of the finite element grid on the coupling boundary $\Gamma$, 
then a basis of global shape functions corresponding to the exterior domain is also of tensor product structure:
\begin{equation}\label{eq:tensor_basis} 
\psi_{\alpha}(r, \hat{x} ) = g_{\mu}(r)\FEshape_{i}( \hat{x} ) , \quad \alpha  = (i,\mu)\in\{0,\dots,N\}\times\{1,\dots,n_{\Gamma}\}
\end{equation}
Here as in \cref{sec:derivation} $\FEshape_{i}$ are finite element basis functions on the coupling interface $\Gamma$. 
To mold the arising linear systems into a convenient form a special partition of the \dofs~is employed. 
In the following,  $\left(u_{\alpha} \right) $ is considered as a block-vector where each block contains  all unknowns on the coupling interface labeled by $i=0,1,..,n_{\Gamma}$ for fixed $\mu$:
\begin{equation}\label{eq:dof_partition}
  \left(\uh_{\alpha}\right) = \Big(  \overbrace{(\uh_{\cdot,\mu=0})}^{ \Gamma \text{-\dof}}, \overbrace{(\uh_{\cdot,\mu=1}), \ldots, (\uh_{\cdot,\mu=N})}^{E \text{-\dof}}  \Big)^{\top} 
= (\uh_{\Gamma},\uh_{E})^{\top}. 
\end{equation}
Additionally, in order to obtain a convenient form of  $\dDtN$, it is assumed that the $\mu = 0$-\dof~of the infinite elements is associated with the coupling interface (label $\Gamma$) while the $\mu >0$-\dofs~belong to the exterior (label $E$).

For a variational formulation reflecting the structure \eqref{eq:separable_bvp2} of the exterior PDE, such an ordering of the \dof~leads 
to the tensor product structure \eqref{eq:ansatz}. Rather than formalizing this statement, we will 
illustrate it at the example of perfectly matched layers, which will also serve as a reference in our numerical examples. 
We emphasize that perfectly matched layers are 
typically not derived as infinite elements, but tensor product discretizations of perfectly matched layers may be interpreted as 
infinite elements as discussed below. 

\subsection{Perfectly Matched Layers}\label{sec:TP_PML}
Perfectly Matched Layers (PML) as first introduced in \cite{B94} implement artificial absorbing layers next to the boundary of the exterior domain and truncate the exterior domain rendering it bounded.
In our setting of a separable geometry (cf.\ also \cite{CM98} for the PML formulation in curvilinear coordinates)  
the crucial idea can be described as 
stretching of the $r$ coordinate into the complex plane by means of a transformation
\begin{equation}\label{eq:complex_radius}
\tilde{r}(r) =  r + \int\nolimits_{a}^{r} i \sigma(t) \mathrm{d}t,
\end{equation}
where $\sigma$ is a positive function describing unisotropic absorption in propagation direction.  
This requires the coefficients involved in $\stiffr$ and $\massr$ to have analytic extensions. 
%Note that the complex stretching only affects the operators $\stiffr$ and $\massr$. 
In the new $\tilde{r}$ coordinates radiating solutions decay exponentially whereas incoming solutions grow exponentially. 
Therefore, the computational domain can be truncated by restricting $\tilde{r}$ to an interval $[a,R]$. 

Assume that $|\det\Psi(r,\hat{x})| = r^m$, $(\calB u)(r)= \varphi_1(r) u(r)$, and 
$\calA$ is in divergence form 
$(\calA u)(r) = r^{-m}\partial_r (r^m\phi_3(r)\partial_r u(r)) + \varphi_2(r)u(r)$ and  
with analytic functions $\varphi_j$. This covers Examples \ref{ex:Helmholtz_ball} and \ref{ex:Helmholtz_waveguide}. 
Set $\tilde{u}(r,\hat{x}):= u(\Psi(\tilde{r}(r),\hat{x}))$ and analogously for analytic, rapidly decaying test functions $v$. 
Since $(\partial_r u)(\Psi(\tilde{r}(r),\hat{x})) = \partial_r \tilde{u}(r,\hat{x}) / \tilde{r}'(r)$ by the chain rule, a contour deformation 
in the standard variational formulation of \eqref{eq:bvp_ext} and truncation of the interval $[a,\infty)$ to $[a,\eta]$ leads to the variational formulation 
\begin{align*}
\int_{\Gamma}\int_a^{\eta} \left[\tilde{\varphi}_3\frac{\partial_r\tilde{u}\partial_r\tilde{v}}{(\tilde{r}')^2} 
+ \tilde{\varphi}_2 \tilde{u}\tilde{v} + \tilde{\varphi}_1\nabla_{\hat{x}}\tilde{u}\cdot \nabla_{\hat{x}}\tilde{v}\right] \tilde{r}^m \tilde{r}' \,\mathrm{d}r \,\mathrm{d}\hat{x} =0
\end{align*}
with $\tilde{\varphi}_j:=\varphi\circ\tilde{r}$, the unknown function $\tilde{u}$ satisfying $\tilde{u}(a,\cdot)=u_0$ and test functions
$\tilde{v}$ satisfying $\tilde{v}(a,\cdot)=0$. 

It is natural -- although not necessary -- to use a tensor product infinite elements for the discretization of the 
PML layer $[a,\eta]\times \Gamma$. In this case one may formally combine all one-dimensional finite elements in propagation direction $r$ to 
one infinite element. This shows that tensor product PMLs may also be interpreted as infinite elements. 
Expanding $\tilde{u}$ and $\tilde{v}$ in the basis \eqref{eq:tensor_basis} leads 
to a system of equations of the form \eqref{eq:ansatz} with the matrices 
 \begin{equation}\label{eq:Ls_integrals}
 \dstiffr_{\mu \nu}  = \int\nolimits_{a}^{\eta}{  \left[\tilde{\varphi}_3 \frac{ \partial_r g_{\mu}\, \partial_r g_{\nu}}{(\tilde{r}')^2}  
 + \tilde{\varphi}_{2} g_{\mu} g_{\nu}  \right]  \tilde{r}^{m}\tilde{r'} \, \text{d} r  } \quad\text{and}\quad    
 \dmassr_{\mu \nu}  = \int\nolimits_{a}^{\eta}{  \tilde{\varphi}_{1}  g_{\mu} g_{\nu} \, \tilde{r}^{m}\tilde{r'}\, \text{d} r  }.
 \end{equation}

% It is also easy to derive the relation between $\underline{ \DtN }$ and the Schur complement as given in  \cref{def:discreteDtN}.
%To this end, consider the equation
%\begin{equation*}
%- a^{-2} r^{1-d} \partial_r ( r^{d-1} \varphi_1(r) \partial_r u ) - \varphi_{2}(r) \Delta_{\Gamma}u + \varphi_{3}(r)u = 0,     \text{ in } [1,\infty) \times \Gamma,
% \end{equation*}
% Multiplying by a test function $v$, integrating over $r > 1$ with respect to the measure $ a r^{d-1} \text{d}r \text{d} \hat{x}$ and using the definition of the DtN operator, cf. \cref{sec:dtnmaps}, yields
% \begin{alignat*}{3}
%\begin{aligned}
%0 &=   \int\limits_{1}^{\eta}  \int\limits_{\Gamma}  \left( - a^{-1} r^{1-d} \partial_r ( r^{d-1} \varphi_1(r) \partial_r u )  v  
%                            + a \varphi_{2}(r) \nabla_{\Gamma} u \nabla_{\Gamma} v 
%+  a  \varphi_{3}(r)  u v \right) r^{d-1} \text{d}r \text{d} \hat{x}   \\ 
%& = a(u,v) + \int\limits_{ \Gamma} \left[  \frac{1}{a} r^{d-1} \varphi_1(r)   \partial_r u  \right] \vert_{r=1}  v(1,\cdot)   \text{d} \hat{x} = a(u,v) -  \int\limits_{ \Gamma}  ( \underline{\DtN } u )  v(1,\cdot)  \text{d} \hat{x}.
%\end{aligned}  
%\end{alignat*} 
%Hence $F_{\Gamma}  \coloneqq   \massFEM \underline{ \DtN }( u_{\Gamma})$ solves the system 
%\begin{equation*}
%  \begin{bmatrix}
%  \dL_{\Gamma \Gamma}  & \dL_{\Gamma E}  \\
%  \dL_{ E \Gamma} & \dL_{E E}  \\
%\end{bmatrix} 
%  \begin{bmatrix}
%   u_{\Gamma}  \\
%     u_{E}   \\
%\end{bmatrix} 
%= 
%  \begin{bmatrix}
%   F_{\Gamma}  \\
%     0   \\
%\end{bmatrix} .
%\end{equation*}
%Similar as above one then arrives at \eqref{def:discreteDtN}.

\subsection{Other infinite elements}\label{sec:otherIE}
Typically, there is no straightforward variational formulation of time harmonic wave equations in exterior domains. 
The reasons are that neither the solutions themselves nor their gradients are square integrable and that the radiation condition 
is difficult to incorporate. Therefore, whereas finite elements specify a generic 
finite dimensional function space for a given mesh, infinite elements are usually intimately linked to  a specific PDE and/or a type of variational formulation.  
In classical infinite elements for the Helmholtz equation, shape functions are chosen according to  asymptotic expansions of the exterior solution, and solution and test function spaces are equipped with 
different weights (cf.\ \cite{DG:98}). As discussed above, PMLs use standard shape functions, but 
a transformation of the exterior PDE or its variational formulation. 
% \subsubsection{Classical infinite element method}
%In classical infinite element methods, cf. \cite{U73,ZB76,AU77,B77}, basis functions in radial direction are directly defined, for instance based on a priori knowledge of the solution or on basic approximation properties. Given the basis functions, a Galerkin discretization yields the matrices  $\dstiffr$ and $\dmassr$.
% \subsubsection{Hardy space infinite element method}
In a similar vein,  Hardy space infinite elements (HSIE) \cite{HN09} are based on a transformation 
of the PDE to a Hardy space in the propagation direction in order to incorporate the radiation condition. 
The matrices $\dstiffr$ and $\dmassr$ are obtained by a Galerkin approximation of the transformed 
variational formulation with respect to a basis of the Hardy space. 

\subsection{Learned infinite elements}
\begin{figure}[htbp]
  \centering
  \vspace*{-0.2cm}
\includegraphics[scale=0.5]{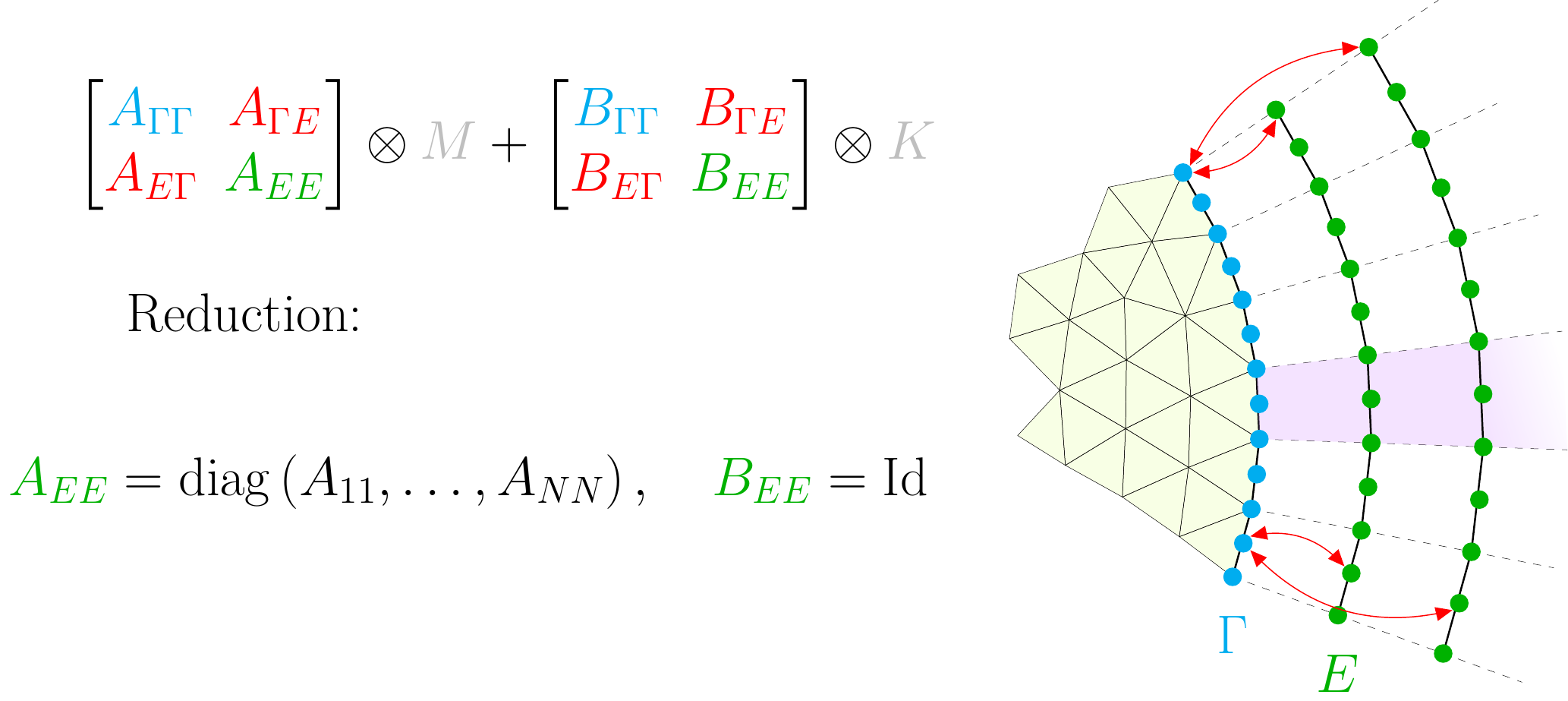}
  \vspace*{-0.2cm}
\caption{Schematic illustration of learned infinite elements. The violet region indicates one such infinite element. 
The first infinite element \dofs~lie on the transparent boundary $\Gamma$ and are associated with the matrices $\dstiffr_{\Gamma \Gamma}$ and $\dmassr_{\Gamma \Gamma}$.
	Additional infinite element \dofs~in the exterior $E$ are tensorized with copies of the FEM discretization of $(\text{Id}_{\Gamma},-\Delta_{\Gamma})$ on  $\Gamma$. The reduction step of \cref{ssection:diag_and_poles} decouples exterior \dofs~(of different rings) from each other. The exterior $E$ couples to $\Gamma$ (and thereby to $\Omega_{\text{int}}$) by means of  $\dstiffr_{\Gamma E }, \dmassr_{\Gamma E },$ and  $\dstiffr_{E \Gamma}, \dmassr_{E \Gamma}$.    }
\label{fig:dofs_sketch}
\end{figure}
In contrast to classical infinite elements where the matrices $\dstiffr$ and $\dmassr$ in \eqref{eq:ansatz} are computed from 
some discretization of the (possibly transformed) exterior PDE,  
with learned infinite elements we only access the exterior PDE via the associated Dirichlet-to-Neumann map $\DtN$. 
Since the matrices $\dstiffr$ and $\dmassr$ will be learned from $\DtN$,  the choice of the element 
$(T_r,\Pi_r,\Sigma_r)$ is irrelevant, only the dimension of $\Pi_r$ counts. We will simply choose $T_r$ as discrete set 
$T_N:= \{0,1,\dots,N\}$. With $\Pi_N:=\mathbb{C}^{T_N}$ and the canonical dual basis 
$\Sigma_N:=\{F_0,\dots, F_N\}$ given by 
$F_j(g):=g(j)$ for $g\in \Pi_N$ and $j\in T_N$, tensor product infinite elements of order $N\in \mathbb{N}\cup\{0\}$ are defined by 
\[
\left(T_N\times T_{\Gamma},\Pi_N\otimes \Pi_{\Gamma},\Sigma_N\otimes \Sigma_{\Gamma}\right)
\]
for finite elements $(T_{\Gamma},\Pi_{\Gamma},\Sigma_{\Gamma})$ on $\Gamma$. This is illustrated in \cref{fig:dofs_sketch}. 
Before the learning step discussed in \cref{sec:approx_DtN} they are more generic, and after the learning step they are more 
specific than classical inifinite elements. 
Moreover, as learned infinite elements are based on the same underlying algebraic structure \cref{eq:ansatz} as many other transparent boundary conditions like tensor-product PMLs well-established methods for solving the linear systems can be applied.
%
%\begin{lemma}\label{lemma:system matrix}
% The system matrix $K_{\alpha \beta} = a(\psi_{\beta} ,\psi_{\alpha} ) $ for  $\alpha  = (i,\mu)$ and  $ \beta  = (j,\nu)$ is given by 
 %\begin{equation*}
% K_{\alpha \beta}  = L_{\mu \nu}^{(1)} M_{i j}  +  L_{\mu \nu}^{(2)} A_{i j},  
 %\end{equation*}
% where $M$ and $A$ are the mass and stiffness matrix on the coupling boundary
 %\begin{equation*}
% M_{i j} = \int\limits_{ \Gamma }{  \FEshape_{i}  \FEshape_{j} \, \text{d} \hat{x}  }, \quad  A_{i j} = \int\limits_{ \Gamma }{   \nabla_{\Gamma} \FEshape_{i}  \nabla_{\Gamma} \FEshape_{j}    \,   \text{d} \hat{x}  },
 %\end{equation*}
%\end{lemma}
%
% To this end, rewrite the separable problem as 
% \begin{equation*}
% \calL u = f, \quad + \text{ radiation condition},
% \end{equation*}
% with $\calL  = \calL^{1}_{r} \cdot \text{Id}_{\Gamma} + \calL_{r}^{2} \cdot \Delta_{\Gamma}$ where
% \begin{equation*}
% \calL^{1}_{r}  = D_r + \varphi_{2}(r), \quad \calL^{2}_{r} = - \varphi_{1}(r) 
% \end{equation*}
% are differential operators acting only on $r$.
% A straightforward discretization is obtained

\section{Numerical results}
\label{sec:numexp}

This section presents a variety of numerical experiments to illustrate the performance of learned infinite elements. 
Firstly, problems with a homogeneous exterior domain are considered to facilitate a comparison with other transparent boundary conditions.
Afterwards, increasingly more inhomogeneous exterior domains are tackled leading finally to a realistic test case for the solar atmosphere.
The reduced ansatz for the learned infinite element matrices is used by default. 
All experiments involving finite elements have been implemented using the library \texttt{Netgen/NGSolve}, see \cite{JS97,JS14}. 
The order of the finite element discretization is denoted by $p$.
A Docker image providing all software, data and instructions to reproduce the experiments is available at \cite{GRO_HLP21}.
For getting a quick first impression of the new method we also offer a couple of notebooks \cite{NotebooksHLP} which can 
be run live in the browser.

\subsection{Scattering of plane wave from a disk}\label{ssection:ex_plane_wave_hom}

Consider a plane wave $g = \exp( i \wavenr x)$ which is incident on a disk with radius $\RScatter=1/2$. For sound-soft scattering, i.e.
% \begin{equation*}
  $
   u  = g \text{ at } r = \RScatter,
  $
% \end{equation*}
the solution of the Helmholtz equation is given by (see the appendix for derivation)
\begin{equation}\label{eq:PlaneWaveHomSolution}
u(r,\varphi) = \frac{ H^{(1)}_{0}(k r) }{ H^{(1)}_{0}(k \RScatter)  }  J_{0}( \wavenr \RScatter)  + \sum\limits_{\ell = 1}^{\infty} \frac{ H^{(1)}_{\ell}(k r) }{ H^{(1)}_{\ell}(k \RScatter)  }  2 i^{\ell} J_{\ell}(\wavenr \RScatter) \cos(\ell \varphi). 
\end{equation}
Here $J_{n}$ and $H_{n}^{(1)}$ denote the Bessel and Hankel function of the first kind of order $n$ respectively. \par 

The problem is discretized on an annulus $ \Omega_{\mathrm{int}} = \{ x \in \mathbb{R}^{2} \mid  \RScatter \leq \Vert x \Vert \leq a \}$ using the Dirichlet boundary condition 
 at $r = \RScatter$ and learned infinite elements at $r = a$. The resolution is increased by raising the polynomial degree while the mesh remains fixed.   
 As discussed in \cref{ssection:minimization_problem} the weights in the misfit function should reflect the regularity of the expected solution.
 This motivates the choice $w_{\ell} \sim \vert H^{(1)}_{\ell}(\wavenr a) / H^{(1)}_{\ell}(\wavenr \tilde{r}(\wavenr)) \vert$,
  where $ \tilde{r}(\wavenr) = \min \{ \RScatter,\RScatter \cdot \wavenr/16 \} \leq \RScatter$ describes the radius of the smallest ball into which the solution is expected to extend analytically into the scatterer.
 \par 
 
 The relative error on $ \Omega_{\mathrm{int}}  $ for increasing number $N$ of infinite element \dofs~is shown in \cref{fig:plane-wave-disc-relerr}. The convergence in $N$ is extremely fast. For $p=6$ and $\wavenr = 16$  the spatial accuracy is reached with only $N=3$ \dofs~as \cref{fig:plane-wave-disc-relerr-p} shows. \cref{fig:plane-wave-disc-relerr-k} demonstrates that the method delivers good results for a wide range of wavenumbers. \par 
 
The performance of most transparent boundary conditions is known to depend on the number of wavelength that fit between scatterer and the coupling boundary. 
The further away the coupling boundary the more the highly oscillatory components of the solution have already decayed when reaching it. 
Hence, it suffices to focus attention only on a couple of dtn numbers associated with slowly propagating modes, i.e.\ to  fit 
 $\vert  \dtn^{\mathrm{hom}}(\lambda_{\ell}) -   \ddtn(  \lambda_{\ell}) \vert  $ for small $\ell$ well, to obtain 
 an accurate solution when $\vert a-\RScatter \vert $ is large.  
 In this regard, the choice $\vert a-\RScatter \vert = 1/2 $ considered for the previous experiment is rather generous. 
Therefore, an additional experiment is performed in which the coupling boundary is moved progressively closer to the scatterer. 
To this end, $\RScatter=1/2$ is fixed and $a$ is decreased so that $\vert a-\RScatter \vert  \in [1/2,1/4,1/8,1/16] $. 
The results for $\wavenr = 16$ are displayed in \cref{fig:plane-wave-coupling-radius}. 
As expected, the convergence rate of the relative error slows down as the coupling boundary approaches the scatterer.
Nevertheless, the convergence is still exponential.
\cref{fig:plane-wave-a-uniformerror} also shows that the weighted uniform error of the dtn modes is proportional to the relative $L^2$-error. 
This provides an a priori estimate on the accuracy of learned IEs in the sense that the error stemming from the DtN approxmation can be measured directly after the minimization problem \cref{eq:minimization problem} has been solved. 

 \begin{figure}[htbp]
\centering
\subfloat[ $\wavenr = 16$ ]{ \label{fig:plane-wave-disc-relerr-p} 
\includegraphics[scale=0.735]{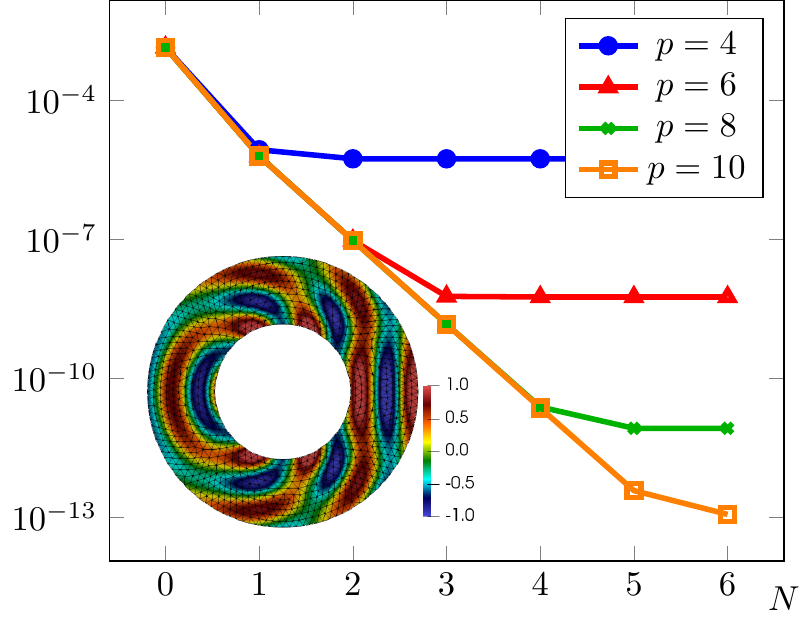}
}
\subfloat[ $ p = 10$ ]{ \label{fig:plane-wave-disc-relerr-k}
\includegraphics[scale=0.735]{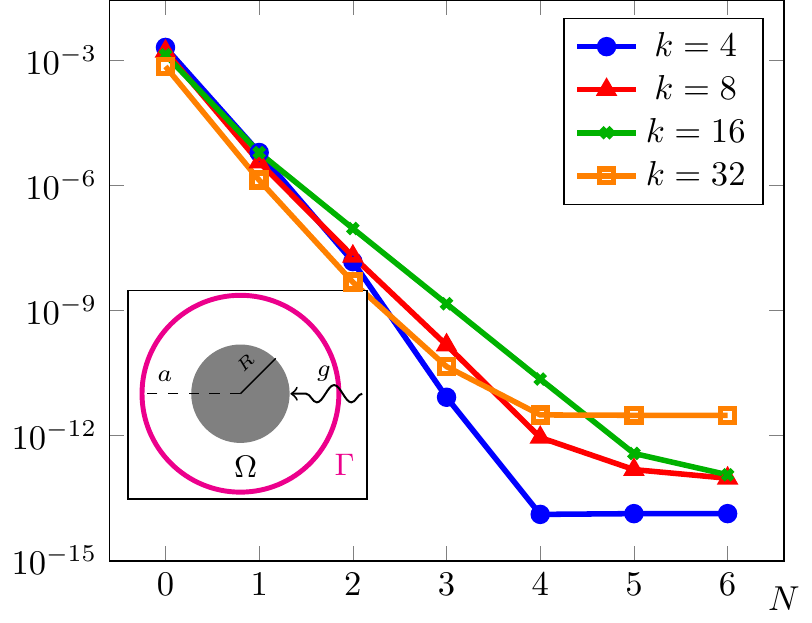}
 }
\caption{ Comparison of relative error $\Vert u-u_h \Vert_{L^2(\Omega_{\mathrm{int}})} /\Vert u \Vert_{L^2(\Omega_{\mathrm{int}})}  $ for the scattering of a plane wave from a disk. Additionally, the real part 
	 of the reference solution for $\wavenr = 16$ is shown in the left figure.}
\label{fig:plane-wave-disc-relerr}
\end{figure}

 \begin{figure}[htbp]
\centering
\subfloat[  $\Vert u-u_h \Vert_{L^2(\Omega_{\mathrm{int}})} /\Vert u \Vert_{L^2(\Omega_{\mathrm{int}})}  $  ]{ \label{fig:plane-wave-a-relerr} 
\includegraphics[scale=0.735]{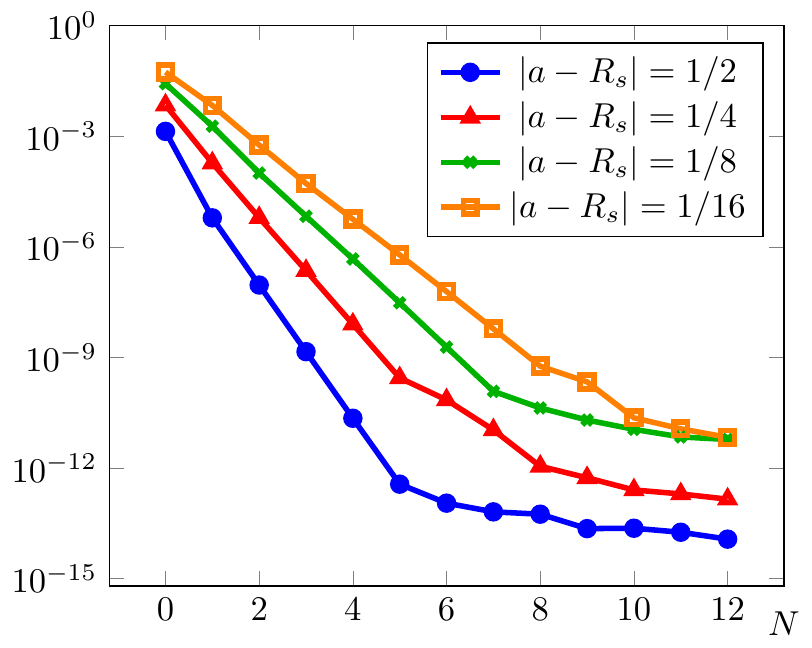}
}
\subfloat[  $  \sup_{\ell}   \vert w_{\ell}(  \dtn^{\mathrm{hom}}(\lambda_{\ell}) -   \ddtn(  \lambda_{\ell} ) ) \vert  $ ]{ \label{fig:plane-wave-a-uniformerror}
\includegraphics[scale=0.735]{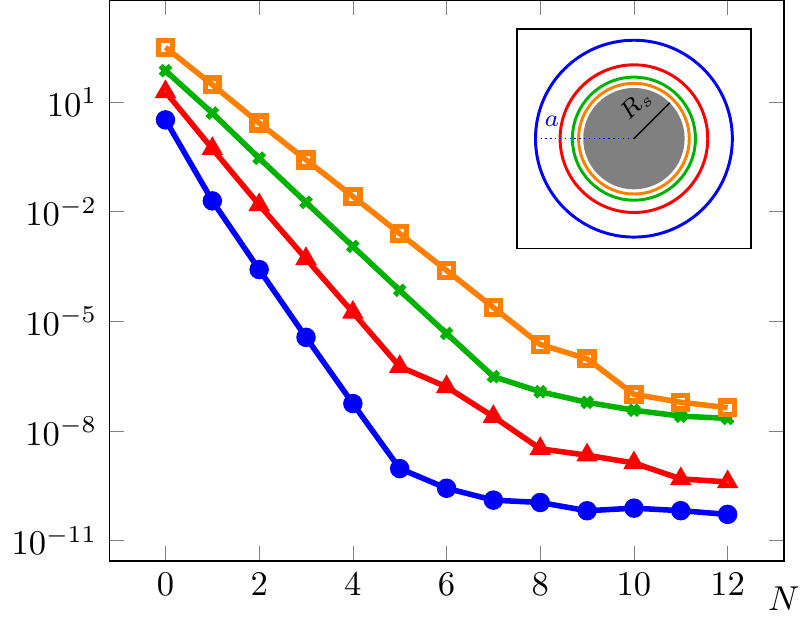}
 }
\caption{ Influence of the distance $\vert a-\RScatter \vert $ between coupling boundary and scatterer on the performance of learned IEs for $\wavenr = 16$ and $p=12$. Relative $L^2$-error on $\Omega_{\mathrm{int}}  $ on the left and weighted uniform 
error of the dtn modes on the right. }
\label{fig:plane-wave-coupling-radius}
\end{figure}

\subsection{Point source inside unit disk}\label{ssection:point_source}

The fundamental solution $\Phi(x,y) = (i/4)H_{0}^{(1)}( \wavenr \Vert x-y \Vert )$ fulfills
\begin{equation*}
- \Delta_{x}{\Phi(x,y)} - \wavenr^2 \Phi(x,y) = \delta(x-y)  := \delta_{y}(x),
\end{equation*}
where $\delta_{y}$ is the Dirac distribution at $y \in \mathbb{R}^2$.
The point source $y$ is placed inside a disk with radius $a= 1$.
 In this example $\nabla_{x} \Phi(x,y) \cdot \nu(x)$ is directly used as Neumann data for the exterior problem realized e.g. by the learned infinite elements.
 In particular, there is no interior discretization. 
 The quality of the solution is assessed by measuring the relative $L^2$-error in the Dirichlet data on $\Gamma$.
 The difficulty of the problem increases with shrinking distance of the source to the boundary as this adds more significant modes to the solution. 
 The weights for the minimization problem are chosen accordingly  as $w_{\ell} \sim \vert H^{(1)}_{\ell}(\wavenr a) / H^{(1)}_{\ell}(\wavenr  \Vert y \Vert ) \vert$.
 To investigate the influence on the transparent boundary condition the source positions $y = (0.5,0.0)$ and $y = (0.95,0.0)$ are considered in the experiments.
     \par 

To evaluate the performance of learned infinite elements a comparison with other transparent boundary conditions is presented.  
%These methods are briefly described in the following paragraphs.  \par 
% A PML formulation, cf. \cref{ssec:PML} applies a complex stretching  stretched into the complex plane by  means of the transformation 
% \begin{equation}\label{eq:complex_radius}
% \tilde{r} =  r + \int\limits_{a}^{r} i \sigma(t),
% \end{equation}
% where $\sigma$ is the absorption coefficient. 
% Applying this transformation  to the Helmholtz equation in the exterior leads to a tensor product system of the form \cref{eq:TP_discretization_exterior}, see \cref{ex:TP_PML}.
% Since the PML is chosen to be radial only the matrices $\dstiffr$ and $\dmassr$ are affected by this coordinate stretching.
%
%Using a PML formulation, cf. \cref{ssec:PML}, we obtain the following matrices $\dstiffr$ and $\dmassr$:
%\begin{equation*}
%\dstiffr_{\mu \nu} = \int\nolimits_{a}^{ \eta} \left( \frac{r}{\rho} \tilde{s} s \frac{\partial g_\nu}{\partial r} \frac{\partial g_\mu}{\partial r} - \frac{\omega^2 r \tilde{s}}{s \rho c^2} g_\nu g_\mu  \right) ~dr, \quad 
%\dmassr_{\mu \nu} = \int\nolimits_{a}^{ \eta} \frac{1}{r \tilde{s} s \rho } g_\nu  g_\mu ~dr,
%\end{equation*}
%where $\eta - a  > 0 $ is the thickness of the PML and 
%\begin{equation*}
%s(r) = \left( 1+ i \sigma(r) \right)^{-1}, \quad \tilde{s}(r) = 1 + \frac{i}{r} \int\nolimits_{a}^{r} \sigma(t) ~dt. 
%\end{equation*}
For the experiments below a quadratic absorption coefficient $\sigma(t) = (C / \wavenr) (1 / ( \eta - a) ) ( (t-a)/ (\eta-a) )^2$ was used in \eqref{eq:complex_radius}. 
%For such a profile the PML extends over the whole computational domain $[a,\eta]$.
The parameters $C = 40$ and $\eta  - a = 0.02$ 
have been found to yield good results. The implementation uses finite elements of order four with uniform mesh refinements.  \par 

Tuning of the PML and discretization parameters in this approach is tedious and may lead to suboptimal results. 
A better strategy has been proposed in \cite{CL06}.
Following \cite{CM98} the medium parameter $\sigma_{0}$ in $\sigma(t) = \sigma_{0} ( (r-a)/(\eta-a))^{m}, m \in \mathbb{N}$ and the thickness $\eta - a$
in the complex stretching of the PML are determined through an a posteriori error analysis.
This is achieved by splitting the error into a finite element discretization error and a term describing the modeling error introduced by the PML (see Theorem 3.1. in  \cite{CL06}).
The contribution from the PML decreases with the exponentially decaying factor 
\begin{equation*}
\exp{  \left( -  \wavenr \Im(\tilde{\eta}) \left( 1 - \frac{a^2}{ \vert \tilde{\eta} \vert^2 }  \right)^{1/2} \right) }, \quad \tilde{\eta} =   \eta + i \frac{\sigma_{0}}{m+1} (\eta - a).
\end{equation*}
By choosing $m=2$, $\sigma_{0} = 4.5$, and $\eta = 2.5$ the PML error is of the order of the machine precision for the considered example. 
With the PML parameters being fixed an adaptive mesh refinement strategy based on a standard residual error estimator is employed 
to reduce the discretization error. Note that this leads to an unstructured mesh. 
In particular this method is not of the tensor product form \eqref{eq:ansatz}.
For the experiments presented here finite elements of order six have been chosen. \par 

As a last candidate we consider the Hardy space infinite element method (HSIE), cf.~\cref{sec:otherIE}. This method has a parameter which is usually called $\kappa_{0}$. For the experiments below $\kappa_{0} = a \wavenr$ 
was used for $y = (0.5,0)$ while a larger value of $\kappa_{0} = 3 a \wavenr$ turned out to be beneficial for $y = (0.95,0)$.

In \cref{fig:ptsource-relerr} the different methods are compared in terms of the 
number of \dofs~(`ndof') and number of nonzero entries (`nze') of the resulting linear system.
\begin{itemize}
\item The lower panel displays the results for a source 
far away from the boundary. In this case both methods based on infinite elements outperform the PML approaches. 
Although the HSIE already works very well for this problem, the learned approach yields a further improvement. 
The reduced ansatz for the infinite elements matrices (`Learned reduced') results in a sparser system matrix but 
makes the optimization problem a bit harder to solve. 
As a result, it may happen that the minimization gets stuck in a local minimum or makes slower progress compared to the dense ansatz.
However, so far this has only been observed in the very high accuracy range.
 In \cref{fig:ptsource-far-ndof} and \cref{fig:ptsource-far-nze} this occured for relative accuracy better than $10^{-12}$.  
\item The upper panel gives the results for the source close to the boundary. Here, the performance of the HSIE degrades strongly. Both PML approaches perform better in terms of nonzero entries of the linear system. Among the PML methods, the adaptive discretization achieves substantially better results than the tensor product PML. In particular, it leads to very sparse matrices. Still, the learned infinite elements are able to improve on this.
\end{itemize}

 \begin{figure}[htbp]
\centering
\subfloat[ $y = (0.95,0.0)$ ]{ \label{fig:ptsource-close-ndof} 
\includegraphics[scale=0.75]{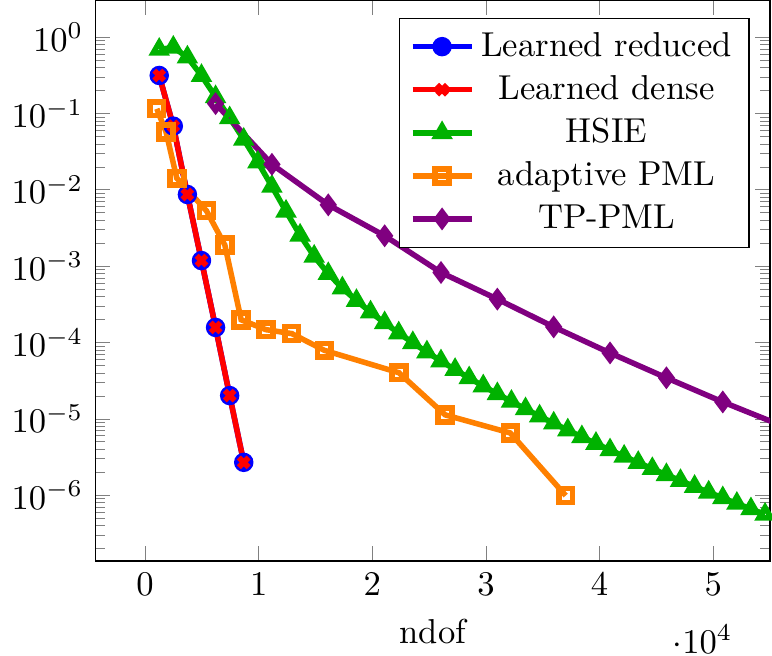}
}
\subfloat[ $y = (0.95,0.0)$  ]{ \label{fig:ptsource-close-nze}
\includegraphics[scale=0.75]{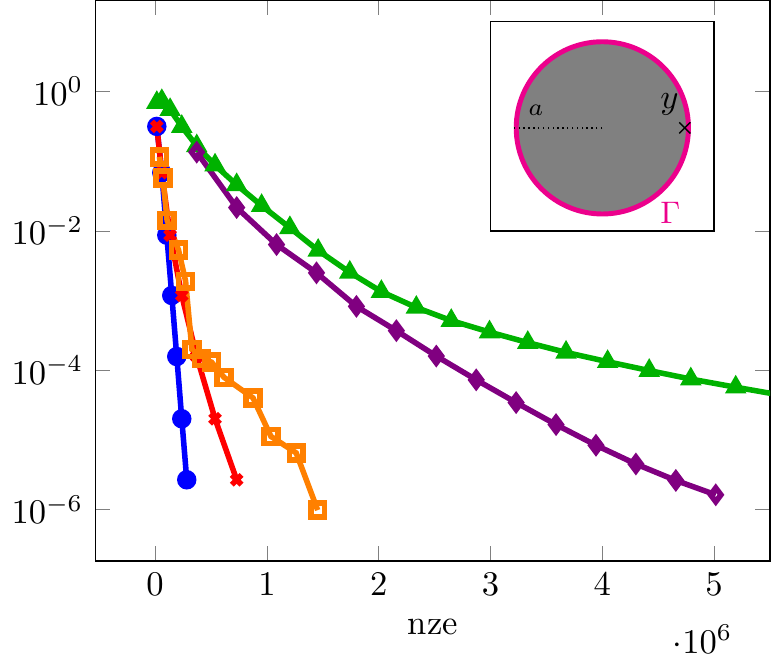}
 } \\
\subfloat[ $y = (0.5,0.0)$ ]{ \label{fig:ptsource-far-ndof} 
\includegraphics[scale=0.75]{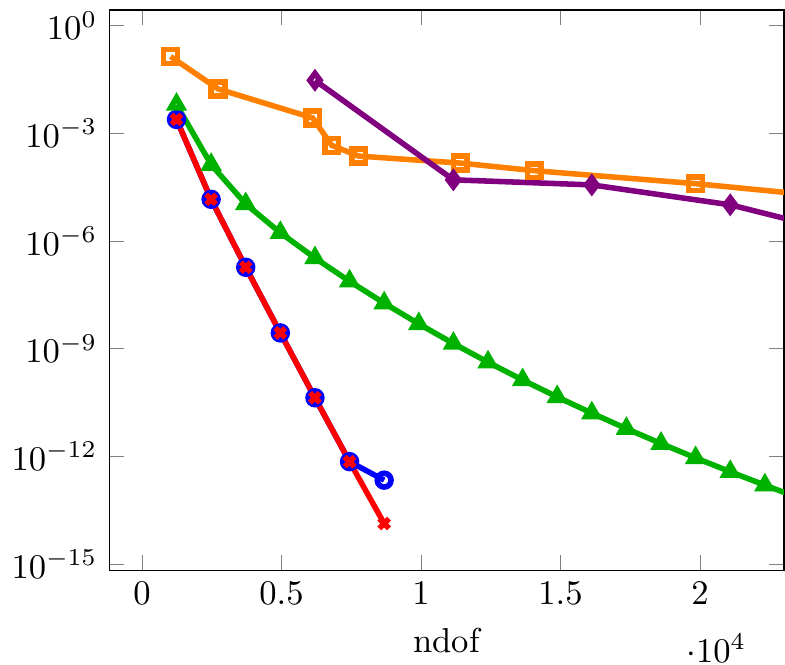}
}
\subfloat[ $y = (0.5,0.0)$  ]{ \label{fig:ptsource-far-nze}
\includegraphics[scale=0.75]{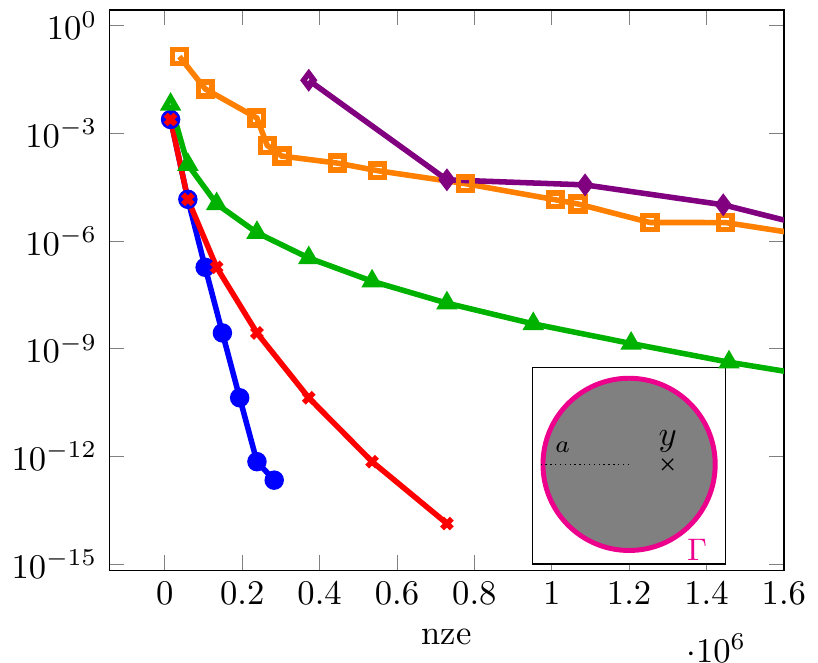}
 }
\caption{ Comparison of relative error $\Vert u-u_h \Vert_{L^2(\Gamma)} /\Vert u \Vert_{L^2(\Gamma)}  $ for the point source inside the unit disk at position $y$ with $\wavenr=16$.  }
\label{fig:ptsource-relerr}
\end{figure}

\subsection{Jump in exterior wavenumber}\label{ssection:jump_numexp}

 \begin{figure}[htbp]
\centering
%\subfloat[ $\Re(\dtn(\lambda))$ ]{ \label{fig:zeta_omega_inf_real} 
%\includegraphics[scale=0.75]{Figures/dtn-jump-infAll-real.pdf}
%}
\subfloat[ $\wavenr_{\infty} = \wavenr_{I} = 16 $  ]{ \label{fig:omega_inf_16_real}
\includegraphics[scale=0.5]{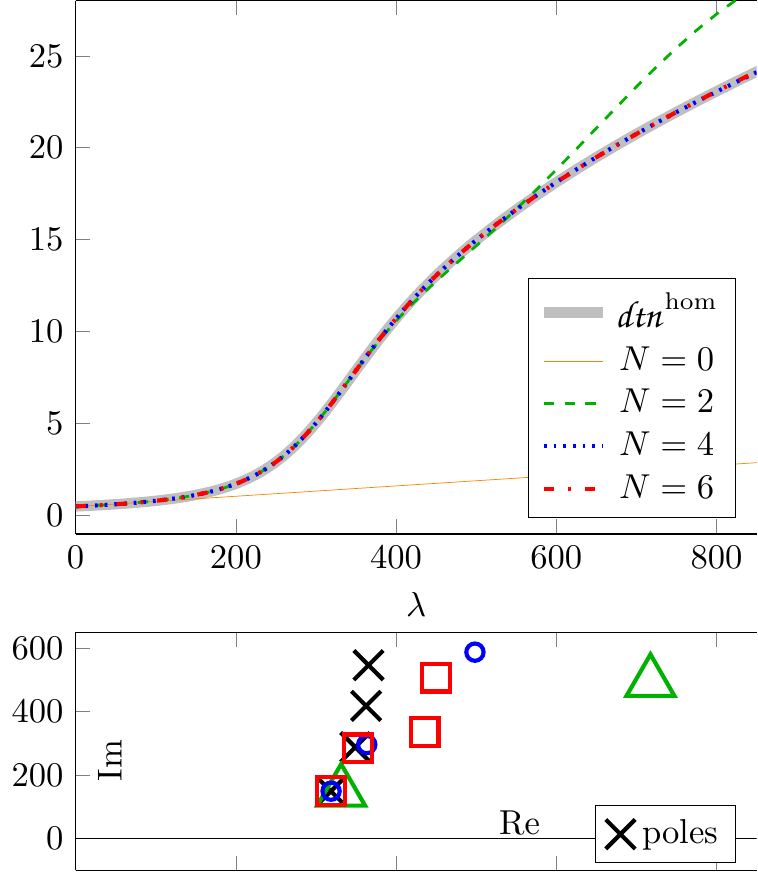}
 } 
\subfloat[  Relative $L^2$-error ]{ \label{fig:jump_relerr} 
\includegraphics[scale=0.5]{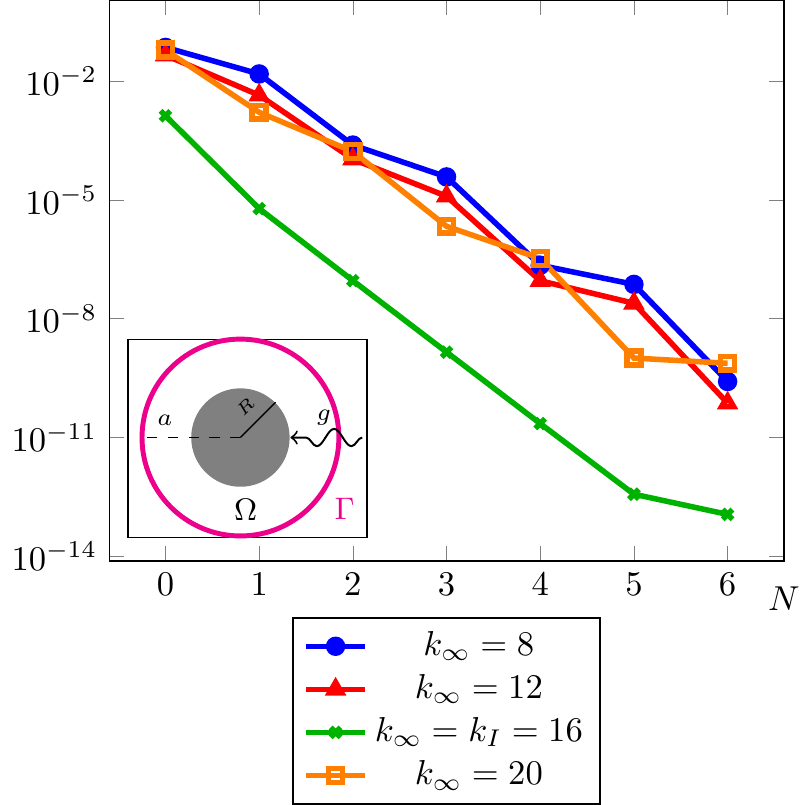}
}
%\subfloat[  $\wavenr_{\infty} = 12$ ]{ \label{fig:omega_inf_12_real} 
%\includegraphics[scale=0.5]{Figures/dtn-jump-inf12-real.pdf}
%}
\subfloat[ $\wavenr_{\infty} = 8$  ]{ \label{fig:omega_inf_8_real}
\includegraphics[scale=0.5]{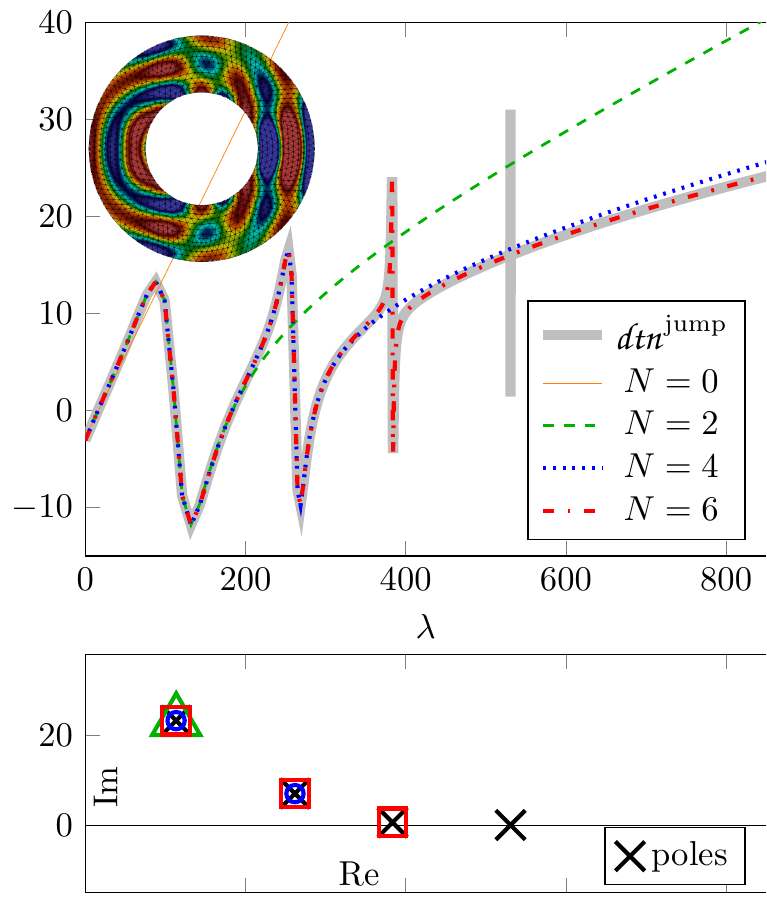}
 }
\caption{ Approximation results for the real part of $\dtn^{\mathrm{jump}}$ from Example \ref{ex:jump}
with a discontinuous exterior wavenumber jumping from $\wavenr_{I} = 16$ to $\wavenr_{\infty} = 8$ shown in \cref{fig:omega_inf_8_real} compared
to the $\dtn^{\mathrm{hom}}$ function of the homegeneous medium, which is recovered for the case $\wavenr_{I} = \wavenr{\infty} = 16 $, in \cref{fig:omega_inf_16_real}.
In the lower panel we also compare the poles of $\dtn$ (black crosses)  to the poles of $\ddtn$ (colors according to legend). Only the poles of $\ddtn$ close to
	 the real axis are displayed. The relative $L^2$-error for the scattering of a plane wave is shown in the central figure. 
%\thnote{Sollten wir (a) lieber $\omega_\infty=20$ nehmen? Sonst ist es ja eine Wiederholung des 
%ersten Beispiels.}	 
	 }
\label{fig:dtn_omega_inf_real}
\end{figure}

For the numerical experiments $a = 1$, $\RJump=2$ and $\wavenr_{I} = 16$ are chosen in \cref{ex:jump}. 
In \cref{fig:omega_inf_16_real} and \cref{fig:omega_inf_8_real} the approximation $\ddtn(\lambda) $ obtained with the reduced ansatz \cref{eq:zeta_learned_diag} for different $N$ is shown.
While \cref{fig:omega_inf_16_real} shows the case $\wavenr_{\infty} = \wavenr_{I} = 16$ in which there is no jump and the $\dtn^{\mathrm{hom}}$ function of the homogeneous medium is recovered, 
\cref{fig:omega_inf_8_real} displays a case in which the wavenumber jumps drastically.  
Additionally, the poles of the meromorphic extension of $\dtn$ and the learned poles (see \cref{eq:learned_poles}) are displayed.
The approximation results for the case $\wavenr_{\infty} = 8$ shown in \cref{fig:omega_inf_8_real} are astonishing.
The rapid improvement as $N$ increases can be explained by considering the positioning of the learned poles. 
For $N=2$ the optimization routine has placed one of the learned poles at the position of the first exact pole. 
This captures the first peak of $\dtn^{\mathrm{jump}}$. 
For $N=4$ the next pole has been fitted and accordingly the next peak is covered by the approximant. 
It is remarkable that for $N=6$ the peak can be captured accurately even though it is not resolved 
by the sampling of the $\lambda$-axis in the minimization problem. 
Similar results shown in the appendix are obtained for $\wavenr_{\infty} = 12$ and $\wavenr_{\infty} = 20$. 
This clearly demonstrates the benefits of using rational functions for approximating the $\dtn$. \par 

In analogy to \cref{ssection:ex_plane_wave_hom} we consider the scattering of a plane wave from a disk with radius $\RScatter=1/2$. 
An analytic reference solution for this problem, displayed in \cref{fig:omega_inf_8_real} for $\wavenr_{\infty} = 8$, is derived in the appendix.
For the case $\wavenr_{\infty} = \wavenr_{I} = \wavenr$ it reduces to the solution \cref{eq:PlaneWaveHomSolution} for the homogeneous problem, whose real part was shown in \cref{fig:plane-wave-disc-relerr-p}. 
The weights are chosen as  $w_{\ell} \sim \vert u_{\ell}(a) / u_{\ell}( \tilde{r}(\wavenr_I))  \vert$, where 
$u_{\ell}(r) = A_{\ell} J_{\ell}(\wavenr_I r )  + B_{\ell} Y_{\ell}(\wavenr_I r )$ is the solution of the radial equation for $r < \RJump $ and 
$\tilde{r}(\wavenr_I)$ is as in \cref{ssection:ex_plane_wave_hom}.  
The relative $L^2$-error on the computational domain $ \Omega_{\mathrm{int}}$  is shown in \cref{fig:jump_relerr} for different $\wavenr_{\infty}$. 
Although the error is a bit larger when a jump actually occurs, i.e $\wavenr_{\infty} \neq \wavenr_{I}$, the convergence is nevertheless exponential.

%For illustrating the influence of the jump, we also plot the real part of the solution of a scattering problem on an 
%annulus $1/2 < r < a$ in which the learned IEs with $N=6$ are used as transparent boundary condition.
%The Neumann trace of a point source located at $y = (1/3,-1/3)$ is used as data on the scatterer similar as in \cref{ssection:point_source}. 
%In case of $\wavenr_{\infty} = \wavenr_{I}$ the dtn function of the homogeneous problem is recovered. 
%For this case the relative error on the modes and for the solution have been shown in \cref{fig:learning_curves} and \cref{fig:plane-wave-disc-relerr-k} respectively. \par  

\subsection{Waveguide}\label{ssection:waveguide}

In this section we consider the waveguide introduced in \cref{ex:Helmholtz_waveguide}.
Setting $\tilde{\Gamma} = [0,\pi]$ and letting the waveguide start at the origin the solution shall be computed in $\Omega_{\mathrm{int}}  = [0,a) \times \tilde{\Gamma}$.
The geometry is sketched in \cref{fig:dtn-guide}.
It is convenient to choose Dirichlet boundary conditions on $\partial \Omega_{\mathrm{int}} \setminus ( \{a \} \times \tilde{\Gamma})$ so that 
the exact solution is given by 
\[ u = \sum\limits_{\ell = 0}^{L} \sin\left(y \sqrt{ \lambda_{\ell} }\right) \exp\left(i x \sqrt{\wavenr^2- \lambda_{\ell}}\right),  \]
for $\lambda_{\ell} = \ell^2$.
For the numerical experiment the parameters $a = 2 \pi$, $\wavenr = 16.5$ and $L = 33$  are chosen. 
Note that \cref{eq:waveguide_assump} holds. 
The weights are chosen to cover all propagating waveguide modes equally well and emulate the decay of the evanescent modes, i.e. $w_{\ell} \sim 1 $ for $\lambda_{\ell} \leq \wavenr^2$ and $w_{\ell} \sim \vert \exp{(ia \sqrt{\wavenr^2- \lambda_{\ell}}}) \vert $ else. \par  

In \cref{fig:waveguide} results for the approximation of $\Im(\dtn^{\mathrm{guide}})$ are shown. 
(Note the $\Re(\dtn^{\mathrm{guide}})$ is zero for most relevant modes.) 
As one might expect, the kink of $\dtn^{\mathrm{guide}} $ is hardest to resolve.
Interestingly, the learned poles are observed to accumulate in the part of the complex plane lying to the right of the kink, which
is probably beneficial for approximation. 
We also note that the approximation at the sample points  $\dtn^{\mathrm{guide}}(\lambda_{\ell}) $ is extremely accurate for large $N$. 
This can be expected from the paper \cite{N:64} by Newman who showed that the function $|x|$ can be approximated by rational 
functions of order $N$ with  an error $\leq 3\exp(-\sqrt{N})$ uniformly on the interval $[-1,1]$.   
This (heuristically) explains this type of fast convergence of the relative $L^2$-error that can be observed in \cref{fig:waveguide_L2}. 

 \begin{figure}[htbp]
\centering
\subfloat[ $\dtn^{\mathrm{guide}} $   ]{ \label{fig:dtn_waveguide_imag} 
\includegraphics[scale=0.75]{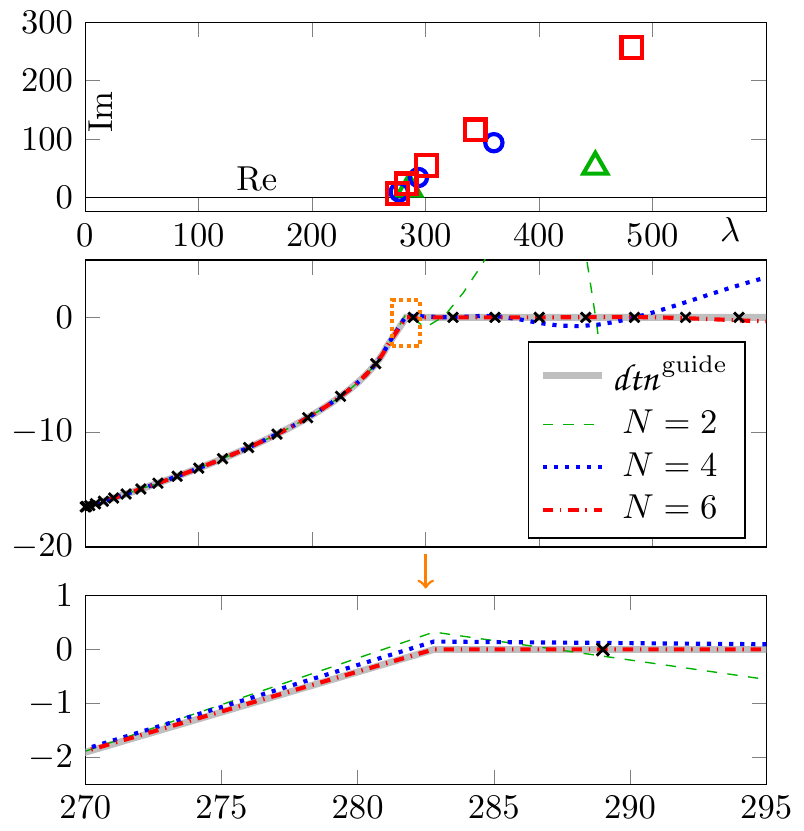}
}
\subfloat[ Relative $L^2$-error ]{ \label{fig:waveguide_L2} 
\includegraphics[scale=0.75]{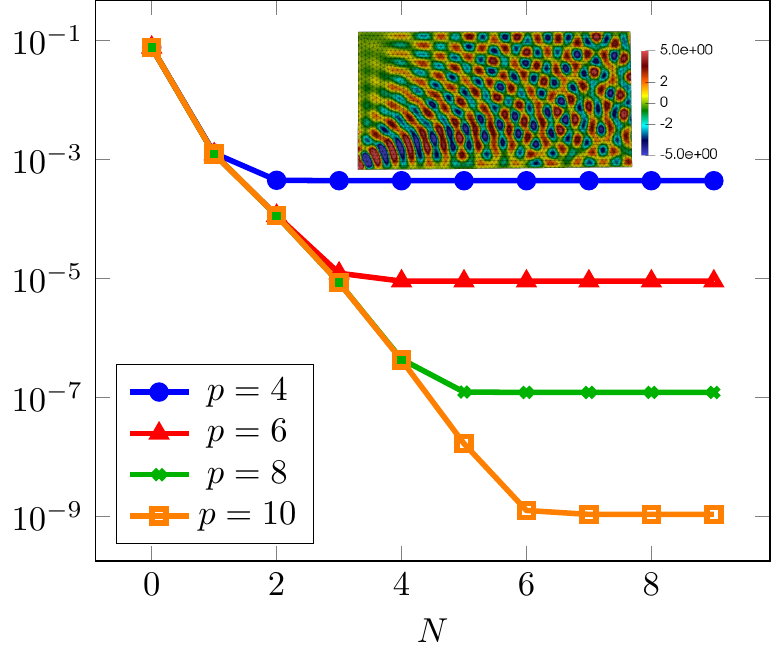}
}
	 \caption{  Left: Approximation results for imaginary part of $\dtn^{\mathrm{guide}} $ with $\wavenr = 16.5$. The sample points $\dtn^{\mathrm{guide}}(\lambda_{\ell}) $  are displayed as black crosses. In the topmost plot also some poles of $\ddtn$ are displayed. 
	   Right: Relative $L^2$ error on $\Omega_{\mathrm{int}}$. Additionally, the real part of the reference solution is shown. }
\label{fig:waveguide}
\end{figure}

\subsection{Application to helioseismology}\label{ssection:helio}

We now proceed to the main application introduced in \cref{sec:helio_models} which fueled the development of learned infinite elements, the approximimation of the VAL-C model of the solar atmosphere. At first sight it may seem that a full finite element discretization of the thin layer of validity of the VAL-C model colored  
blue color in \ref{fig:dtn-VALC}, which comprises only $1.3\%$ of the volume of the Sun, costs only 
negligible additional effort. However, the wavelength $\lambda = 2 \pi c / \omega$ decreases significantly close to and above the visible surface of the Sun such that this thin layer contains 
$2$ wavelengths in radial direction at a frequency of $7.0$ mHz, compared to $25$ wavelengths in the 
interior. In order to resolve wave oscillations, 
the local mesh size $h$ should be roughly proportional to the wavelength. Assuming a \emph{regular} meshes  
(in the sense that for each finite element domain $D$ the ratio of the largest ball contained 
in $D$ to the smallest ball containing $D$ is uniformly bounded) 
 with $h$ proportional to $\lambda$, the number of elements in the ball of radius $r$ is proportional to  
%In particular, a discretization of the solar atmosphere can be avoided. 
%This is of considerable practical importance since the wavelength $\lambda = 2 \pi c / \omega$ in the solar atmosphere is extremely small.
%To be more concrete, let us measure the number of wavelength contained in a ball in $\mathbb{R}^{d},$ of radius $r$, i.e.  
\[ I(r) :=   \int\nolimits_{ \vert x \vert_2 \leq r}{ \lambda^{-d} ~ dx}, 
\]
in $d$ dimensions. 
\cref{tab:dof_helio_comp} displays the ratio between the wavelength to be resolved in a $d$-dimensional discretization of the Sun up to the photosphere $I(a)$ and a discretization which covers an additional $2300$ km of the solar atmosphere $I(\RVALC)$ corresponding to the domain of the VAL-C model.  
Whereas in one dimension, the additional effort for discretization of the atmosphere may be regarded as
acceptable, in
three dimensions such an additional discretization of the atmosphere would at least triple the computational burden.
In view of the notoriously challenging linear systems arising from time-harmonic wave equations such an increase must absolutely be avoided.

\begin{table}[htbp]
\begin{center}
\begin{tabular}{|c|c|c|c|} \hline 
$d$           &  $1$ (radial symmetry)  & $2$ (axial symmetry) & $3$ (no symmetry)\\ \hline 
$I(\RVALC) / I(a)$ &  $1.1$ & $1.8$ & $3.2$ \\ \hline 
\end{tabular}
\end{center}
\caption{\label{tab:dof_helio_comp} Ratio of the numbers of finite elements of regular $d$-dimensional 
finite element meshes of the Sun with mesh size proportional to wavelength for different radii $\RVALC$ and $a$. 
$a$ corresponds to the radius of the photosphere (the visible surface of the Sun) whereas 
$\RVALC=a+2300$ km is the outer radius of the domain of validity of the VAL-C model of the solar atmosphere. 
These numbers demonstrate the necessity of transparent boundary conditions for the solar atmosphere. }
% These results were obtained for a frequency of $7.0$ mHz. }
\end{table} 

\cref{fig:dtn_VALC}  shows the approximations $\ddtn(\lambda_{\ell}) $ of $\dtn^{\mathrm{VAL-C}}$ 
for different $N$. 
For efficiency purposes the 
distance to $\dtn^{\mathrm{VAL-C}}$  has only been minimized at a subset of the continuous eigenvalues $\lambda_{\ell} = \ell (\ell +1)/a^2$. 
This sets consists of two parts. 
The first is an equidistant sampling in $\ell$ to resolve the general behavior of $\dtn^{\mathrm{VAL-C}}$. 
The second is a fine sampling near the poles, which is readily obtained by computing the gradient of the array containing the samples of $\dtn^{\mathrm{VAL-C}}$.
The weights decay exponentially as $w_{\ell} \sim \exp(-  \ell  / 800 )$.

Using this strategy similar approximation results as for the example from \cref{ssection:jump_numexp} are obtained.
Adding more infinite element \dofs~allows to resolve an increasing number of the narrow turning points of 
$\dtn^{\mathrm{VAL-C}}$. 
The behavior of $\dtn^{\mathrm{VAL-C}}$ on the domain of interest is so completely dominated by the closely located poles that $N=5$ suffices 
to obtain an excellent approximation. \par 

 \begin{figure}[htbp]
\centering
\subfloat[ Real part. ]{ \label{fig:dtn_omega_7mhz_real}
\includegraphics[scale=0.89]{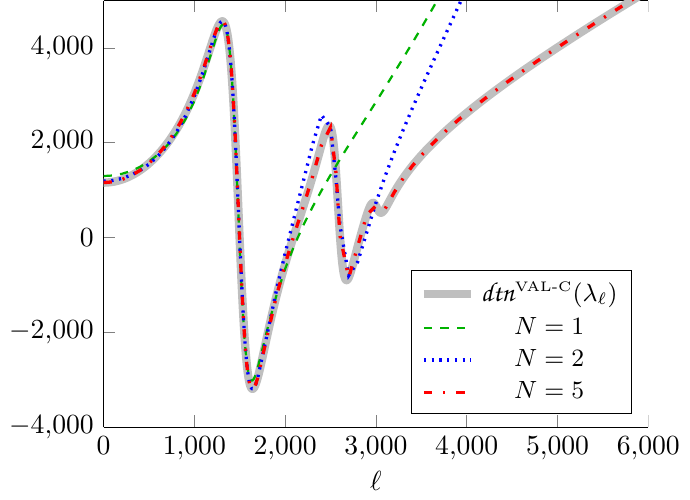}
 }
 \centering
 \subfloat[ Imaginary part. ]{ \label{fig:dtn_omega_7mhz_imag}
\includegraphics[scale=0.89]{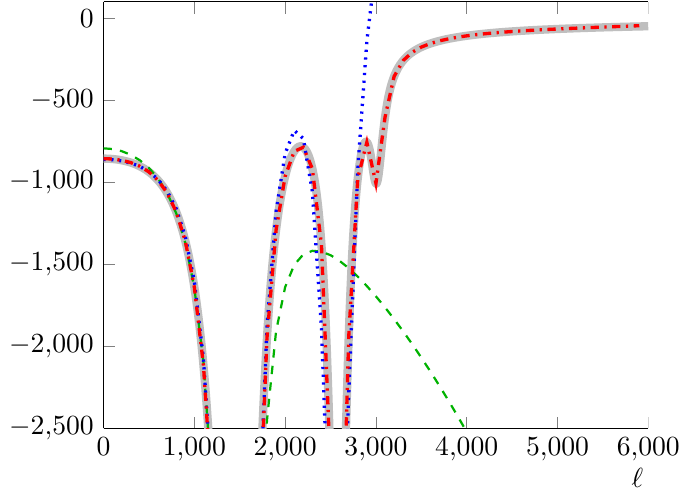}
 }
\caption{ Approximation results for problem from computational helioseismology at $7.0$ mHz.}
\label{fig:dtn_VALC}
\end{figure}

These results demonstrate that learned infinite elements allow to absorb the atmospheric behavior into a transparent boundary condition with very few DOFs which can be placed directly at the photosphere.

%\begin{figure}[htbp]
%  \centering
%  \label{fig:a}\includegraphics{lexample_fig1}
%  \caption{Example figure using external image files.}
%  \label{fig:testfig}
%\end{figure}

\section{Conclusions and outlook}
\label{sec:conclusions}

We have proposed a new type of transparent boundary condition in the form of infinite elements.  
In contrast to traditional methods which involve some explicit modelling of the exterior differential equation, we suggest 
to obtain a sparse and accurate approximation of the DtN map by `learning' the matrix elements from the 
DtN map in a preprocessing step. Our analysis applies 
to separable problems, in particular  spherically symmetric problems where the learning process only requires knowledge of the eigenvalues of the continuous DtN map, which are often given analytically, but neither discrete eigenvalues nor discrete eigenvectors are needed. Numerical test show that learned infinite elements for the Helmholtz equation outperform traditional methods such as PMLs both in terms of numbers of degrees of freedom and numbers of nonzero matrix elements. At the same time they are more flexible and applicable to stratified media with non-analytic separable coefficients. A preliminary theoretical analysis 
suggests exponential convergence with increasing infinite element degrees of freedom. 
A crucial role is played by 
the position of singularities of analytic extensions of the `eigenvalue function' of the DtN map, which is initially defined on the spectrum of the Laplace-Beltrami operator of the coupling boundary. 
This raises new theoretical questions, e.g.\ concerning potential connections to resonances. 

Let us discuss some potential extensions of our approach:
\begin{itemize}
\item \emph{Relaxation of the separation assumption \eqref{eq:separable_bvp2}.} This assumption 
might be replaced by a separation assumption with respect to operator pencils different from 
the Laplace-Beltrami and the identity operator on the coupling boundary. This could allow to 
treat systems of PDEs and separation with respect to elliptic coordinates. 
\item \emph{Infinite elements which work on intervals of wave numbers.}
Currently the learned parameters only work for one specific wave number. Therefore, they may hardly be reused. With the proposed extension libraries of optimized coefficients may be set up to allow the use of 
infinite elements without the need to implement the learning process. 
Moreover, this may enable the application of infinite elements for the solution of resonance problems. 
\item \emph{Non-separable differential equations.}
To treat such problems one may study learned infinite elements which are not of tensor product form. Here probably it cannot be avoided to set up a discrete DtN map for the learning process in some form, and the number of degrees of freedom would increase significantly. Nevertheless, we see some potential for such approaches, e.g.\ in the context of domain decomposition methods.  
\end{itemize}
Our favorable numerical results encourage further research on learned infinite elements in these and other directions in order to extend their range of applicability and improve theoretical foundation.

\appendix

\section{Derivation of $\dtn$ and reference solution for jump in exterior wavenumber}

For $\nu,a > 0$ consider the problem 
\begin{align*}
- \frac{1}{r} \frac{\partial}{\partial r} \left( r \frac{\partial u}{\partial r} \right) +  \left( \wavenr^2(r) - \frac{\nu^2}{r^2} \right)u &= 0 \quad r > a,  \\ 
u(a) &= 1,
\end{align*}
with radiation condition at infinity. For some $a \leq  \RJump < \infty$ let the wavenumber be given 
by 
\begin{equation*}
\wavenr(r) = \begin{cases}
\wavenr_{I} & r < \RJump, \\
\wavenr_{\infty} & r > \RJump,
\end{cases}
\end{equation*}
for some $\wavenr_{I}, \wavenr_{\infty} > 0$. \par 

This problem has the solution 
\begin{equation*}
u_{\nu}(r) = \begin{cases}
A_{\nu} J_{\nu}(\wavenr_{I}r) + B_{\nu} Y_{\nu}(\wavenr_I r) & r < \RJump, \\
C_{\nu} H_{\nu}^{(1)}(\wavenr_{\infty}r) & r > \RJump,
\end{cases}
\end{equation*}
with constants $A_{\nu},B_{\nu},C_{\nu} \in \mathbb{C}$ to be determined by the following three constraints:
\begin{itemize}
\item Boundary condition at $r = a$ 
\begin{equation}\label{eq:Dirichlet_a}
A_{\nu} J_{\nu}(\wavenr_I a ) + B_{\nu} Y_{\nu}(\wavenr_I a ) = 1.
\end{equation}
\item Continuity at $r = \RJump $
\begin{equation}\label{eq:Continuity_R}
A_{\nu} J_{\nu}(\wavenr_I \RJump ) + B_{\nu} Y_{\nu}(\wavenr_I \RJump ) - C_{\nu} H_{\nu}^{(1)}(\wavenr_{\infty} \RJump ) = 0.
\end{equation}
\item Continuity of derivative at $r=\RJump $
\begin{equation}\label{eq:Continuity_deriv_R}
A_{\nu} \wavenr_I  J_{\nu}^{\prime}(\wavenr_I \RJump ) + B_{\nu} \wavenr_I Y_{\nu}^{\prime}(\wavenr_I \RJump ) - C_{\nu} \wavenr_{\infty} (H_{\nu}^{(1)})^{\prime}(\wavenr_{\infty} \RJump ) = 0.
\end{equation}
\end{itemize}
Solving (\ref{eq:Dirichlet_a}) for $A_{\nu}$ gives 
\begin{equation}\label{eq:A}
A_{\nu} = \frac{1-B_{\nu} Y_{\nu}(\wavenr_I a ) }{ J_{\nu}(\wavenr_I a ) }.
\end{equation}
Using this equation to eliminate $A_{\nu}$ from \ref{eq:Continuity_R} and \cref{eq:Continuity_deriv_R} leads to the linear system 
\begin{equation*}
 \begin{bmatrix}
 \frac{J_{\nu}(\wavenr_I a ) Y_{\nu}(\wavenr_I \RJump ) - Y_{\nu}(\wavenr_I a ) J_{\nu}(\wavenr_I \RJump ) }{ J_{\nu}(\wavenr_I a ) }   & - H_{\nu}^{(1)}(\wavenr_{\infty} \RJump )     \\
 \frac{ Y_{\nu}^{\prime}(\wavenr_I \RJump ) J_{\nu}(\wavenr_I a ) - Y_{\nu}(\wavenr_I a ) J_{\nu}^{\prime}(\wavenr_I \RJump )   }{ J_{\nu}(\wavenr_I a )  }   &  -\frac{\wavenr_{\infty}}{\wavenr_I} (H_{\nu}^{(1)})^{\prime}(\wavenr_{\infty} \RJump ) \\
\end{bmatrix}
 \begin{bmatrix}
 B_{\nu} \\
 C_{\nu}
 \end{bmatrix}
 = 
 \begin{bmatrix}
  - \frac{J_{\nu}(\wavenr_I \RJump )}{ J_{\nu}(\wavenr_I a ) } \\
  - \frac{ J_{\nu}^{\prime}(\wavenr_I \RJump )}{ J_{\nu}(\wavenr_I a ) }
 \end{bmatrix}.
\end{equation*}
Denote the matrix in this equation as $M_{\nu}$. We have 
\begin{align*}
\det(M_{\nu}) &= -\frac{\wavenr_{\infty}}{\wavenr_I} \frac{(H_{\nu}^{(1)})^{\prime}(\wavenr_{\infty} \RJump )}{ J_{\nu}(\wavenr_I a )} \left[ J_{\nu}(\wavenr_I a ) Y_{\nu}(\wavenr_I \RJump ) - Y_{\nu}(\wavenr_I a ) J_{\nu}(\wavenr_I \RJump ) \right] \\
 & + \frac{ H_{\nu}^{(1)}(\wavenr_{\infty} \RJump ) }{J_{\nu}(\wavenr_I a )} \left[ Y_{\nu}^{\prime}(\wavenr_I \RJump ) J_{\nu}(\wavenr_I a ) - Y_{\nu}(\wavenr_I a ) J_{\nu}^{\prime}(\wavenr_I \RJump )  \right].
\end{align*}
Hence, 
\begin{equation*}
M^{-1}_{\nu} =  \frac{1}{\det(M_{\nu})}   
\begin{bmatrix}
-\frac{\wavenr_{\infty}}{\wavenr_I} (H_{\nu}^{(1)})^{\prime}(\wavenr_{\infty} \RJump )    &  H_{\nu}^{(1)}(\wavenr_{\infty} \RJump )     \\
 \frac{ - \left[ Y_{\nu}^{\prime}(\wavenr_I \RJump ) J_{\nu}(\wavenr_I a ) - Y_{\nu}(\wavenr_I a ) J_{\nu}^{\prime}(\wavenr_I \RJump )  \right]  }{ J_{\nu}(\wavenr_I a )  }   &  \frac{J_{\nu}(\wavenr_I a ) Y_{\nu}(\wavenr_I \RJump ) - Y_{\nu}(\wavenr_I a ) J_{\nu}(\wavenr_I \RJump ) }{ J_{\nu}(\wavenr_I a ) }  \\
\end{bmatrix}.
\end{equation*}
The solution of the linear system is given by 
\begin{equation}\label{eq:B}
B_{\nu} =  \frac{1}{\det(M_{\nu})} \left[ \frac{\wavenr_{\infty}}{\wavenr_I} (H_{\nu}^{(1)})^{\prime}(\wavenr_{\infty} \RJump ) \frac{J_{\nu}(\wavenr_I \RJump )}{ J_{\nu}(\wavenr_I a ) } -    H_{\nu}^{(1)}(\wavenr_{\infty} \RJump ) \frac{ J_{\nu}^{\prime}(\wavenr_I \RJump )}{ J_{\nu}(\wavenr_I a )  } \right]   
\end{equation}
and 
\begin{align}
C_{\nu} = \frac{1}{\det(M_{\nu})} \Big[  &  \left[ Y_{\nu}^{\prime}(\wavenr_I \RJump ) J_{\nu}(\wavenr_I a ) - Y_{\nu}(\wavenr_I a ) J_{\nu}^{\prime}(\wavenr_I \RJump )  \right] \frac{J_{\nu}(\wavenr_I \RJump )}{ J_{\nu}(\wavenr_I a )^2 }  \label{eq:C} \\ 
 & -  \left[ J_{\nu}(\wavenr_I a ) Y_{\nu}(\wavenr_I \RJump ) - Y_{\nu}(\wavenr_I a ) J_{\nu}(\wavenr_I \RJump ) \right] \frac{ J_{\nu}^{\prime}(\wavenr_I \RJump )}{ J_{\nu}(\wavenr_I a )^2 } \nonumber  \Big].
\end{align}
This yields the DtN function 
\begin{equation}\label{eq:dtn_jump_formula}
\zeta(\nu) = -\frac{\partial u_{\nu}(a)}{\partial r } = - \wavenr_I \left[ A_{\nu} J_{\nu}^{\prime}(\wavenr_I a )  + B_{\nu} Y_{\nu}^{\prime}(\wavenr_I a )  \right].
\end{equation}

A reference solution for a sound-soft scattering problem in the exterior of a disk with radius $\RScatter $ can easily be obtained from the previous computations. 
Firstly, the radius of the truncation boundary has to be replaced by the radius of the scatterer, i.e. set $a = \RScatter$. 
Assume that a plane wave $g = \exp(i k_I x)$ is incident on the disk. 
It is well-known that $g$ admits an expansion into orthogonal eigenfunctions of the Laplace-Beltrami operator on the boundary of the scatterer: 
\[ g = J_{0}( \wavenr_I r)  + \sum\limits_{\ell = 1} 2 i^{\ell} J_{\ell}(\wavenr_I r) \cos(\ell \varphi),   \] 
using standard polar coordinates $x = r \cos(\varphi)$ and $y = r \sin(\varphi)$.
Hence, by separation of variables the solution is given by 
\begin{equation}\label{eq:jump_refsol}
u(r,\varphi)  = u_{0}(r) J_{0}( \wavenr_I \RScatter)  + \sum\limits_{\ell = 1}^{\infty} u_{\ell}(r) 2 i^{\ell} J_{\ell}(\wavenr_I \RScatter) \cos(\ell \varphi).
\end{equation}

\subsection{Case of no jump}

In the special case 
\begin{equation}\label{eq:jump_equal_wavenumber}
\wavenr_{\infty} = \wavenr_{I} 
\end{equation}
there is no jump and the following computation shows that the $\dtn^{\mathrm{hom}}$ function of the homogeneous medium is recovered from \cref{eq:dtn_jump_formula}. 
A straight-forward calculation using \cref{eq:jump_equal_wavenumber}  yields:
\begin{equation*}
\det(M_{\nu})  = \frac{ H_{\nu}^{(1)}(\wavenr a ) }{  J_{\nu}(\wavenr a ) } W \{ J_{\nu}(\wavenr \RJump), Y_{\nu}(\wavenr \RJump ) \}, 
\end{equation*} 
where $W \{ J_{\nu}(\wavenr \RJump), Y_{\nu}(\wavenr \RJump ) \} =  J_{\nu}(\wavenr \RJump) Y_{\nu}^{\prime}(\wavenr \RJump ) - J_{\nu}^{\prime}(\wavenr \RJump)  Y_{\nu}(\wavenr \RJump )$ is the Wronskian. 
A similar calculation gives
\begin{equation*}
B_{\nu}  =  \frac{1}{\det(M_{\nu})} \frac{i}{ J_{\nu}(\wavenr a ) }  W \{ J_{\nu}(\wavenr \RJump), Y_{\nu}(\wavenr \RJump ) \}  = \frac{i}{ H_{\nu}^{(1)}(\wavenr a ) }.  
\end{equation*}
As 
\begin{equation*}
1-B_{\nu} Y_{\nu}(\wavenr_I a )  = \frac{ J_{\nu}(\wavenr a ) }{H_{\nu}^{(1)}(\wavenr a )   } 
\end{equation*}
it follows from \cref{eq:A} that $A_{\nu} = 1/H_{\nu}^{(1)}(\wavenr a )$. 
Inserting this into the ansatz for the solution yields $ u_{\nu}(r) =  H^{(1)}_{\nu}(\wavenr r) / H_{\nu}^{(1)}(\wavenr a ) $ for $r < \RJump$ and hence  $ \dtn^{\mathrm{jump}}(\lambda) =  \dtn^{\mathrm{hom}}(\lambda)  $.
The reference solution \cref{eq:jump_refsol} then takes the explicit form 
\[ u(r,\varphi) = \frac{ H^{(1)}_{0}(k r) }{ H^{(1)}_{0}(k \RScatter)  }  J_{0}( \wavenr \RScatter)  + \sum\limits_{\ell = 1}^{\infty} \frac{ H^{(1)}_{\ell}(k r) }{ H^{(1)}_{\ell}(k \RScatter)  }  2 i^{\ell} J_{\ell}(\wavenr \RScatter) \cos(\ell \varphi)    . \]

\section{Implementational details}\label{section:impl_details}

In this section some suggestions on how to solve the optimization problem introduced in \cref{ssection:minimization_problem} are provided. 
Although we had good success with the technique oulined below other approaches may be possible.
The \cref{ssection:impl_fixed_N} describes how the optimization problem for one specific $N$ is solved and \cref{ssection:successive_learning} then outlines the whole pipeline. 

\subsection{Solving the optimization problem for a fixed $N$}\label{ssection:impl_fixed_N}

The objective function for the minimization problem \cref{eq:misfit}-\cref{eq:minimization problem} can be written in the form 

\begin{equation}\label{eq:misfit_LM}
	J(\dstiffr,\dmassr) = \frac{1}{2} \sum\nolimits_{\ell}{    w_{\ell}^2  \vert f_{\ell}  \vert^2},
\end{equation}
with $f_{\ell}(\dstiffr,\dmassr) := \dtn(\lambda_{\ell}) -   \ddtn( \lambda_{\ell}   )$. 
Here we consider the case where $\ddtn( \lambda_{\ell}   )$ is given by the reduced ansatz \cref{eq:zeta_learned_diag}. 
The general case can be treated along the same lines, yet requires some more computations. 
The Levenberg-Marquardt algorithm is one of the commonly employed methods for solving non-linear least-squares problems of the form \cref{eq:misfit_LM}. 
Various open source implementations are readily available. 
These implementations usually require from the user the definition of the parameters, here the entries of $\dstiffr$ and $\dmassr$ to be optimized, 
the cost functions $f_{\ell}$ and their gradients with respect to these parameters. 
\begin{itemize}
\item Only the matrix entries of  $\dstiffr$ and $\dmassr$ that are not known a priori should be considered as parameters. So for ansatz \cref{eq:zeta_learned_diag} these are given by 
	\[  \{ \dstiffr_{0j} \}_{j=0}^{N} \cup \{ \dstiffr_{j0} \}_{j=1}^{N} \cup \{ \dstiffr_{jj} \}_{j=1}^{N} \cup \{ \dmassr_{0j} \}_{j=0}^{N}.
   \]
\item A potential difficulty could arise from the fact that $\dstiffr$ and $\dmassr$ are complex matrices while the available implementation of the Levenberg-Marquardt algorithm might be limited to 
real parameters. 
However, this is easily resolve by splitting into real and imaginary parts 
		\[ \dstiffr_{nm} = \Re \dstiffr_{nm} + i \Im \dstiffr_{nm} \]   
and treating $\Re \dstiffr_{nm}$ and $\Im \dstiffr_{nm}$ as two separate real parameters.
\end{itemize}
The gradients of the cost functions can easily be calculated analytically: 
\begin{equation*}
	\frac{\partial f_{\ell}}{\partial \Re \dstiffr_{00}} = -1, \quad 	\frac{\partial f_{\ell}}{\partial \Re \dmassr_{00}} = - \lambda_{\ell}. 
\end{equation*}
For $n \geq 1$: 
\begin{equation*}
\frac{\partial f_{\ell}}{\partial \Re \dstiffr_{0n}} = \frac{ \dstiffr_{n0} + \lambda_{\ell}}{ \dstiffr_{nn} + \lambda_{\ell}}, \; 
	\frac{\partial f_{\ell}}{\partial \Re \dmassr_{0n}} = \lambda_{\ell} \frac{ \left( \dstiffr_{n0} + \lambda_{\ell} \right)}{ \dstiffr_{nn} + \lambda_{\ell}}, \; 
	\frac{\partial f_{\ell}}{\partial \Re \dstiffr_{n0}} = \frac{ \dstiffr_{0n} + \lambda_{\ell} \dmassr_{0n} }{ \dstiffr_{nn} + \lambda_{\ell}}, \; 
\end{equation*}
and 
\begin{equation*}
	\frac{\partial f_{\ell}}{\partial \Re \dstiffr_{nn}} = - \frac{  \left( \dstiffr_{0n} + \lambda_{\ell} \dmassr_{0n} \right)   \left(\dstiffr_{n0} + \lambda_{\ell} \right) }{ \left( \dstiffr_{nn} + \lambda_{\ell}\right)^2 }. 
\end{equation*}
The derivatives with respect to the imaginary part can be obtained from 
\begin{equation*}
\frac{\partial f_{\ell}}{\partial \Im \dstiffr_{nm}}  = i \frac{\partial f_{\ell}}{\partial \Re \dstiffr_{nm}}.
\end{equation*}
Please note that the cost functions are complex-valued, so depending on the implementation they may have to be split into real and imaginary parts as well.

\subsection{Successive learning}\label{ssection:successive_learning}

An algorithm for the whole pipeline starting with the computation of $\dtn(\lambda_{\ell})$ is given in \cref{alg:learning_matrices_loop}.

\begin{algorithm}
\caption{Successive learning of matrices $\dstiffr$ and $\dmassr$ }
\label{alg:learning_matrices_loop}
\begin{algorithmic}[1]
\STATE{Define $N_{\max}$, $L_{\max}$}
\FOR{$\ell=0:L_{\max}$ }
\STATE{Compute $\dtn(\lambda_{\ell})$ by solving \cref{eq:ODE_dtn_numbers}-\cref{eq:bc_dtn_numbers}.}
\ENDFOR
\STATE{Initialize $\tilde{\dstiffr} \in \mathbb{C},\tilde{\dmassr}\in \mathbb{C}$ randomly.}
\FOR{$N=0:N_{\max}$ }
	\STATE{ Obtain  $\dstiffr,\dmassr \in \mathbb{C}^{(N+1) \times (N+1)}$ by solving \cref{eq:misfit_LM} as described in \cref{ssection:impl_fixed_N} using  $\tilde{\dstiffr}, \tilde{\dmassr}$ as initial guess. }
	\STATE{Prepare new initial guess: $\tilde{\dstiffr},\tilde{\dmassr} \in \mathbb{C}^{(N+2) \times (N+2)}$.   }
	\STATE{ Initialize upper block of  $ \tilde{\dstiffr},\tilde{\dmassr} $ by  $\dstiffr,\dmassr$.   }
	\STATE{ Set $ \tilde{\dstiffr}_{N+1,N+1} $ to one or a good guess for the poles if available. } 
	\STATE{ Fill $\tilde{\dstiffr}_{N+1,0},\tilde{\dstiffr}_{0,N+1}$ and $\tilde{\dmassr}_{0,N+1}$ with small (nonzero) random numbers.  } 
\ENDFOR
\RETURN Learned matrices $\dstiffr,\dmassr $ for $N=0:N_{\max}$.
\end{algorithmic}
\end{algorithm}

A few additional remarks are given below. 
\begin{itemize}
\item Often there is an analytic reference solution for $\dtn(\lambda_{\ell})$ available so that the solution of the ODEs can be skipped.
\item The purpose of the loop over $N$ in \cref{alg:learning_matrices_loop} is to provide a good initial guess for the optimization. 
Basically, the learned matrices from step $N$ are reused as an initial guess for step $N+1$. 
With this technique the learned $\ddtn(\lambda_{\ell})$-function in step $N+1$ starts at the value of the minimizer of step $N$ plus the additional term 
		\[  \frac{( \tilde{\dstiffr}_{0,N+1} + \lambda_{\ell} \tilde{\dmassr}_{0,N+1}) ( \tilde{\dstiffr}_{N+1,0} + \lambda_{\ell}  )  }{  \tilde{\dstiffr}_{N+1,N+1} + \lambda_{\ell}  }.   \]
In order to prevent the cost function from starting too far away from the minimizer of the previous step this contribution should be small. 
This is the motivation for filling $\tilde{\dstiffr}_{N+1,0},\tilde{\dstiffr}_{0,N+1}$ and $\tilde{\dmassr}_{0,N+1}$ with small random numbers. 
These should be nonzero to avoid getting stuck at the minimizer of step $N$. 
If a good guess for the poles is available then setting $ -\tilde{\dstiffr}_{N+1,N+1} $ equal to one of these poles can be beneficial. 
\end{itemize}

\section{Additional plots for numerical experiments}

\cref{fig:dtn_omega_inf_imag}, \cref{fig:dtn_kinf12}, \cref{fig:dtn_kinf20}, \cref{fig:waveguide_im} and \cref{tab:opt_performance} show some additional aspects of the numerical experiments described in the main article.

\begin{table}
\begin{center}
\begin{tabular}{ccccc}
  \toprule
% \hline\noalign{\smallskip}
$N$ & iter & cost & gradient  & time[s]       \\
  \midrule
0  &     \colora{ 3} / \colorb{ 3  } &  \colora{$8.3 \cdot 10^{5}$ } /  \colorb{$8.3 \cdot 10^{5}$  } & \colora{$7.9 \cdot 10^{4}$ } /  \colorb{$7.9 \cdot 10^{4}$     }& \colora{$1.8 \cdot 10^{-4}$ } /  \colorb{$2.7 \cdot 10^{-4}$ }    \\
1   &  \colora{34 } / \colorb{16  } &  \colora{$1.3 \cdot 10^{2}$} /  \colorb{$1.3 \cdot 10^{2}$ } & \colora{$2.9 \cdot 10^{2}$ } /  \colorb{$1.3 \cdot 10^{2}$  }& \colora{ $8.3 \cdot 10^{-3}$} /  \colorb{$6.3 \cdot 10^{-4}$  }    \\
2    &    \colora{10} / \colorb{ 16   } &  \colora{  $6.1 \cdot 10^{-2}$} /  \colorb{ $6.1 \cdot 10^{-2}$    } & \colora{$7.5  \cdot 10^{-1 }$} /  \colorb{  $3.8  \cdot 10^{2 }$ }& \colora{ $8.3  \cdot 10^{-3 }$} /  \colorb{ $1.0  \cdot 10^{-3 }$ }    \\
3  &     \colora{ 9 } / \colorb{ 496  } &  \colora{$2.9 \cdot 10^{-5}$} /  \colorb{$2.9 \cdot 10^{-5}$ } & \colora{$1.1  \cdot 10^{-1 }$} /  \colorb{$5.8  \cdot 10^{0 }$  }& \colora{$1.7  \cdot 10^{ -2}$} /  \colorb{ $4.7  \cdot 10^{ -2}$ }    \\
4  &     \colora{12} / \colorb{ 4620  } &  \colora{$1.4 \cdot 10^{-8}$ } /  \colorb{ $1.4 \cdot 10^{-8}$ } & \colora{$4.5  \cdot 10^{-3 }$} /  \colorb{ $3.5  \cdot 10^{-1 }$  }& \colora{ $4.9  \cdot 10^{ -2}$} /  \colorb{$5.9  \cdot 10^{ -1}$   }    \\
5  &     \colora{20  } / \colorb{ 177  } &  \colora{ $7.2 \cdot 10^{-12}$} /  \colorb{$7.2 \cdot 10^{-12}$  } & \colora{$4.6  \cdot 10^{-3 }$ } /  \colorb{ $3.8  \cdot 10^{-4 }$ }& \colora{  $1.6  \cdot 10^{-1 }$} /  \colorb{  $3.0  \cdot 10^{-2}$  }    \\
6  &     \colora{ 1491 } / \colorb{ 293 } &  \colora{ $3.7 \cdot 10^{-15}$  } /  \colorb{ $3.7 \cdot 10^{-15}$ } & \colora{$2.9  \cdot 10^{-3 }$ } /  \colorb{$3.8  \cdot 10^{-3}$   }& \colora{  $2.9  \cdot 10^{1 }$} /  \colorb{ $8.1  \cdot 10^{-2 }$ }    \\
\bottomrule
  % \noalign{\smallskip}\hline
\end{tabular}
\end{center}
\caption{Performance of the Levenberg-Marquardt algorithm for solving the minimization for the Helmholtz equation with $a=1$ and $\wavenr = 16$ from \cref{ssection:diag_and_poles}. Results for the dense / reduced ansatz are given in \colora{red} / \colorb{blue}. The cost denotes the value of the objective function at the minimizer, similarly for the gradient.  }
\label{tab:opt_performance}
\end{table}

\begin{figure}[htbp]
\centering
%\subfloat[ $\Im(\text{dtn}(\lambda))$ ]{ \label{fig:zeta_omega_inf_imag} 
%\includegraphics[scale=0.75]{Figures/dtn-jump-infAll-imag.pdf}
%}
\subfloat[ $\wavenr_{\infty} = 16 $  ]{ \label{fig:k_inf_16_imag}
\includegraphics[scale=0.75]{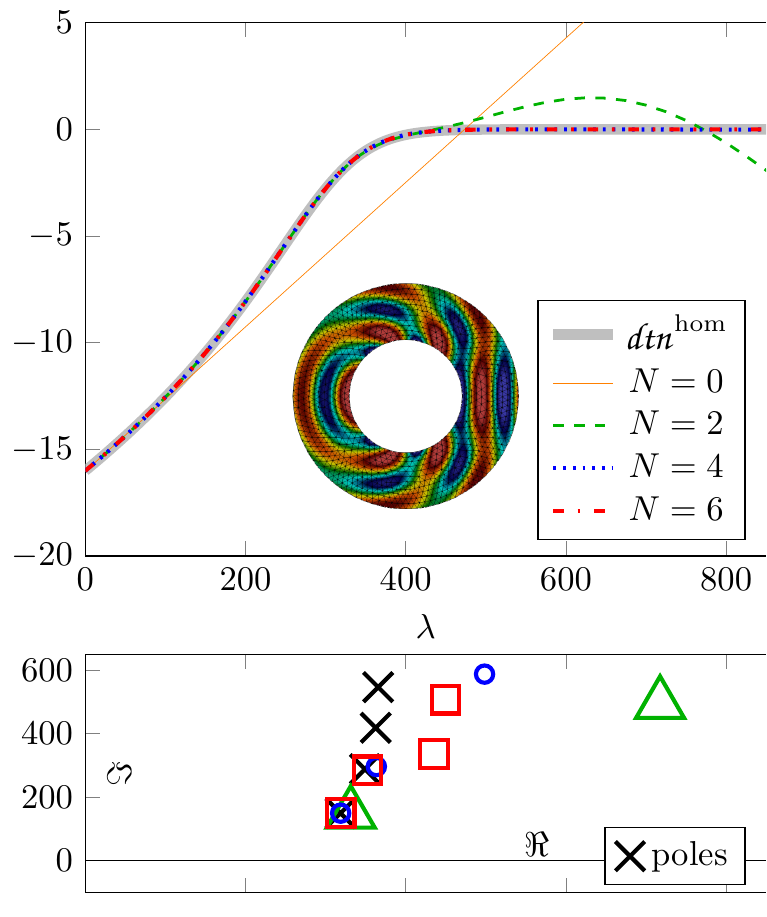}
 } 
%\subfloat[  $\wavenr_{\infty} = 12$ ]{ \label{fig:omega_inf_12_imag} 
%\includegraphics[scale=0.75]{Figures/dtn-jump-inf12-imag.pdf}
%}
\subfloat[ $\wavenr_{\infty} = 8$  ]{ \label{fig:k_inf_8_imag}
\includegraphics[scale=0.75]{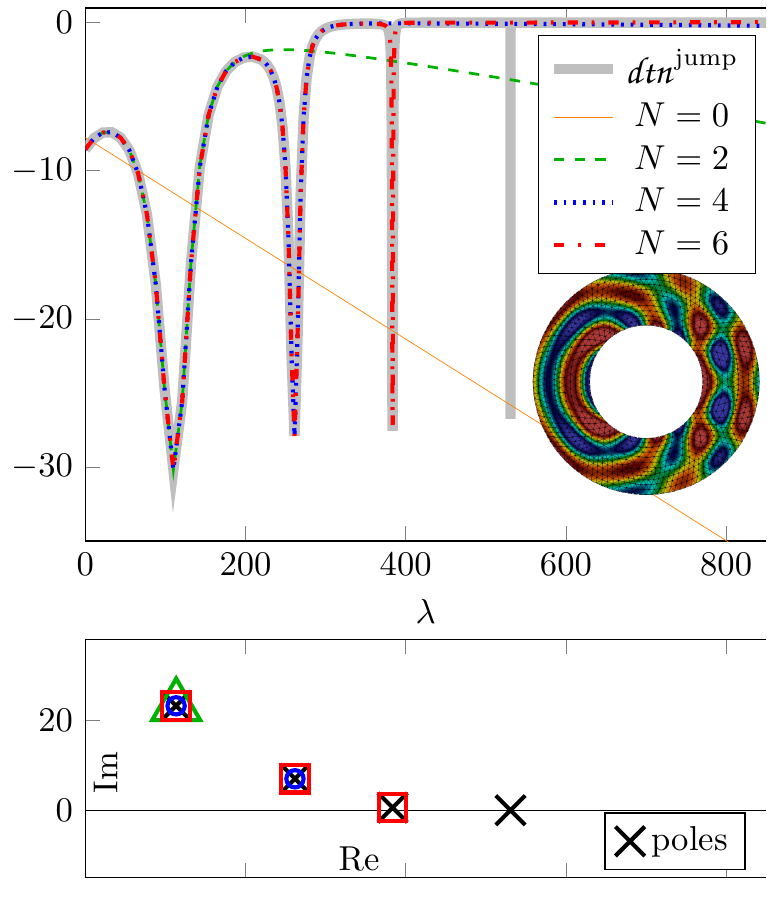}
 }
\caption{  Approximation results for the imaginary part of $\dtn$ belonging to a discontinuous exterior wavespeed which jumps from $\wavenr_{I} = 16$ to $\wavenr_{\infty}$. 
	  In the lower panel we also compare the poles of $\dtn$ (black crosses)  with the poles of $\ddtn$ (color according to legend). Only the poles of $\ddtn$ close to
	 the real axis are displayed.}
\label{fig:dtn_omega_inf_imag}
\end{figure}

 \begin{figure}[htbp]
\centering
\subfloat[Real part. ]{ \label{fig:dtn_kinf_12_real} 
\includegraphics[scale=0.75]{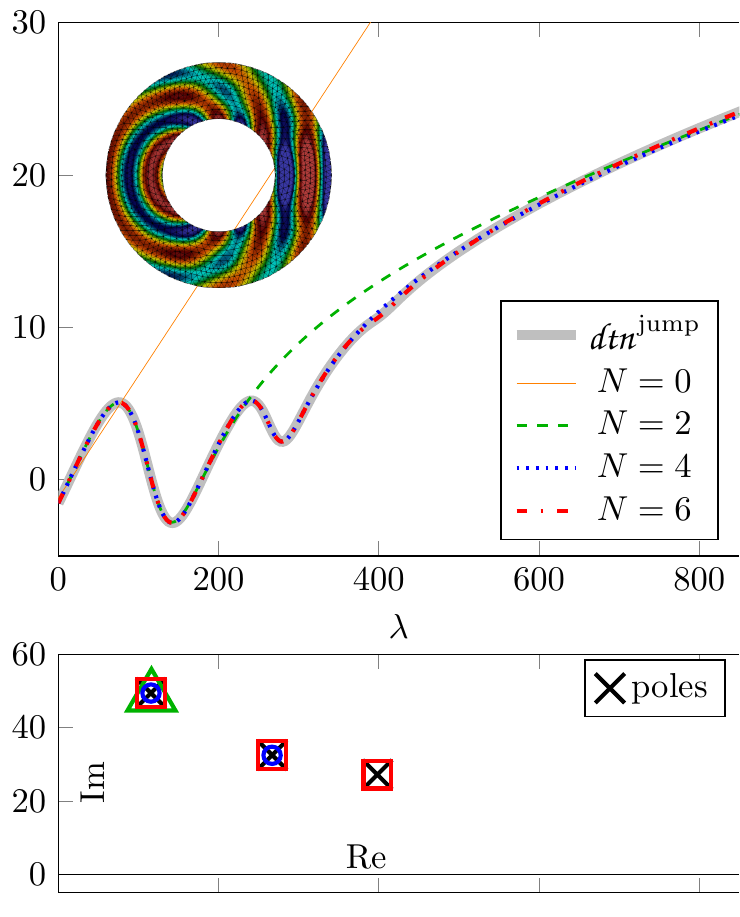}
}
\subfloat[Imaginary part. ]{ \label{fig:kinf_12_imag}
\includegraphics[scale=0.75]{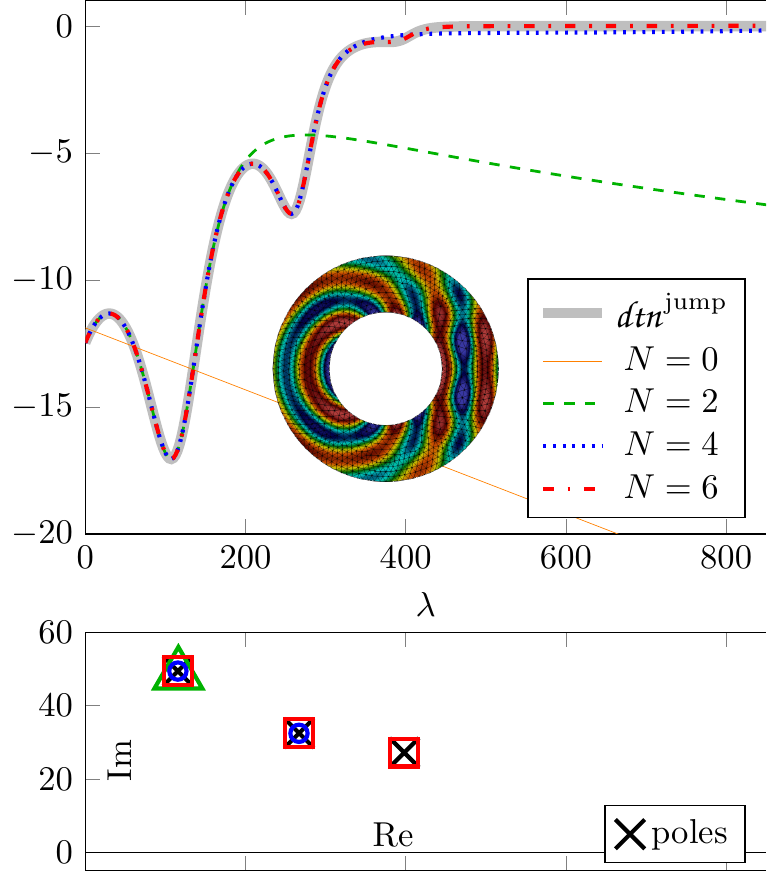}
 }
\caption{  Approximation results for the imaginary part of $\dtn$ belonging to a discontinuous exterior wavespeed which jumps from $\wavenr_{I} = 16$ to $\wavenr_{\infty}=12$.  }
\label{fig:dtn_kinf12}
\end{figure}

 \begin{figure}[htbp]
\centering
\subfloat[Real part. ]{ \label{fig:dtn_kinf_20_real} 
\includegraphics[scale=0.75]{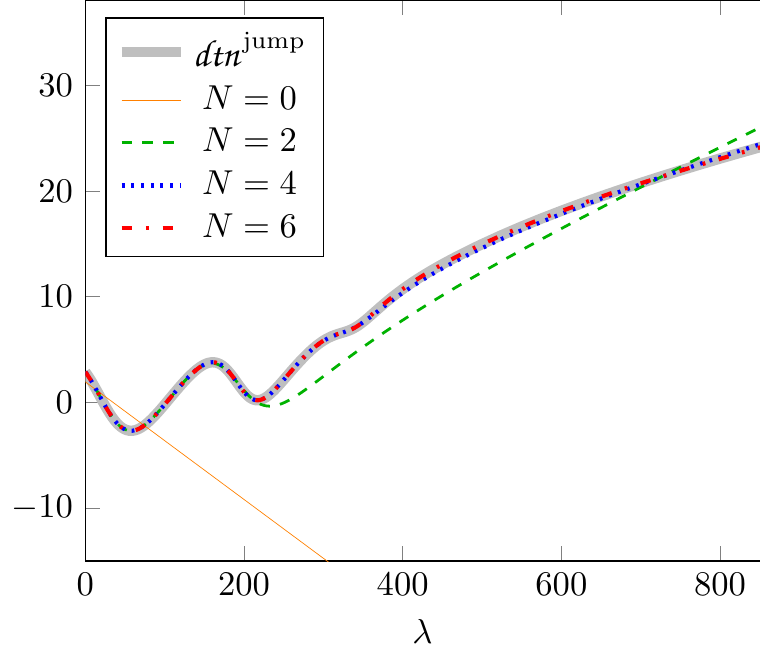}
}
\subfloat[Imaginary part. ]{ \label{fig:dtn_kinf_20_imag}
\includegraphics[scale=0.75]{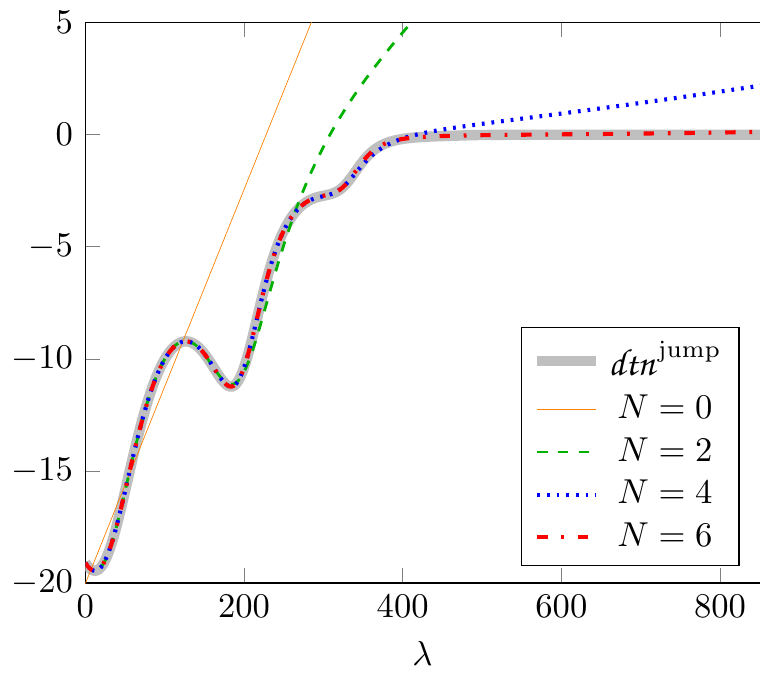}
 }
\caption{  Approximation results for the imaginary part of $\dtn$ belonging to a discontinuous exterior wavespeed which jumps from $\wavenr_{I} = 16$ to $\wavenr_{\infty}=20$.  }
\label{fig:dtn_kinf20}
\end{figure}

 \begin{figure}[htbp]
\centering
\includegraphics[scale=1.0]{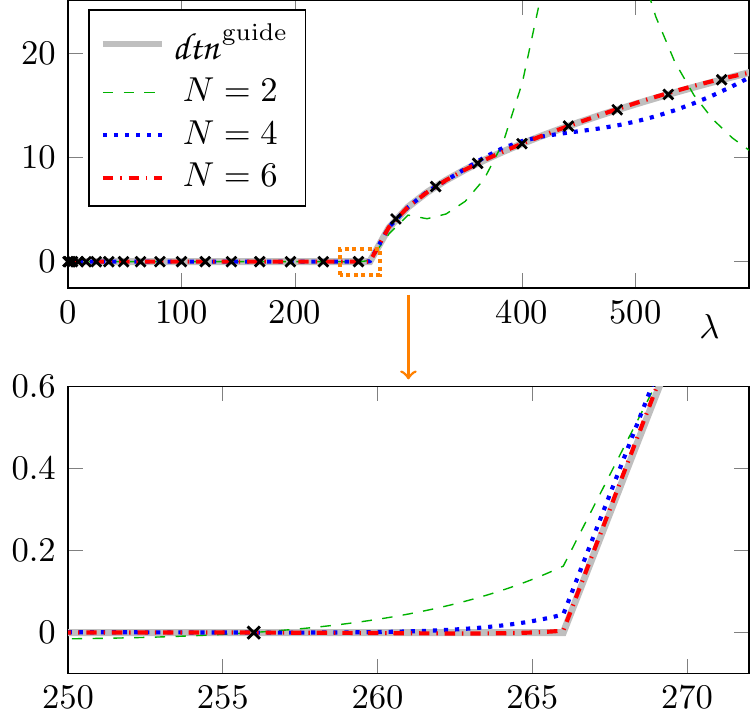}
\caption{ Approximation results for real part of $\dtn^{\mathrm{guide}} $ with $\wavenr = 16.5$. The sample points $\dtn^{\mathrm{guide}}(\lambda_{\ell}) $  are displayed as black crosses. }
\label{fig:waveguide_im}
\end{figure}

%\section{An example appendix} 

%\begin{lemma}
%Test Lemma.
%\end{lemma}

\section*{Acknowledgments}
%The authors would like to thank Lothar Nannen for advice on fine tuning the HSIE method. \par
This research was funded by the Deutsche Forschungsgemeinschaft (DFG, German Research Foundation) -- Project-ID 432680300 -- SFB 1456.

J.P.\  is  a  member  of  the  International  Max  Planck Research School for Solar System Science at the University of G\"ottingen.\\[1ex]
\textbf{Author contributions:} T.H.\ initiated the research on learned infinite elements. All authors designed and performed the research.  All implementations and numerical computations were performed by J.P.\ using the finite element library \texttt{Netgen/NGSolve}
and the Levenberg-Marquardt algorithm from the \texttt{Ceres-solver}. J.P.\ also conceived \cref{ssection:convergence} and drafted the paper. 
All authors contributed to the final manuscript.

\bibliographystyle{siamplain}
\bibliography{references}
\end{document}